\documentclass{amsart}

\usepackage[T1]{fontenc}
\usepackage[latin1]{inputenc}
\usepackage{url}
\usepackage{dsfont}
\usepackage{graphicx}
\usepackage{enumitem}
\usepackage[colorlinks=true, linktoc=page]{hyperref}

\numberwithin{equation}{section}

\newcommand{\ie}{{\em i.e.}~}

\newcommand{\dd}{\mathrm{d}}
\newcommand{\R}{\mathbb{R}}
\newcommand{\Exp}{\mathbb{E}}

\newcommand{\ind}[1]{\mathds{1}_{\{#1\}}}
\newcommand{\sgn}{\mathrm{sgn}}

\newcommand{\grandO}{\mathcal{O}}

\newcommand{\Mu}{\mathrm{M}}

\newcommand{\revision}[1]{#1}

\newcommand{\bu}{\mathbf{u}}
\newcommand{\bv}{\mathbf{v}}

\newcommand{\bm}{\mathbf{m}}

\newcommand{\rx}{\mathrm{x}}
\newcommand{\ry}{\mathrm{y}}
\newcommand{\rblambda}{\bar{\lambda}}
\newcommand{\rbmu}{\bar{\mu}}

\newcommand{\Ls}{\mathrm{L}} 
\newcommand{\Cs}{\mathrm{C}} 
\newcommand{\Ps}{\mathrm{P}} 
\newcommand{\Ws}{\mathrm{W}} 

\newcommand{\Unif}{\mathrm{U}} 

\newcommand{\ConstUSH}{L_{\mathrm{USH}}} 
\newcommand{\ConstBoundS}{L_{\mathrm{C}}} 
\newcommand{\ConstLip}{L_{\mathrm{LC}}} 
\newcommand{\ConstStab}{\mathcal{L}} 

\newcommand{\x}{\mathbf{x}}
\newcommand{\y}{\mathbf{y}}
\newcommand{\z}{\mathbf{z}}

\newcommand{\Part}{{P_n^d}} 

\newcommand{\Rb}{\mathrm{R}} 
\newcommand{\Nb}{\mathrm{N}} 

\newcommand{\tPhi}{\tilde{\Phi}} 

\newcommand{\tinter}{\tau^{\mathrm{coll}}} 

\newcommand{\bblambda}{\bar{\boldsymbol{\lambda}}} 

\newcommand{\tlambda}{\tilde{\lambda}} 
\newcommand{\tblambda}{\tilde{\boldsymbol{\lambda}}} 

\DeclareMathOperator{\argsinh}{argsinh}

\renewcommand{\bar}[1]{\overline{#1}}

\newtheorem{defi}{Definition}[section]
\newtheorem{lem}[defi]{Lemma}
\newtheorem{prop}[defi]{Proposition}
\newtheorem{theo}[defi]{Theorem}
\newtheorem{cor}[defi]{Corollary}

\newtheoremstyle{myremark}{}{}{}{0pt}{\bfseries}{.}{ }{}
\theoremstyle{myremark}
\newtheorem{rk}[defi]{Remark}

\title[Optimal convergence rate of the multitype sticky particle dynamics]{Optimal convergence rate of the multitype sticky particle approximation of one-dimensional diagonal hyperbolic systems with monotonic initial data}
\author{Benjamin Jourdain}
\address{{\bf Benjamin Jourdain}\newline
{\rm \indent Universit\'e Paris-Est, CERMICS (ENPC), INRIA, F-77455 Marne-la-Vall\'ee.}}
\email{\href{mailto:jourdain@cermics.enpc.fr}{jourdain@cermics.enpc.fr}}
\author{Julien Reygner}
\address{{\bf Julien Reygner}\newline
{\rm \indent Universit\'e Paris-Est, CERMICS (ENPC), F-77455 Marne-la-Vall\'ee.}}
\email{\href{mailto:julien.reygner@enpc.fr}{julien.reygner@enpc.fr}}
\thanks{This research benefited from the support of the French National Research Agency (ANR) under the program ANR-12-BLAN Stab.}

\makeatletter

\def\part{\@startsection{part}{1}%
\z@{.7\linespacing\@plus\linespacing}{.5\linespacing}%
{\large\normalfont\bfseries}}

\def\section{\@startsection{section}{1}%
\z@{.7\linespacing\@plus\linespacing}{.5\linespacing}%
{\normalfont\bfseries\centering}}

\def\@settitle{\begin{center}%
  \baselineskip14\p@\relax
    \bfseries
    \LARGE\@title
  \end{center}%
}

\def\@setauthors{%
  \begingroup
  \trivlist
  \centering\footnotesize \@topsep30\p@\relax
  \advance\@topsep by -\baselineskip
  \item\relax
  \andify\authors
  \def\\{\protect\linebreak}%
 {\Large\authors}%
  \endtrivlist
  \endgroup
}

\def\maketitle{\par
  \@topnum\z@ 
  \@setcopyright
  \thispagestyle{firstpage}
  \ifx\@empty\shortauthors \let\shortauthors\shorttitle
  \else \andify\shortauthors
  \fi
  \@maketitle@hook
  \begingroup
  \@maketitle
  \toks@\@xp{\shortauthors}\@temptokena\@xp{\shorttitle}%
  \toks4{\def\\{ \ignorespaces}}
  \edef\@tempa{%
    \@nx\markboth{\the\toks4
      \@nx{\the\toks@}}{\the\@temptokena}}%
  \@tempa
  \endgroup
  \c@footnote\z@
  \def\do##1{\let##1\relax}%
  \do\maketitle \do\@maketitle \do\title \do\@xtitle \do\@title
  \do\author \do\@xauthor \do\address \do\@xaddress
  \do\email \do\@xemail \do\curraddr \do\@xcurraddr
  \do\commby \do\@commby
  \do\dedicatory \do\@dedicatory \do\thanks \do\thankses
  \do\keywords \do\@keywords \do\subjclass \do\@subjclass
}

\makeatother

\begin{document}

\begin{abstract}
  Brenier and Grenier~[SIAM J. Numer. Anal., 1998] proved that sticky particle dynamics with a large number of particles allow to approximate the entropy solution to scalar one-dimensional conservation laws with monotonic initial data. In~[arXiv:1501.01498], we introduced a multitype version of this dynamics and proved that the associated empirical cumulative distribution functions converge to the viscosity solution, in the sense of Bianchini and Bressan~[Ann. of Math. (2), 2005], of one-dimensional diagonal hyperbolic systems with monotonic initial data of arbitrary finite variation. In the present paper, we analyse the $\Ls^1$ error of this approximation procedure, by splitting it into the discretisation error of the initial data and the non-entropicity error induced by the evolution of the particle system. We prove that the error at time $t$ is bounded from above by a term of order $(1+t)/n$, where $n$ denotes the number of particles, and give an example showing that this rate is optimal. We last analyse the additional error introduced when replacing the multitype sticky particle dynamics by an iterative scheme based on the typewise sticky particle dynamics, and illustrate the convergence of this scheme by numerical simulations.
\end{abstract}
\maketitle


\section{Introduction}

Systems of sticky particles have been known to reproduce the phenomenological behaviour of one-dimensional conservation laws in various physical contexts, in particular in astrophysics or in the study of gas dynamics~\cite{zeldo,vergassola}. In such systems, finitely many particles evolve on the real line at constant velocity and stick together at collisions, with preservation of mass and momentum but dissipation of energy. The relation between these discrete systems and the equations of continuum physics was formalised by Brenier and Grenier~\cite{bregre}, who showed that sticky particle dynamics with a large number of particles allow to approximate the entropy solution to scalar one-dimensional conservation laws with monotonic initial data. We also refer to Bouchut~\cite{bouchut}, Grenier~\cite{grenier}, and E, Rykov and Sinai~\cite{eryksin} for previous results in this direction. Based on this idea, we recently introduced a multitype sticky particle dynamics~\cite{jr} in order to approximate the viscosity solution, in the sense of Bianchini and Bressan~\cite{bianbres}, of one-dimensional diagonal hyperbolic systems with monotonic initial data of arbitrary finite variation.

These sticky particle dynamics provide natural numerical schemes for the corresponding solutions to scalar conservation laws or diagonal hyperbolic systems. It is thus of interest to control the approximation error due to this procedure. This is the purpose of this article. We shall rely on the remark that sticky particle dynamics generically induce exact weak solutions to the considered equation, but for discrete initial data. Besides, these weak solutions need not satisfy Kru\v{z}kov's entropy or Bianchini-Bressan's viscosity condition. This leads us to split the total approximation error into a discretisation error of the initial data, and a non-entropicity error induced by the evolution of the particle system. 

The discretisation error of the initial data is addressed in Section~\ref{s:ci}. In particular, if the initial conditions have a compactly supported distributional derivative, this error in $\Ls^1$ distance for $n$ particles is proved to be bounded from above by a term of order $1/n$. The error due to the evolution of the particle system is studied in Section~\ref{s:scl} for the case of scalar conservation laws and in Section~\ref{s:syst} for the case of diagonal hyperbolic systems. In both cases, this error at time $t \geq 0$ is proved to be bounded from above by a term of order $t/n$. This leads to a global convergence rate of order $(1+t)/n$ in the number $n$ of particles. The precise statements for scalar conservation laws and diagonal hyperbolic systems are respectively given in Theorems~\ref{theo:rateSPD} and~\ref{theo:rateMSPD}, which are the main results of this paper. These results are finally illustrated with numerical simulations in Section~\ref{s:num}. \revision{We emphasise the fact that the sticky particle approach is essentially restricted to the one-dimensional case and mention that (non-)existence and (non-)uniqueness issues related to its multidimensional generalisation were pointed out by Bressan and Nguyen~\cite{BreNgu14}.}

The remainder of this introduction is dedicated to a detailed presentation of the sticky particle dynamics (SPD) and the multitype sticky particle dynamics (MSPD). 


\subsection{SPD and scalar conservation laws} This subsection is dedicated to the introduction of the SPD, which allows to approximate the entropy solution to scalar conservation laws in one space dimension. 

\subsubsection{Scalar conservation laws} Let us consider the scalar conservation law
\begin{equation}\label{eq:scl}
  \left\{\begin{aligned}
    & \partial_t u + \partial_x\left(\Lambda(u)\right) = 0, \qquad t \geq 0, \quad x \in \R,\\
    & u(0,x) = u_0(x),
  \end{aligned}\right.
\end{equation}
for a nonconstant, monotonic and bounded initial condition $u_0$. Up to an affine transform of the flux function $\Lambda$, one can assume that $u_0$ is the cumulative distribution function (CDF) of a probability measure $m$ on the real line, which we denote $u_0 = H*m$ where $H(x) = \ind{x \geq 0}$ is the Heaviside function. The space of probability measures on the real line is denoted by $\Ps(\R)$. Then $\Lambda$ only needs to be defined on the interval $[0,1]$, and it shall be assumed to have the following regularity.
\begin{enumerate}[ref=C, label=(C)]
  \item\label{ass:C} The function $\Lambda$ is of class $\Cs^1$ on $[0,1]$.
\end{enumerate}
Under Assumption~\eqref{ass:C}, we denote $\lambda = \Lambda'$ and $\ConstBoundS = \sup_{u \in [0,1]} |\lambda(u)|$.

The following existence and uniqueness result follows from Kru\v{z}kov's theorem, see~\cite[Theorem~2.3.5 and Proposition~2.3.6, pp.~36-37]{serre}.
\begin{theo}\label{theo:kruz}
  Let $\Lambda : [0,1] \to \R$ satisfying Assumption~\eqref{ass:C} and $u_0 = H*m$ for $m \in \Ps(\R)$. There exists a unique weak solution $u : [0,+\infty) \times \R \to [0,1]$ to the scalar conservation law~\eqref{eq:scl} satisfying the entropy condition that, for all $c \in [0,1]$,
  \begin{equation*}
    \partial_t |u-c| + \partial_x \left(\sgn(u-c)(\Lambda(u)-\Lambda(c))\right) \leq 0
  \end{equation*}
  in the distributional sense, where $\sgn(v) := \ind{x \geq 0} - \ind{x<0}$.

  In addition, it satisfies the following properties:
  \begin{enumerate}[label=(\roman*), ref=\roman*]
    \item\label{it:kruz:1} preservation of total variation: for all $t \geq 0$, $u(t,\cdot)$ coincides $\dd x$-almost everywhere with the CDF of a probability measure $m_t$ on the real line;
    \item\label{it:kruz:2} finite speed of propagation: if $u_0(a)=0$, then $u(t,a-t\ConstBoundS)=0$ for all $t \geq 0$, and if $u_0(b)=1$, then $u(t,b+t\ConstBoundS)=1$ for all $t \geq 0$;
    \item\label{it:kruz:3} stability: if $u$ and $v$ refer to the entropy solutions to the scalar conservation law with respective initial data $u_0$ and $v_0$, then for all $t \geq 0$,
    \begin{equation*}
      \|u(t,\cdot)-v(t,\cdot)\|_{\Ls^1(\R)} \leq \|u_0-v_0\|_{\Ls^1(\R)}.
    \end{equation*}
  \end{enumerate}
\end{theo}

\subsubsection{Sticky Particle Dynamics} For $n \geq 1$, we denote by $D_n$ the polyhedron of $\R^n$ defined by
\begin{equation*}
  D_n := \{\rx = (x_1, \ldots, x_n) : x_1 \leq \cdots \leq x_n\}.
\end{equation*}
Let $\rx \in D_n$ be a vector of initial positions and $\rblambda = (\rblambda_1, \ldots, \rblambda_n) \in \R^n$ a vector of initial velocities. Under the SPD, the particle with index $k \in \{1, \ldots, n\}$ has initial position $x_k$, initial velocity $\rblambda_k$ and mass $1/n$. It evolves at constant velocity on the real line, up to the first collision with another particle. At collisions, the particles stick together and form a cluster: its mass is given by the number of colliding particles over $n$, and its velocity by the average of the pre-collisional velocities of the particles. More generally, when several clusters collide, they form a single cluster with conservation of total mass and momentum.

For all $t \geq 0$, the position of the $k$-th particle at time $t \geq 0$ is denoted by $\phi_k[\rblambda](\rx;t)$, and it is easy to check that the process $(\phi[\rblambda](\rx;t))_{t \geq 0}$ defined by $\phi[\rblambda](\rx;t) := (\phi_1[\rblambda](\rx;t), \ldots, \phi_n[\rblambda](\rx;t))$ induces a continuous flow in $D_n$. Its stability with respect to the initial configuration and the vector of initial velocities is detailed in Proposition~\ref{prop:stabSPD} below. Before stating this result, let us define the normalised $\Ls^1$ distance on $D_n$ by
\begin{equation*}
  \|\rx-\ry\|_1 := \frac{1}{n} \sum_{k=1}^n |x_k-y_k|.
\end{equation*}

\begin{prop}\cite[Proposition~3.1.9, (i)]{jr}\label{prop:stabSPD}
  Let $\rx, \ry \in D_n$ and $\rblambda, \rbmu \in \R^n$. For all $0 \leq s \leq t$,
  \begin{equation}\label{eq:stabSPD}
    \|\phi[\rblambda](\rx;t) - \phi[\rbmu](\ry;t)\|_{1}\leq \|\phi[\rblambda](\rx;s) - \phi[\rbmu](\ry;s)\|_{1} + \frac{t-s}{n}\sum_{k=1}^n |\rblambda_k-\rbmu_k|.
  \end{equation}
\end{prop}

\subsubsection{Approximation of the scalar conservation law} Let $\Lambda$ satisfy Assumption~\eqref{ass:C}. In order to approximate the entropy solution to the scalar conservation law~\eqref{eq:scl}, we specify a choice of initial velocities for the SPD by defining $\rblambda \in \R^n$ as
\begin{equation*}
  \forall k \in \{1, \ldots, n\}, \qquad \rblambda_k := n \int_{w=(k-1)/n}^{k/n} \lambda(w)\dd w.
\end{equation*}
Given an initial configuration $\rx \in D_n$, we define the empirical distribution of the SPD at time $t \geq 0$ by
\begin{equation*}
  \mu_t[\rx] := \frac{1}{n} \sum_{k=1}^n \delta_{\phi_k[\rblambda](\rx;t)} \in \Ps(\R),
\end{equation*}
and the associated empirical CDF by
\begin{equation*}
  u_n[\rx](t,x) := H*\mu_t[\rx](t,x) = \frac{1}{n} \sum_{k=1}^n \ind{\phi_k[\rblambda](\rx;t) \leq x}.
\end{equation*}

Given $n \geq 1$ and $\rx \in D_n$, it is easily checked that the empirical CDF $u_n[\rx]$ satisfies the properties~\eqref{it:kruz:1}, \eqref{it:kruz:2} and~\eqref{it:kruz:3} of Theorem~\ref{theo:kruz}. The preservation of the total variation is obvious, and the finite speed of propagation is just the transcription of the fact that the modulus of the initial velocities $\rblambda_k$ is bounded by $\ConstBoundS$, uniformly with respect to $n$. Finally, the $\Ls^1$ stability follows from~\eqref{eq:stabSPD} with $\rbmu=\rblambda$. 

It can also be shown that $u_n[\rx]$ is a weak solution to the scalar conservation law~\eqref{eq:scl} with discrete initial condition~\cite[Proposition~4.2.1]{jr}. However for fixed $n$, it does not necessarily satisfy the entropy condition of Theorem~\ref{theo:kruz}. This property is recovered when taking the limit of an infinite number of particles, as is expressed by the following result, the proof of which is originally due to Brenier and Grenier~\cite{bregre}, see also~\cite{jourdain:sticky} and~\cite[Lemma~8.2.3]{jr} for appropriate generalisations.
\begin{theo}\label{theo:cvSPD}
  Let $\Lambda$ satisfy Assumption~\eqref{ass:C} and let $m \in \Ps(\R)$. Let $(\rx(n))_{n \geq 1}$ be a sequence of initial configurations such that, for all $n \geq 1$, $\rx(n) \in D_n$ and the empirical distribution
  \begin{equation*}
    \mu_0[\rx(n)] = \frac{1}{n} \sum_{k=1}^n \delta_{x_k(n)}
  \end{equation*}
  converges weakly to $m$. 
  
  For all $t \geq 0$, the empirical distribution $\mu_t[\rx(n)]$ converges weakly to the probability measure $m_t \in \Ps(\R)$ such that $u(t,x) := H*m_t(x)$ is the unique entropy solution of the scalar conservation law~\eqref{eq:scl} with initial condition $u_0 = H*m$. Equivalently, for all $t \geq 0$, the empirical CDF $u_n[\rx(n)](t,\cdot)$ converges $\dd x$-almost everywhere to $u(t,\cdot)$.
\end{theo}

Using Theorem~\ref{theo:cvSPD} to pass to the limit $n \to +\infty$ in~\eqref{eq:stabSPD}, the stability inequality of~\eqref{it:kruz:3} can be extended in order to take the dependence of the entropy solution on the flux function into account. This is done in the next proposition, which is of independent interest, and the proof of which is postponed to Appendix~\ref{app:pf}.

\begin{prop}\label{prop:genstaconscal}
  Let $\Lambda, \Mu : [0,1] \to \R$ satisfying Assumption~\eqref{ass:C}, and $u_0, v_0$ be CDFs on the real line. Denote $\lambda := \Lambda'$ and $\mu := \Mu'$, and call $u$ and $v$ the entropy solutions of the scalar conservation law with respective flux function $\Lambda$ and $\Mu$, and respective initial data $u_0$ and $v_0$. Then, for all $0 \leq s \leq t$,
  \begin{equation*}
    \|u(t,\cdot)-v(t,\cdot)\|_{\Ls^1(\R)} \leq \|u(s,\cdot)-v(s,\cdot)\|_{\Ls^1(\R)} + (t-s) \int_{w=0}^1 |\lambda(w)-\mu(w)|\dd w.
  \end{equation*}
\end{prop}

\subsubsection{Rate of convergence}\label{sss:rateSPD} Given a sequence of initial configurations $(\rx(n))_{n \geq 1}$ satisfying the assumptions of Theorem~\ref{theo:cvSPD}, our first purpose in this article is to estimate the error when approximating $u$ with $u_n[\rx(n)]$. On account of the stability property stated in Theorem~\ref{theo:kruz}, a fairly natural distance to measure this error is the $\Ls^1$ distance on $\R$. Indeed, this stability property allows us to write, for all $t \geq 0$,
\begin{equation}\label{eq:decompscl}
  \begin{aligned}
    & \|u(t,\cdot) - u_n[\rx(n)](t,\cdot)\|_{\Ls^1(\R)}\\
    & \qquad \leq \|u(t,\cdot) - u_{\infty}[\rx(n)](t,\cdot)\|_{\Ls^1(\R)} + \|u_{\infty}[\rx(n)](t,\cdot)-u_n[\rx(n)](t,\cdot)\|_{\Ls^1(\R)}\\
    & \qquad \leq \|u_0 - u_{\infty,0}[\rx(n)]\|_{\Ls^1(\R)} + \|u_{\infty}[\rx(n)](t,\cdot)-u_n[\rx(n)](t,\cdot)\|_{\Ls^1(\R)},
  \end{aligned}
\end{equation}
where we have introduced the entropy solution $u_{\infty}[\rx(n)]$ to the scalar conservation law~\eqref{eq:scl} with discretised initial condition $u_{\infty,0}[\rx(n)] := H*\mu_0[\rx(n)]$. The two terms in the right-hand side above are of a very different nature, and can be estimated separately: the first term corresponds to the discretisation error of the measure $m$, while the second term only measures the non-entropicity induced by the evolution of the particle system for a given initial condition $u_{\infty,0}[\rx(n)]$.

The discretisation error of the measure $m$ is addressed in Section~\ref{s:ci}. There, we use the fact that, given two probability measures $m$ and $m'$ on the real line, the $\Ls^1$ distance between $H*m$ and $H*m'$ is the Wasserstein distance of order $1$ between $m$ and $m'$, defined by
\begin{equation}\label{eq:defiwass}
  \Ws_1(m,m') := \inf_{\mathfrak{m} <^m_{m'}} \int_{(x,x') \in \R^2} |x-x'| \mathfrak{m}(\dd x\dd x'),
\end{equation}
where the infimum runs over all the probability measures $\mathfrak{m} \in \Ps(\R^2)$ such that, for all Borel sets $A, A' \subset \R$,
\begin{equation*}
  \mathfrak{m}(A \times \R) = m(A), \qquad \mathfrak{m}(\R \times A') = m'(A').
\end{equation*}
This is due to the fact that, on the real line, the measure 
\begin{equation*}
  \mathfrak{m} = \Unif \circ \left((H*m)^{-1}, (H*m')^{-1}\right)^{-1},
\end{equation*}
where $\Unif$ refers to the Lebesgue measure on $[0,1]$, realises the infimum in~\eqref{eq:defiwass}. In this definition, the pseudo-inverse $F^{-1}$ of a CDF is defined by
\begin{equation}\label{eq:pseudoinv}
  \forall v \in (0,1), \qquad F^{-1}(v) := \inf\{x \in \R : F(x) \geq v\}.
\end{equation}
We deduce that
\begin{equation}\label{eq:W1L1}
  \Ws_1(m,m') = \int_{v=0}^1 |(H*m)^{-1}(v)-(H*m')^{-1}(v)|\dd v = \int_{x \in \R} |H*m(x)-H*m'(x)|\dd x, 
\end{equation}
whence $\|H*m-H*m'\|_{\Ls^1(\R)} = \Ws_1(m,m')$.

As a consequence, the first term in the right-hand side of~\eqref{eq:decompscl} rewrites
\begin{equation*}
  \|u_0 - u_{\infty,0}[\rx(n)]\|_{\Ls^1(\R)} = \Ws_1(m,\mu^n), \qquad \mu^n = \frac{1}{n} \sum_{k=1}^n \delta_{x_k(n)}.
\end{equation*}
Precise bounds on $\Ws_1(m,\mu^n)$ in terms of $n$ for the optimal discretisation of $m$ are derived in Lemma~\ref{lem:discr}. They depend heavily on the tail of $m$. In particular, an important remark to be done at this point is the following. Assume that $m$ has an infinite first order moment. Then, since by~\eqref{eq:defiwass} and the triangle inequality,
\begin{equation*}
  \Ws_1(m,m') \geq \int_{x \in \R} |x|m(\dd x) - \int_{x' \in \R} |x'|m'(\dd x'),
\end{equation*}
any approximation of $m$ by a measure $\mu^n$ with finite first order moment necessarily satisfies $\Ws_1(m,\mu^n)=+\infty$. As a consequence, there is no purpose in trying to compute a rate of convergence in this case. Therefore, although our results hold true without any assumption on $m$, they only have a nontrivial content when $m$ has a finite first order moment.

The non-entropicity error is then addressed in Section~\ref{s:scl}, where given arbitrary $n \geq 1$ and $\rx \in D_n$, an estimation is first derived on the $\Ls^1$ distance between $u_n[\rx](t, \cdot)$ and $u_{\infty}[\rx](t, \cdot)$ in Proposition~\ref{prop:rateSPD}. This result holds under the following strengthening of Assumption~\eqref{ass:C}.
\begin{enumerate}[ref=LC, label=(LC)]
  \item\label{ass:LC} The function $\lambda=\Lambda'$ is $\ConstLip$-Lipschitz continuous.
\end{enumerate}
Combining the results of Section~\ref{s:ci} with Proposition~\ref{prop:rateSPD} yields complete rates of convergence of the SPD, as is stated in Theorem~\ref{theo:rateSPD}.

\subsection{MSPD and diagonal hyperbolic systems} This subsection is dedicated to the introduction of the MSPD, which allows to approximate the semigroup solution to diagonal hyperbolic systems in one space dimension. We refer to~\cite{jr} for more details.

\subsubsection{Diagonal hyperbolic systems} Let us fix an integer $d \geq 2$, and consider the diagonal hyperbolic system
\begin{equation}\label{eq:syst}
  \forall \gamma \in \{1, \ldots, d\}, \qquad \left\{\begin{aligned}
    & \partial_t u^{\gamma} + \lambda^{\gamma}(\bu)\partial_x u^{\gamma} = 0, \qquad t \geq 0, \quad x \in \R,\\
    & u^{\gamma}(0,x) = u_0^{\gamma}(x),
  \end{aligned}\right.
\end{equation}
for nonconstant, monotonic and bounded initial data $u^1_0, \ldots, u^d_0$. Once again, we shall assume that there exists $\bm = (m^1, \ldots, m^d) \in \Ps(\R)^d$ such that, for all $\gamma \in \{1, \ldots, d\}$, $u_0^{\gamma} = H*m^{\gamma}$, and look for solutions $\bu = (u^1, \ldots, u^d)$ of~\eqref{eq:syst} such that, for all $t \geq 0$, for all $\gamma \in \{1, \ldots, d\}$, $u^{\gamma}(t,\cdot)$ remains the CDF of a probability measure $m^{\gamma}_t$ on the real line. The characteristic fields $\lambda^1, \ldots, \lambda^d$ are therefore defined on $[0,1]^d$, and we shall extend Assumptions~\eqref{ass:C} and~\eqref{ass:LC} as follows.
\begin{enumerate}[ref=C, label=(C)]
  \item\label{ass:C:2} For all $\gamma \in \{1, \ldots, d\}$, the function $\lambda^{\gamma}$ is continuous on $[0,1]^d$.
\end{enumerate}
Under Assumption~\eqref{ass:C:2}, we denote $\ConstBoundS = \max_{1 \leq \gamma \leq d} \sup_{\bu \in [0,1]^d} |\lambda^{\gamma}(\bu)|$.
\begin{enumerate}[ref=LC, label=(LC)]
  \item\label{ass:LC:2} There exists $\ConstLip \in [0,+\infty)$ such that
  \begin{equation*}
    \forall \gamma \in \{1, \ldots, d\}, \quad \forall \bu,\bv \in [0,1]^d, \qquad |\lambda^{\gamma}(\bu) - \lambda^{\gamma}(\bv)| \leq \ConstLip\sum_{\gamma'=1}^d |u^{\gamma'}-v^{\gamma'}|.
  \end{equation*}
\end{enumerate}
We also require the system to be uniformly strictly hyperbolic, in the sense of the following assumption.
\begin{enumerate}[label=(USH), ref=USH]
  \item\label{ass:USH} There exists $\ConstUSH \in (0,+\infty)$ such that
  \begin{equation*}
    \forall \gamma \in \{1, \ldots, d-1\}, \qquad \inf_{\bu \in [0,1]^d} \lambda^{\gamma}(\bu) - \sup_{\bu \in [0,1]^d} \lambda^{\gamma+1}(\bu) \geq \ConstUSH.
  \end{equation*}
\end{enumerate}
On account of the fact that the system~\eqref{eq:syst} is written in a nonconservative form, both the notions of weak and entropy solution are not canonically defined. An appropriate notion of weak solution is introduced in~\cite[Definition~2.4.1]{jr}, while a criterion for uniqueness is stated in~\cite[Definition~8.2.5]{jr} by adapting the notion of viscosity solution by Bianchini and Bressan~\cite{bianbres}, to which we shall refer as the semigroup solution to~\eqref{eq:syst}. These notions are used in Theorem~\ref{theo:cvMSPD} below.

\subsubsection{Typewise and Multitype Sticky Particle Dynamics} In order to approximate each coordinate $u^{\gamma}$ of the solution to the system~\eqref{eq:syst}, we shall now introduce $d$ systems of $n$ particles evolving on the real line, each system being associated with a type $\gamma \in \{1, \ldots, d\}$, so that the empirical CDF of the system of type $\gamma$ is supposed to approximate $u^{\gamma}$. The $k$-th particle of type $\gamma$ is referred to as the particle $\gamma:k$, and the set of such indices is denoted $\Part$. A configuration is described by an element $\x = (x^{\gamma}_k)_{\gamma:k \in \Part}$ of the Cartesian product $D_n^d$. The normalised $\Ls^1$ distance is defined on $D_n^d$ by
\begin{equation*}
  \|\x-\y\|_1 := \frac{1}{n} \sum_{\gamma:k \in \Part} |x^{\gamma}_k-y^{\gamma}_k|.
\end{equation*}

The first step towards the definition of the MSPD is the introduction of the Typewise Sticky Particle Dynamics (TSPD). Given an array of initial velocities $\bblambda = (\rblambda^{\gamma}_k)_{\gamma:k \in \Part}$ and an initial configuration $\x \in D_n^d$, we denote by
\begin{equation*}
  \tPhi[\bblambda](\x;t) = (\tPhi_k^{\gamma}[\bblambda](\x;t))_{\gamma:k \in \Part}
\end{equation*}
the process in $D_n^d$ obtained by letting the system of type $\gamma$ evolve according to the SPD with initial configuration $\rx^{\gamma} = (x^{\gamma}_1, \ldots, x^{\gamma}_n) \in D_n$ and initial velocity vector $\rblambda^{\gamma} = (\rblambda^{\gamma}_1, \ldots, \rblambda^{\gamma}_n) \in \R^n$, without any interaction with other systems. 

The MSPD is built from the TSPD as follows: 
\begin{enumerate}[ref=\roman*, label=(\roman*)]
  \item the particle $\gamma:k$ is initialised with an array of velocities corresponding to a discretisation of the function $\lambda^{\gamma}$ at a point $\bu \in [0,1]^d$ recording the rank of the particle $\gamma:k$ in each of the $d$ systems, see~\eqref{eq:vitesses} below;
  \item when clusters of particles of different types collide, they cross each other and the TSPD is restarted with initial velocities depending on the post-collisional rank of the particles in each system.
\end{enumerate}
Assumption~\eqref{ass:USH} essentially implies that whatever the arrangement of the particles, particles of lower type always have a larger velocity. This prescribes the post-collisional order to update the velocities of the TSPD, and ensures that clusters of different types drift away from each other immediately after a collision. For an initial configuration $\x \in D_n^d$, the array of initial velocities $\tblambda(\x)=(\tlambda^{\gamma}_k(\x))_{\gamma:k \in \Part}$ is defined under Assumption~\eqref{ass:C:2} by
\begin{equation}\label{eq:vitesses}
  \tlambda_k^{\gamma}(\x) := n \int_{w=(k-1)/n}^{k/n} \lambda^{\gamma}\left(\omega_{\gamma:k}^1(\x), \ldots, \omega_{\gamma:k}^{\gamma-1}(\x), w, \omega_{\gamma:k}^{\gamma+1}(\x), \ldots, \omega_{\gamma:k}^d(\x)\right)\dd w,
\end{equation}
where $\omega^{\gamma'}_{\gamma:k}(\x)$ denotes the (scaled) rank of the particle $\gamma:k$ within the system of type $\gamma'$, formally defined by
\begin{equation*}
  \omega^{\gamma'}_{\gamma:k}(\x) := \begin{cases}
    \frac{1}{n} \sum_{k'=1}^n \ind{x^{\gamma'}_{k'} < x^{\gamma}_k} & \text{if $\gamma'<\gamma$,}\\
    \frac{1}{n} \sum_{k'=1}^n \ind{x^{\gamma'}_{k'} \leq x^{\gamma}_k} & \text{if $\gamma'>\gamma$,}\\
  \end{cases}
\end{equation*}
where particles of different types sharing the same location are counted according to the post-collisional rank imposed by Assumption~\eqref{ass:USH}. 

The resulting dynamics in $D_n^d$ is the MSPD, denoted by $(\Phi(\x;t))_{t \geq 0}$. More details on its construction are given in~\cite[Subsection~3.2]{jr}, where it is proved that it defines a continuous flow. Its stability with respect to the initial configuration is described by the next result, which forms the core of the article~\cite{jr}.

\begin{prop}\cite[Theorem~2.5.2]{jr}\label{prop:stabMSPD}
  Under Assumptions~\eqref{ass:LC:2} and~\eqref{ass:USH}, there exists $\ConstStab_1 \in [1,+\infty)$ depending only on $d$ and the ratio $\ConstLip/\ConstUSH$ such that, for all $n \geq 1$, for all $\x, \y \in D_n^d$, for all $0 \leq s \leq t$,
  \begin{equation*}
    \|\Phi(\x;t)-\Phi(\y;t)\|_1 \leq \ConstStab_1 \|\Phi(\x;s)-\Phi(\y;s)\|_1.
  \end{equation*}
\end{prop}
In contrast with Proposition~\ref{prop:stabSPD}, we note that no stability result with respect to the characteristic fields $\lambda^1, \ldots, \lambda^d$ is available.

\subsubsection{Approximation of the diagonal hyperbolic system} Similarly to the scalar case, we define the empirical distribution of the system of type $\gamma$ in the MSPD at time $t \geq 0$ as the probability measure on the real line given by
\begin{equation*}
  \mu^{\gamma}_t[\x] := \frac{1}{n} \sum_{k=1}^n \delta_{\Phi^{\gamma}_k(\x;t)},
\end{equation*}
and the vector of empirical CDFs $\bu_n[\x] = (u^1_n[\x], \ldots, u^d_n[\x])$ by
\begin{equation*}
  u^{\gamma}_n[\x](t,x) := H*\mu^{\gamma}_t[\x](x) = \frac{1}{n} \sum_{k=1}^n \ind{\Phi^{\gamma}_k(\x;t) \leq x}.
\end{equation*}

It is proved in~\cite[Proposition~4.2.1]{jr} that, for all $\x \in D_n^d$, $\bu_n[\x]$ is a weak solution to the system~\eqref{eq:syst}. Given a sequence of initial configurations $(\x(n))_{n \geq 1}$ such that, for all $n \geq 1$, $\x \in D_n^d$ and, for all $\gamma \in \{1, \ldots, d\}$, the empirical distribution
\begin{equation*}
  \mu^{\gamma}_0[\x(n)] = \frac{1}{n} \sum_{k=1}^n \delta_{x^{\gamma}_k(n)}
\end{equation*}
converges weakly to some $m^{\gamma} \in \Ps(\R)$, one could by analogy with Theorem~\ref{theo:cvSPD} expect $\bu_n[\x(n)]$ to converge to a weak solution of the system~\eqref{eq:syst} with initial data $u^{\gamma}_0 = H*m^{\gamma}$, satisfying some specific entropy-like condition making it unique and physically meaningful. Although it is true that, up to extracting a subsequence, $\bu_n[\x(n)]$ actually converges to a weak solution of the system~\cite[Theorem~2.4.5]{jr}, following Bianchini and Bressan's construction we were only able in~\cite{jr} to identify the limit if it satisfies some semigroup and stability estimate with respect to the initial data. For this purpose, we introduced~\cite[Definition~2.6.4]{jr} a discretisation operator $\chi_n : \Ps(\R)^d \to D_n^d$, defined by $\chi_n\bm = \x(n)$ with
\begin{equation}\label{eq:chin}
  \forall \gamma:k \in \Part, \qquad x_k^{\gamma}(n) = (n+1)\int_{w=(2k-1)/(2(n+1))}^{(2k+1)/(2(n+1))} (H*m^{\gamma})^{-1}(w)\dd w.
\end{equation}
For all $\gamma \in \{1, \ldots, d\}$, the empirical measure $\frac{1}{n} \sum_{k=1}^n \delta_{x_k^{\gamma}(n)}$ converges weakly to $m^{\gamma}$, see~\cite[Lemma~8.1.5]{jr}, and we obtained the following convergence theorem for the MSPD, which follows from the more general statement of~\cite[Theorem~2.6.5]{jr}.

\begin{theo}\label{theo:cvMSPD}
  Under Assumptions~\eqref{ass:LC:2} and~\eqref{ass:USH}, let $\bm = (m^1, \ldots, m^d) \in \Ps(\R)^d$. 
  
  For all $t \geq 0$, for all $\gamma \in \{1, \ldots, d\}$, the empirical distribution $\mu^{\gamma}_t[\chi_n\bm]$ converges weakly to the probability measure $m^{\gamma}_t \in \Ps(\R)$ such that $\bu = (u^1, \ldots, u^d)$ defined by $u^{\gamma}(t,x) = H*m^{\gamma}_t(x)$ is the unique semigroup solution to the diagonal hyperbolic system~\eqref{eq:syst} with initial data $u^{\gamma}_0 = H*m^{\gamma}$. Equivalently, for all $t \geq 0$, for all $\gamma \in \{1, \ldots, d\}$, the empirical CDF $u^{\gamma}_n[\chi_n\bm](t,\cdot)$ converges $\dd x$-almost everywhere to $u^{\gamma}(t, \cdot)$.
  
  Each coordinate $u^{\gamma}$ of the semigroup solution $\bu$ satisfies the preservation of total variation and finite speed of propagation properties stated in Theorem~\ref{theo:kruz}. The stability property writes as follows: for all $\bm, \bm' \in \Ps(\R)^d$, the semigroup solutions $\bu$ and $\bv$ to the system~\eqref{eq:syst} with respective initial data $\bu_0$, $\bv_0$ defined by $u^{\gamma}_0 = H*m^{\gamma}$, $v^{\gamma}_0 = H*m'^{\gamma}$ satisfy, for all $0 \leq s \leq t$,
  \begin{equation*}
    \begin{aligned}
      \|\bu(t,\cdot)-\bv(t,\cdot)\|_{\Ls^1(\R)^d} & := \sum_{\gamma=1}^d \|u^{\gamma}(t,\cdot)-v^{\gamma}(t,\cdot)\|_{\Ls^1(\R)}\\
      & \leq \ConstStab_1\sum_{\gamma=1}^d \|u^{\gamma}(s,\cdot)-v^{\gamma}(s,\cdot)\|_{\Ls^1(\R)} =: \ConstStab_1\|\bu(s,\cdot)-\bv(s,\cdot)\|_{\Ls^1(\R)^d}.
    \end{aligned}
  \end{equation*}
\end{theo}

The notion of semigroup solution is adapted from~\cite{bianbres} and detailed in~\cite[Definition~8.2.5]{jr}.

In Corollary~\ref{cor:cvMSPD}, we shall generalise this result and prove the convergence of the empirical cumulative distribution functions of the MSPD to the semigroup solution for a large class of sequences of initial configurations $(\x(n))_{n \geq 1}$ approximating a vector of initial measures $\bm = (m^1, \ldots, m^d)$ such that, for all $\gamma \in \{1, \ldots, d\}$, $m^{\gamma}$ has a finite first-order moment.

\subsubsection{Rate of convergence and approximation by the iterated TSPD} The second purpose of this paper is to supplement Theorem~\ref{theo:cvMSPD} with a rate of convergence of $\bu_n[\chi_n\bm](t,\cdot)$, or more generally $\bu_n[\x(n)](t,\cdot)$ for a sequence of initial configurations $(\x(n))_{n \geq 1}$ approximating $\bm$, to $\bu(t,\cdot)$. As in the scalar case, we split the approximation error by writing
\begin{equation*}
  \begin{aligned}
    & \|\bu(t, \cdot) - \bu_n[\x(n)](t, \cdot)\|_{\Ls^1(\R)^d}\\
    & \qquad \leq \|\bu(t, \cdot) - \bu_{\infty}[\x(n)](t, \cdot)\|_{\Ls^1(\R)^d} + \|\bu_{\infty}[\x(n)](t, \cdot) - \bu_n[\x(n)](t, \cdot)\|_{\Ls^1(\R)^d}\\
    & \qquad \leq \ConstStab_1 \|\bu_0 - \bu_{\infty,0}[\x(n)]\|_{\Ls^1(\R)^d} + |\bu_{\infty}[\x(n)](t, \cdot) - \bu_n[\x(n)](t, \cdot)\|_{\Ls^1(\R)^d},
  \end{aligned} 
\end{equation*}
where $\bu_{\infty}[\x(n)]$ is the semigroup solution, in the sense of Theorem~\ref{theo:cvMSPD}, of the system~\eqref{eq:syst} with initial condition $\bu_{\infty,0}[\x(n)]$ defined by, for all $\gamma \in \{1, \ldots, d\}$, $u^{\gamma}_{\infty,0}[\x(n)] = H*\mu^{\gamma}_0[\x(n)]$. In the last line above, we used the stability estimate of Theorem~\ref{theo:cvMSPD}Â on semigroup solutions.

The estimation of the discretisation error of the initial data follows from the analysis of the scalar case carried out in Section~\ref{s:ci}. We shall use the distance $\Ws_1^{(d)}$ on $\Ps(\R)^d$ defined by, for all $\bm=(m^1, \ldots, m^d), \bm'=(m'^1, \ldots, m'^d) \in \Ps(\R)^d$,
\begin{equation}\label{eq:W1d}
  \Ws_1^{(d)}(\bm,\bm') = \sum_{\gamma=1}^d \Ws_1(m^{\gamma},m'^{\gamma}) = \sum_{\gamma=1}^d \|H*m^{\gamma}-H*m'^{\gamma}\|_{\Ls^1(\R)} = ||\bu-\bv||_{\Ls^1(\R)^d},
\end{equation}
with $\bu:= (H*m^{1},\cdots,H*m^{d})$, $\bv:= (H*m'^{\gamma},\cdots,H*m'^{d})$.

The analysis of the error due to the evolution of the MSPD turns out to be more delicate than in the scalar case, and we shall resort to an approximation of this dynamics by the following scheme. Fix a time step $\Delta > 0$ and define the process $(\tPhi_{\Delta}(\x;t))_{t \geq 0}$ in $D_n^d$ by letting, on each time interval $[(L-1)\Delta, L\Delta]$, $L \geq 1$, the particles evolve according to the TSPD with initial velocity vector $\tblambda(\tPhi_{\Delta}(\x;(L-1)\Delta))$. This amounts to neglecting the collisions between clusters of different types on $[(L-1)\Delta, L\Delta]$ and only updating the velocities with respect to the new ordering of the system at the end of each such interval. From a computational point of view, this approximated dynamics is expected to be easier to simulate than the MSPD, as one does not have to keep track of all the collisions. The approximation error induced by this scheme is addressed in Section~\ref{s:syst}, where we also use it as a theoretical tool to obtain rates of convergence in Theorem~\ref{theo:rateMSPD}. We finally present a numerical implementation of this scheme in Section~\ref{s:num}.


\section{\texorpdfstring{$\Ws_1$}{W1} discretisation error of initial conditions}\label{s:ci}

In this section, we want to approximate the probability measure $m$ on the real line, with CDF $F = H*m$, by the empirical measure 
\begin{equation*}
  \mu^n=\frac{1}{n}\sum_{k=1}^n\delta_{x_k}
\end{equation*}
of a vector $\rx=(x_1, \ldots, x_n)$ of $n$ deterministic points on the real line, with $x_1\leq \cdots \leq x_n$. On account of~\eqref{eq:W1L1}, one has
\begin{equation*}
   \Ws_1(m,\mu^n)=\sum_{k=1}^n\int_{y=F^{-1}((k-1)/n)}^{F^{-1}(k/n)}|y-x_k|m(\dd y),
\end{equation*}
where we recall the definition~\eqref{eq:pseudoinv} of the pseudo-inverse $F^{-1}$ and complement it with the convention $F^{-1}(0)=\inf_{v\in(0,1)}F^{-1}(v)$ and $F^{-1}(1)=\sup_{v\in(0,1)}F^{-1}(v)$.

The choice 
\begin{equation}\label{eq:optx}
  x_k=F^{-1}\left(\frac{2k-1}{2n}\right),
\end{equation}
the median of the image of the uniform law on $[(k-1)/n,k/n]$ by $F^{-1}$, for all $k\in\{1,\ldots,n\}$, minimises this Wasserstein distance. Let us now analyse this optimal choice. First,
\begin{equation*}
  \begin{aligned}
    \Ws_1(m,\mu^n) & = \sum_{k=1}^n\int_{u=(2k-1)/(2n)}^{k/n}\left(F^{-1}(u)-F^{-1}(u-1/(2n))\right)\dd u\\
    & \leq \int_{u=1/(2n)}^1\left(F^{-1}(u)-F^{-1}(u-1/(2n))\right)\dd u.
  \end{aligned}
\end{equation*}
When 
\begin{equation*}
  \int_{u=0}^1|F^{-1}(u)|\dd u<+\infty,
\end{equation*}
that is to say $m$ has a finite first order moment, one deduces that 
\begin{equation*}
  \Ws_1(m,\mu^n)\leq\int_{u=1-1/(2n)}^1 F^{-1}(u)\dd u-\int_{u=0}^{1/(2n)}F^{-1}(u)\dd u,
\end{equation*}
where the right-hand side goes to $0$ when $n\to+\infty$ by Lebesgue's theorem. With the discussion of the case where the first order moment of $m$ is infinite at the end of~\S\ref{sss:rateSPD}, we deduce the first statement in the next lemma.

\begin{lem}\label{lem:discr}
  Let $\mu^n$ be defined as in the discussion above. The behaviour of $\Ws_1(m,\mu^n)$ depends on the tail of $m$ as is described below.
  \begin{enumerate}[label=(\roman*), ref=\roman*]
    \item\label{it:discr:1} One has $\lim_{n\to+\infty}\Ws_1(m,\mu^n)=0$ if and only if $m$ has a finite first order moment.
    \item\label{it:discr:2} For all $n \geq 1$,
    \begin{equation*}
      \Ws_1(m,\mu^n) \leq \frac{1}{\sqrt{n}} \int_{x \in \R}\sqrt{F(x)(1-F(x))}\dd x,
    \end{equation*}
    where the right-hand side may be infinite.
    \item\label{it:discr:3} If there exist $-\infty<a<b<+\infty$ such that $m([a,b])=1$, then for all $n \geq 1$,
    \begin{equation*}
      \Ws_1(m,\mu^n) \leq \frac{b-a}{2n}.
    \end{equation*}
    \item\label{it:discr:4} If $m(\dd x)=\ind{x \in [a,b]}f(x)\dd x$ with $f$ positive on $[a,b]$ where $-\infty<a<b<+\infty$, then 
    \begin{equation*}
      \lim_{n\to+\infty} n\Ws_1(m,\mu^n)=\frac{b-a}{4}.
    \end{equation*}
  \end{enumerate}
\end{lem}
\begin{proof}
  The assertion~\eqref{it:discr:2} follows from the comparison of $\Ws_1(m,\mu^n)$ with the expected Wasserstein distance between $m$ and the empirical measure $\frac{1}{n}\sum_{k=1}^n \delta_{X_k}$ of independent random variables $X_1, \ldots, X_n$ with identical distribution $m$, a detailed study of which was carried out by Bobkov and Ledoux~\cite{BobLed14}. By the optimality of the choice of $\mu^n$, we first have
  \begin{equation*}
    \Ws_1(m,\mu^n)\leq\Exp\left[\Ws_1\left(m,\frac{1}{n}\sum_{k=1}^n\delta_{X_k}\right)\right].
  \end{equation*}
  According to the proof of~\cite[Theorem~3.2]{BobLed14}, the right-hand side is not greater than
  \begin{equation*}
    \begin{aligned}
      \Exp^{1/2}\left[\Ws_1^2\left(m,\frac{1}{n}\sum_{k=1}^n\delta_{X_k}\right)\right]&=\Exp^{1/2}\left[\left(\int_{x \in \R}\left|F(x)-\frac{1}{n}\sum_{k=1}^n\ind{X_k\leq x}\right|\dd x\right)^2\right]\\
      & \leq \int_{x \in \R}\Exp^{1/2}\left[\left(F(x)-\frac{1}{n}\sum_{k=1}^n\ind{X_k\leq x}\right)^2\right]\dd x\\
      & =\frac{1}{\sqrt{n}} \int_{x \in \R}\sqrt{F(x)(1-F(x))}\dd x.
    \end{aligned}
  \end{equation*}

  If $m([a,b])=1$, then the mass of the measure $\nu$ on $(0,1)$ such that $\nu((0,u))=F^{-1}(u)-F^{-1}(0^+)$ for all $u\in(0,1)$ is smaller than $b-a$. Moreover,
  \begin{equation*}
    \begin{aligned}
      \Ws_1(m,\mu^n) & =\sum_{k=1}^n\int_{u=(2k-1)/(2n)}^{k/n}\int_{v=0}^1 \ind{u-1/(2n)\leq v<u}\nu(\dd v)\dd u\\
      & =\sum_{k=1}^n \int_{v \in [(k-1)/n,k/n)}\min\left(v-\frac{k-1}{n},\frac{k}{n}-v\right)\nu(\dd v)\\
      & \leq \frac{\nu((0,1))}{2n}.
    \end{aligned}
  \end{equation*}

  Let us finally assume that $m(\dd x)=\ind{x \in [a,b]}f(x)\dd x$ with $f$ positive on $[a,b]$. Since $m$ is the image of the Lebesgue measure on $[0,1]$ by $F^{-1}$, one has 
  \begin{equation*}
    F^{-1}(u)-F^{-1}(u-1/(2n)) = \int_{x=F^{-1}(u-1/(2n))}^{F^{-1}(u)}\frac{m(\dd x)}{f(x)}=\int_{v=u-1/(2n)}^u\frac{\dd v}{f(F^{-1}(v))}
  \end{equation*}
  for all $u \in (1/(2n), 1)$. The continuity of translations in $\Ls^1$ ensures that 
  \begin{equation*}
    \lim_{n \to +\infty} \int_{u=0}^1 \left|n\ind{u \in (1/(2n),1)}\left(F^{-1}(u)-F^{-1}(u-1/(2n))\right) - \frac{1}{2f(F^{-1}(u))}\right| \dd u = 0.
  \end{equation*}
  Since
  \begin{equation*}
    \begin{aligned}
      & \left|n\Ws_1(m,\mu^n)-\frac{1}{2}\sum_{k=1}^n\int_{u=(2k-1)/(2n)}^{k/n}\frac{\dd u}{f(F^{-1}(u))}\right|\\
      & \leq \int_{u=0}^1 \left|n\ind{u \in (1/(2n),1)}\left(F^{-1}(u)-F^{-1}(u-1/(2n))\right)-\frac{1}{2f(F^{-1}(u))}\right|\dd u,
    \end{aligned}
  \end{equation*}
  it is enough to check that 
  \begin{equation*}
    \lim_{n \to +\infty} \sum_{k=1}^n\int_{u=(2k-1)/(2n)}^{k/n}\frac{\dd u}{f(F^{-1}(u))} = \frac{1}{2}\int_{u=0}^1\frac{\dd u}{f(F^{-1}(u))}.
  \end{equation*}
  This follows from the weak convergence of $2\sum_{k=1}^n\ind{u \in ((2k-1)/(2n),k/n)}\dd u$ to $\ind{u \in (0,1)}\dd u$ and the density of continuous and bounded functions in $\Ls^1(\R)$.
\end{proof}

\begin{rk}
  Since 
  \begin{equation*}
    \int_{u=0}^1|F^{-1}(u)|\dd u=\int_{x\in \R}|x|m(\dd x) = \int_{x=0}^{+\infty} ((1-F(x)) + F(-x))\dd x ,
  \end{equation*}
  the condition 
  \begin{equation*}
    \int_{u=0}^1|F^{-1}(u)|\dd u<+\infty
  \end{equation*}
  is equivalent to 
  \begin{equation*}
    \int_{x \in \R}F(x)(1-F(x))\dd x<+\infty.
  \end{equation*}
  Similarly, the condition 
  \begin{equation*}
    \int_{x \in \R}\sqrt{F(x)(1-F(x))}\dd x < +\infty
  \end{equation*}
  is an assumption on the decay of the tails of $m$. Since $F(x)(1-F(x))\leq \sqrt{F(x)(1-F(x))}$, it implies that 
  \begin{equation*}
    \int_{x \in \R}|x|m(\dd x) < +\infty.
  \end{equation*}
  On the other hand, if for some $\alpha>2$, 
  \begin{equation*}
    \int_{x \in \R} |x|^\alpha m(\dd x) < +\infty,
  \end{equation*}
  then 
  \begin{equation*}
    \int_{x \in \R}(1+|x|^{\alpha-1})F(x)(1-F(x))\dd x<+\infty
  \end{equation*}
  and, by the Cauchy-Schwarz inequality,
  \begin{equation*}
    \left(\int_{x \in \R}\sqrt{F(x)(1-F(x))}\dd x\right)^2 \leq \int_{x \in \R}(1+|x|^{\alpha-1})^{-1}\dd x \int_{x \in \R}(1+|x|^{\alpha-1})F(x)(1-F(x))\dd x < +\infty.
  \end{equation*}
\end{rk}

We finally discuss some other cases where $\Ws_1(m,\mu^n)$ can be estimated. If $m$ has a positive density $f$ on $(0,\infty)$, one has
\begin{equation*}
  \begin{aligned}
    \Ws_1(m,\mu^n) & =\sum_{k=1}^n\int_{v=(k-1)/n}^{k/n}\min\left(v-\frac{k-1}{n},\frac{k}{n}-v\right)\frac{\dd v}{f(F^{-1}(v))}\\
    & \leq \frac{1}{2n}\int_{v=0}^{1-1/(2n)}\frac{\dd v}{f(F^{-1}(v))} + \int_{1-1/(2n)}^1\frac{(1-v)\dd v}{f(F^{-1}(v))}\\
    & = \frac{1}{2n}\left(F^{-1}\left(1-\frac{1}{2n}\right)-F^{-1}(0^+)\right) + \int_{v= 1-1/(2n)}^1\frac{(1-v)\dd v}{f(F^{-1}(v))}.
  \end{aligned}
\end{equation*}
On the other hand, if $u \mapsto 1/f(F^{-1}(u))$ is nondecreasing on $(\underline{u},1)$ with $\underline{u}\in(0,1)$, then
\begin{equation*}
  \begin{aligned}
    \Ws_1(m,\mu^n) & =\sum_{k=1}^n\int_{v=(k-1)/n}^{k/n}\min\left(v-\frac{k-1}{n},\frac{k}{n}-v\right)\frac{\dd v}{f(F^{-1}(v))}\\
    & \geq \frac{1}{4n}\sum_{k=1}^n\int_{v=(4k-3)/(4n)}^{(4k-1)/(4n)}\frac{\dd v}{f(F^{-1}(v))}\\
    & \geq \frac{1}{8n}\sum_{k\geq n\underline{u}+5/4}\int_{v=(4k-5)/(4n)}^{(4k-1)/(4n)}\frac{\dd v}{f(F^{-1}(v))}\\
    & \geq \frac{1}{8n}\int_{\underline{u}+1/4n}^{1-1/4n}\frac{\dd v}{f(F^{-1}(v))}\\
    & = \frac{1}{8n}\left(F^{-1}\left(1-\frac{1}{4n}\right)-F^{-1}\left(\underline{u}+\frac{1}{4n}\right)\right).
  \end{aligned}
\end{equation*}
We apply these estimates on the following two examples.
\begin{itemize}
  \item For $F(x)=\ind{x>0}(1-x^{-\alpha})$ with $\alpha>1$, one has 
  \begin{equation*}
    F^{-1}(u)=(1-u)^{-1/\alpha}.
  \end{equation*}
  Since 
  \begin{equation*}
    \int_{v=1-1/(2n)}^1 \frac{(1-v)\dd v}{f(F^{-1}(v))}=\grandO(n^{-1+1/\alpha})
  \end{equation*}
  and for $c>0$, $F^{-1}(1-c/n)=\grandO(n^{1/\alpha})$, we conclude that $\Ws_1(m,\mu^n)=\grandO(n^{-1+1/\alpha})$.
  \item For $F(x)=\ind{x>0}(1-\exp(-x^\alpha))$ with $\alpha>0$, one has 
  \begin{equation*}
    F^{-1}(u)=(-\log(1-u))^{1/\alpha},
  \end{equation*}
  and we conclude that $\Ws_1(m,\mu^n)=\grandO((\log n)^{1/\alpha}/n)$.
\end{itemize}


\section{Rate of convergence of the SPD to scalar conservation laws}\label{s:scl}

Under the assumptions of Theorem~\ref{theo:cvSPD}, the purpose of this section is to estimate the $\Ls^1$ distance between the entropy solution $u$ of the scalar conservation law~\eqref{eq:scl} with initial condition $u_0=H*m$ for some $m \in \Ps(\R)$, and the empirical CDF $u_n[\rx(n)]$ associated with the SPD started at some configuration $\rx(n) = (x_1(n), \ldots, x_n(n)) \in D_n$, over finite time horizons. 

\begin{theo}\label{theo:rateSPD}
  Let $\Lambda$ satisfy Assumption~\eqref{ass:LC}. Then for all $n \geq 1$, for all $\rx(n) \in D_n$,
  \begin{equation*}
    \forall t \geq 0, \qquad \|u(t,\cdot) - u_n[\rx(n)](t,\cdot)\|_{\Ls^1(\R)} \leq \Ws_1(m,\mu_0[\rx(n)]) + \frac{t \ConstLip}{n}.
  \end{equation*}
  
  When $m([a,b])=1$ with $-\infty<a<b<+\infty$ and $\rx(n) = (x_1(n), \ldots, x_n(n))$ is given by 
  \begin{equation}\label{eq:discr:opt}
    x_k(n)=u_0^{-1}\left(\frac{2k-1}{2n}\right), \qquad k\in\{1,\ldots,n\},
  \end{equation}
  then for all $n \geq 1$,
  \begin{equation*}
    \forall t \geq 0, \qquad \|u(t,\cdot) - u_n[\rx(n)](t,\cdot)\|_{\Ls^1(\R)} \leq \frac{b-a+2t \ConstLip}{2n}.
  \end{equation*}
\end{theo}
\begin{proof}
Following the approach described in the introduction of the article, we first write, for all $t \geq 0$,
\begin{equation}\label{eq:decomp:scal}
  \begin{aligned}
    \|u(t,\cdot) - u_n[\rx(n)](t,\cdot)\|_{\Ls^1(\R)} & \leq \|u(t,\cdot) - u_{\infty}[\rx(n)](t,\cdot)\|_{\Ls^1(\R)}\\
    & \quad + \|u_{\infty}[\rx(n)](t,\cdot) - u_n[\rx(n)](t,\cdot)\|_{\Ls^1(\R)},
  \end{aligned}
\end{equation}
where $u_{\infty}[\rx(n)]$ refers to the entropy solution of the scalar conservation law~\eqref{eq:scl} with discretised initial condition $u_{\infty,0}[\rx(n)] = H*\mu_0[\rx(n)]$. The $\Ls^1$ stability property of Theorem~\ref{theo:kruz} for the scalar conservation law~\eqref{eq:scl} yields
\begin{equation*}
  \|u(t,\cdot) - u_{\infty}[\rx(n)](t,\cdot)\|_{\Ls^1(\R)} \leq \|u_0 - u_{\infty,0}[\rx(n)]\|_{\Ls^1(\R)} = \Ws_1(m,\mu_0[\rx(n)]),
\end{equation*}
which, according to~\eqref{it:discr:3} in Lemma~\ref{lem:discr}, is smaller than $(b-a)/(2n)$ if $m([a,b])=1$ and $\rx(n)$ is given by~\eqref{eq:discr:opt}. The following proposition allows to control the second term in the right-hand side of~\eqref{eq:decomp:scal}.
\end{proof}

\begin{prop}\label{prop:rateSPD}
  Under Assumption~\eqref{ass:LC}, let $n \geq 1$ and let $\rx \in D_n$. For all $t \geq 0$,
  \begin{equation*}
    \|u_{\infty}[\rx](t,\cdot) - u_n[\rx](t,\cdot)\|_{\Ls^1(\R)} \leq \frac{t \ConstLip}{n}.
  \end{equation*}
\end{prop}
\begin{proof}
  Instead of comparing $u_n[\rx]$ directly with the entropy solution $u_{\infty}[\rx]$, we first compare it with the empirical CDF of the SPD with $2n$ particles, started in the configuration where the positions in $\rx$ are duplicated, and then iterate this comparison. More precisely, for all $n \geq 1$, let us define the operator 
  \begin{equation*}
    \left\{\begin{array}{ccc}
      D_n & \to & D_{2n}\\
      \rx & \mapsto & \hat{\rx}
    \end{array}\right.
  \end{equation*}
  by $\hat{x}_{2k-1} = \hat{x}_{2k} = x_k$ for all $k \in \{1, \ldots, n\}$, and denote by $\hat{\rx}^m \in D_{2^m n}$ the $m$-th iteration of this operator. We first prove that, for all $n \geq 1$, for all $\rx \in D_n$,
  \begin{equation}\label{eq:controln2n:scal}
    \forall t \geq 0, \qquad \|u_n[\rx](t, \cdot) - u_{2n}[\hat{\rx}](t, \cdot)\|_{\Ls^1(\R)} \leq \frac{t \ConstLip}{2n}.
  \end{equation}
  In this purpose, we interpret the function $u_n[\rx]$ as the empirical CDF of the SPD with $2n$ particles, with initial configuration $\hat{\rx}$, and with modified velocity coefficients $\hat{\lambda}_1, \ldots, \hat{\lambda}_{2n}$ defined by
  \begin{equation*}
    \hat{\lambda}_{2k-1} = \hat{\lambda}_{2k} = n \int_{w=(k-1)/n}^{k/n} \lambda(w)\dd w,
  \end{equation*}
  for all $k \in \{1, \ldots, n\}$. Then~\eqref{eq:stabSPD} yields
  \begin{equation*}
    \begin{aligned}
      \|u_n[\rx](t,\cdot) - u_{2n}[\hat{\rx}](t,\cdot)\|_{\Ls^1} & \leq \frac{t}{n} \sum_{\hat{k}=1}^{2n} \left| \hat{\lambda}_{\hat{k}} - 2n \int_{w=(\hat{k}-1)/(2n)}^{\hat{k}/(2n)} \lambda(w) \dd w\right|\\
      & = t \sum_{k=1}^n \bigg\{\left| \int_{w=(k-1)/n}^{k/n} \lambda(w)\dd w - 2 \int_{w=(k-1)/n}^{(k-\frac{1}{2})/n} \lambda(w) \dd w\right|\\
      & \qquad + \left| \int_{w=(k-1)/n}^{k/n} \lambda(w)\dd w - 2 \int_{w=(k-\frac{1}{2})/n}^{k/n} \lambda(w) \dd w\right|\bigg\}
    \end{aligned}
  \end{equation*}
  and Assumption~\eqref{ass:LC} allows to bound the right-hand side above by $t \ConstLip/(2n)$, whence~\eqref{eq:controln2n:scal}.
  
  We now fix $n \geq 1$ and use~\eqref{eq:controln2n:scal} to write that, for all $M \geq 1$,
  \begin{equation*}
    \begin{aligned}
      \|u_n[\rx](t, \cdot) - u_{\infty}[\rx](t, \cdot)\|_{\Ls^1(\R)} & \leq \sum_{m=1}^M \|u_{2^{m-1}n}[\hat{\rx}^{m-1}](t, \cdot) - u_{2^mn}[\hat{\rx}^m](t, \cdot)\|_{\Ls^1(\R)}\\
      & \quad + \|u_{2^Mn}[\hat{\rx}^M](t, \cdot) - u_{\infty}[\rx](t, \cdot)\|_{\Ls^1(\R)}\\
      & \leq \sum_{m=1}^M \frac{t \ConstLip}{2^mn} + \|u_{2^Mn}[\hat{\rx}^M](t, \cdot) - u_{\infty}[\rx](t, \cdot)\|_{\Ls^1(\R)}.
    \end{aligned}
  \end{equation*}
  To complete the proof, we have to check that 
  \begin{equation}\label{eq:limM}
    \lim_{M \to +\infty} \|u_{2^Mn}[\hat{\rx}^M](t, \cdot) - u_{\infty}[\rx](t, \cdot)\|_{\Ls^1(\R)} = 0.
  \end{equation}
  To this aim, we note that, for all $M \geq 1$, the empirical measure $\mu_0[\hat{\rx}^M]$ associated with $\hat{\rx}^M$ does not depend on $M$ and coincides with $\mu_0[\rx]$. As a consequence, by Theorem~\ref{theo:cvSPD}, $u_{2^Mn}[\hat{\rx}^M](t, \cdot)$ converges to $u_{\infty}[\rx](t, \cdot)$, $\dd x$-almost everywhere. Besides, the initial measure $\mu_0[\rx]$ has compact support $\{x_1, \ldots, x_n\}$. Let $-\infty < a < x_1$ and $x_n \leq b < +\infty$, so that $u_{\infty,0}[\rx](a)=0$ and $u_{\infty,0}[\rx](b)=1$. Using the finite speed of propagation for both $u_{2^Mn}[\hat{\rx}^M]$ and $u_{\infty}[\rx]$, we deduce that 
  \begin{equation*}
    \forall x \not\in [a-t\ConstBoundS,b+t\ConstBoundS], \qquad u_{2^Mn}[\hat{\rx}^M](t,x) - u_{\infty}[\rx](t,x) = 0.
  \end{equation*}
  Therefore, $u_{2^Mn}[\hat{\rx}^M](t, \cdot)-u_{\infty}[\rx](t, \cdot)$ has a compact support which does not depend on $M$. As a consequence,~\eqref{eq:limM} follows from Lebesgue's theorem and the proof is completed.
\end{proof}

\begin{rk}\label{rk:rateSPD}
  If the flux function $\Lambda$ is concave, then $u_n[\rx]$ is actually the entropy solution to the scalar conservation law~\eqref{eq:scl} with discrete initial condition~\cite[Lemma~3.3]{boujam99}, so that 
  \begin{equation*}
    \|u_{\infty}[\rx](t,\cdot) - u_n[\rx](t,\cdot)\|_{\Ls^1(\R)}=0,
  \end{equation*}
  for all $n \geq 1$, $\rx \in D_n$ and $t \geq 0$. Note that this remark holds true even without Assumption~\eqref{ass:LC}.
  
  On the contrary, one can construct examples of convex flux functions for which the bound of Proposition~\ref{prop:rateSPD} provides the right order of magnitude for $\|u_{\infty}[\rx](t,\cdot) - u_n[\rx](t,\cdot)\|_{\Ls^1(\R)}$. Let us indeed consider the Burgers equation $\Lambda(u) = u^2/2$ with the initial condition $m=\delta_0$, the Dirac mass at $0$. Fix $n \geq 1$ and let $\rx = (0, \ldots, 0) \in D_n$. On the one hand, the entropy solution $u_{\infty}[\rx]$ writes
  \begin{equation*}
    u_{\infty}[\rx](t,x) = \begin{cases}
      0 & \text{if $x < 0$,}\\
      \frac{x}{t} & \text{if $0 \leq x < t$,}\\
      1 & \text{if $x \geq t$.}
    \end{cases}
  \end{equation*}
  On the other hand, the particles drift away from each other in the SPD, so that 
  \begin{equation*}
    \forall k \in \{1, \ldots, n\}, \qquad \phi_k[\rblambda](\rx;t) = t\frac{k-1/2}{n},
  \end{equation*}
  therefore
  \begin{equation*}
    u_n[\rx](t,x) = \begin{cases}
      0 & \text{if $x < \frac{t}{2n}$,}\\
      \frac{k}{n} & \text{if $t \frac{k-1/2}{n} \leq x < t \frac{k+1/2}{n}$ with $k \in \{1, \ldots, n-1\}$,}\\
      1 & \text{if $x \geq t \frac{n-1/2}{n}$.}
    \end{cases}
  \end{equation*}
  As a consequence,
  \begin{equation*}
    \|u_{\infty}[\rx](t,\cdot) - u_n[\rx](t,\cdot)\|_{\Ls^1(\R)} = \frac{t}{4n},
  \end{equation*}
  which is of the same order as the bound given in Proposition~\ref{prop:rateSPD}.
\end{rk}


\section{Rate of convergence of the MSPD to diagonal hyperbolic systems}\label{s:syst}


\subsection{Approximating the MPSD through the iterated TSPD}\label{ss:approxMSPD} Given $\Delta > 0$, we define the operators $\Phi_{\Delta}$ and $\tPhi_{\Delta}$ on $D_n^d$ by
\begin{equation*}
  \Phi_{\Delta}(\x) := \Phi(\x;\Delta), \qquad \tPhi_{\Delta}(\x) := \tPhi[\tblambda(\x)](\x;\Delta),
\end{equation*}
where we recall that $\Phi(\x;t)$ refers to the MSPD while $\tPhi[\tblambda(\x)](\x;t)$ is the TSPD with initial velocity vector $\tblambda(\x)$ given by~\eqref{eq:vitesses}. 

By the flow property of the MSPD, the $L$-th iteration $\Phi^L_{\Delta}(\x)$ is nothing but $\Phi(\x;L\Delta)$. On the other hand, the $L$-th iteration $\tPhi_{\Delta}(\x)$ is easier to compute as one does not have to take interactions between particles of different types into account. The purpose of this subsection is to estimate the error of the approximation of the MSPD by this iterated TSPD scheme. To this aim, given $\x \in D_n^d$, we first extend the iterated TSPD into a continuous process $(\tPhi_\Delta(\x;t))_{t \geq 0}$ in $D_n^d$ by interpolating between the points of the uniform grid with step $\Delta$ thanks to the TSPD. More precisely, for all $L \geq 1$, for all $t \in [(L-1)\Delta,L\Delta]$, we define
\begin{equation*}
  \tPhi_{\Delta}(\x;t) := \tPhi_{t-(L-1)\Delta}(\tPhi^{L-1}_{\Delta}(\x)) = \tPhi[\tblambda(\tPhi^{L-1}_{\Delta}(\x))](\tPhi^{L-1}_{\Delta}(\x);t-(L-1)\Delta).
\end{equation*}
Our purpose is to prove the following result.

\begin{prop}\label{prop:schema}
  Under Assumptions~\eqref{ass:LC:2} and~\eqref{ass:USH}, there exists $C>0$ that depends neither on $n \geq 1$ nor on $\Delta$ such that, for all $\x \in D_n^d$, 
  \begin{equation}\label{errtspd}
    \sup_{t \geq 0} \|\Phi(\x;t) - \tPhi_{\Delta}(\x;t)\|_1 \leq C \Delta.
  \end{equation}
\end{prop}
The value of $C$ is given in~\eqref{eq:C} below. Of course, for $t$ restricted to the uniform grid $(L\Delta)_{L\ge 0}$, the estimate~\eqref{errtspd} implies
\begin{equation*}\label{errLspd}
  \sup_{L \geq 0} \|\Phi^L_{\Delta}(\x) - \tPhi^L_{\Delta}(\x)\|_1 \leq C \Delta.
\end{equation*}
We begin by proving this restricted estimation before deducing \eqref{errtspd}. We first use the flow property of both the TSPD and the MSPD~\cite[Proposition~3.2.8]{jr} to write
\begin{equation}\label{eq:estimstab}
  \begin{aligned}
    \|\Phi^L_{\Delta}(\x) - \tPhi^L_{\Delta}(\x)\|_1 & \leq \sum_{\ell=1}^L \|\Phi^{L-\ell}_{\Delta}(\Phi_{\Delta}-\tPhi_{\Delta})\tPhi^{\ell-1}_{\Delta}(\x)\|_1\\
    & \leq \ConstStab_1\sum_{\ell=1}^L \|(\Phi_{\Delta}-\tPhi_{\Delta})\tPhi^{\ell-1}_{\Delta}(\x)\|_1,
  \end{aligned}
\end{equation}
where the last inequality follows from the discrete stability estimate of Proposition~\ref{prop:stabMSPD}.

The next lemma allows to estimate each term of the sum in the right-hand side. We recall from~\cite{jr} that, for all $\y \in D_n^d$, the set $\Rb(\y)$ contains the pairs of particles $(\alpha:i,\beta:j) \in (\Part)^2$ such that $\alpha < \beta$ and $y_i^{\alpha} < y_j^{\beta}$, so that in the MSPD started at $\y$, these particles undergo a collision at a finite and positive time $\tinter_{\alpha:i, \beta:j}(\y)$. We also define the first collision time in the MSPD started at $\y$ by
\begin{equation*}
  t^*(\y) := \min\{\tinter_{\alpha:i, \beta:j}(\y), (\alpha:i,\beta:j) \in \Rb(\y)\} \in (0,+\infty],
\end{equation*}
where we take the convention that $t^*(\y) = +\infty$ if $\Rb(\y)$ is empty.

\begin{lem}\label{lem:onestep}
  Under the assumptions of Proposition~\ref{prop:schema}, for all $\y \in D_n^d$ and all $\Delta\ge 0$, \begin{equation*}
    \|\Phi_{\Delta}(\y) - \tPhi_{\Delta}(\y)\|_1 \leq 2\Delta\frac{\ConstLip}{n^2} \sum_{(\alpha:i, \beta:j) \in \Rb(\y)} \ind{\tinter_{\alpha:i, \beta:j}(\y) \leq \Delta}.
  \end{equation*}
\end{lem}

\begin{proof}[Proof of Lemma~\ref{lem:onestep}]
  Let $t_0 := 0 < t_1 < \cdots < t_R < \Delta =: t_{R+1}$ refer to the successive instants of collisions (\ie $t^*(\y)$, $t^*(\Phi(\y;t^*(\y)))$, etc.) in the MSPD started at $\y$, on the time interval $[0,\Delta]$. For all $r \in \{0, \ldots, R+1\}$, let us denote $\y_r := \Phi(\y;t_r)$ and $\tilde{\y}_r := \tPhi[\tblambda(\y)](\y;t_r)$, so that $\y_0 = \tilde{\y}_0 = \y$ while $\y_{R+1} = \Phi_{\Delta}(\y)$ and $\tilde{\y}_{R+1} = \tPhi_{\Delta}(\y)$. Besides, the definition of the MSPD and flow property of both the TSPD and the MSPD yield, for all $r \in \{1, \ldots, R+1\}$,
  \begin{equation*}
    \y_r = \tPhi[\tblambda(\y_{r-1})](\y_{r-1};t_r - t_{r-1}), \qquad \tilde{\y}_r = \tPhi[\tblambda(\y)](\tilde{\y}_{r-1};t_r - t_{r-1}).
  \end{equation*}
  As a consequence,
  \begin{equation*}
    \begin{aligned}
      & \|\Phi_{\Delta}(\y) - \tPhi_{\Delta}(\y)\|_1 = \sum_{r=1}^{R+1} \|\y_r - \tilde{\y}_r\|_1 - \|\y_{r-1} - \tilde{\y}_{r-1}\|_1\\
      & \qquad = \sum_{r=1}^{R+1} \|\tPhi[\tblambda(\y_{r-1})](\y_{r-1};t_r - t_{r-1}) - \tPhi[\tblambda(\y)](\tilde{\y}_{r-1};t_r - t_{r-1})\|_1 - \|\y_{r-1} - \tilde{\y}_{r-1}\|_1\\
      & \qquad  \leq \frac{1}{n}\sum_{r=1}^{R+1} (t_r-t_{r-1}) \sum_{\gamma:k \in \Part} \left|\tlambda^{\gamma}_k(\y_{r-1}) - \tlambda^{\gamma}_k(\y)\right|,
    \end{aligned}
  \end{equation*}
  where the last line is obtained by applying the stability property of the SPD~\eqref{eq:stabSPD} typewise. Using the triangle inequality 
  \begin{equation*}
    \left|\tlambda^{\gamma}_k(\y_{r-1}) - \tlambda^{\gamma}_k(\y)\right| \leq \left|\tlambda^{\gamma}_k(\y_{r-1}) - \tlambda^{\gamma}_k(\y_{r-2})\right| + \cdots + \left|\tlambda^{\gamma}_k(\y_1) - \tlambda^{\gamma}_k(\y_0)\right|
  \end{equation*}
  and summing by parts, we obtain
  \begin{equation*}
    \|\Phi_{\Delta}(\y) - \tPhi_{\Delta}(\y)\|_1 \leq \frac{\Delta}{n} \sum_{r=1}^R \sum_{\gamma:k \in \Part}\left|\tlambda^{\gamma}_k(\y_r) - \tlambda^{\gamma}_k(\y_{r-1})\right|.
  \end{equation*}
  Now the definition~\eqref{eq:vitesses} of the coefficients $\tlambda^{\gamma}_k$ combined with Assumption~\eqref{ass:LC:2} imply
  \begin{equation*}
    \sum_{\gamma:k \in \Part}\left|\tlambda^{\gamma}_k(\y_r) - \tlambda^{\gamma}_k(\y_{r-1})\right| \leq \frac{\ConstLip}{n} \sum_{\gamma:k \in \Part} \sum_{\gamma' \not= \gamma} \sum_{k'=1}^n \ind{\tinter_{\alpha:i, \beta:j}(\y) = t_r},
  \end{equation*}
  where $(\alpha:i, \beta:j)$ in the last term refers to $(\gamma:k, \gamma':k')$ if $\gamma<\gamma'$ and to $(\gamma':k', \gamma:k)$ if $\gamma>\gamma'$. We note that the argument here is similar to the one employed in the proof of Fact~2 in~\cite[Lemma~7.2.4]{jr}.
  
  Since each pair $(\alpha:i, \beta:j) \in \Rb(\y)$ such that $\tinter_{\alpha:i, \beta:j}(\y) = t_r$ appears twice in the sum in the right-hand side above, we deduce that
  \begin{equation*}
    \sum_{\gamma:k \in \Part} \sum_{\gamma' \not= \gamma} \sum_{k'=1}^n \ind{\tinter_{\alpha:i, \beta:j}(\y) = t_r} = 2\sum_{(\alpha:i, \beta:j) \in \Rb(\y)} \ind{\tinter_{\alpha:i, \beta:j}(\y) = t_r},
  \end{equation*}
  whence
  \begin{equation*}
    \sum_{r=1}^R \sum_{\gamma:k \in \Part}\left|\tlambda^{\gamma}_k(\y_r) - \tlambda^{\gamma}_k(\y_{r-1})\right| \leq \frac{2 \ConstLip}{n} \sum_{(\alpha:i, \beta:j) \in \Rb(\y)} \ind{\tinter_{\alpha:i, \beta:j}(\y) \leq \Delta}
  \end{equation*}
  which completes the proof.
\end{proof}

For all $\y \in D_n^d$, we denote by
\begin{equation}\label{defndel}
  \Nb_{\Delta}(\y) := \sum_{(\alpha:i, \beta:j) \in \Rb(\y)} \ind{\tinter_{\alpha:i, \beta:j}(\y) \leq \Delta}
\end{equation}
the number of collisions occurring in the MSPD started as $\y$ on the time interval $[0,\Delta]$. Combining Lemma~\ref{lem:onestep} with the first estimate~\eqref{eq:estimstab}, we obtain
\begin{equation}\label{eq:estim2}
  \|\Phi^L_{\Delta}(\x) - \tPhi^L_{\Delta}(\x)\|_1 \leq 2\Delta\frac{\ConstLip\ConstStab_1}{n^2} \sum_{\ell=1}^L \Nb_{\Delta}(\tPhi^{\ell-1}_{\Delta}(\x)),
\end{equation} 
for all $L \geq 1$. Of course, each term $\Nb_{\Delta}(\tPhi^{\ell-1}_{\Delta}(\x))$ is bounded by $d(d-1)n^2/2$. Our purpose is to exhibit a bound of this order (with respect to $n$) for the whole sum, uniformly in $L$. We shall rely on the following remark, which is a consequence of Assumption~\eqref{ass:USH} and of the boundedness of the velocities.

\begin{rk}\label{rk:csqUSH}
  Let $\y \in D_n^d$ and $\alpha:i, \beta:j \in \Part$ such that $\alpha < \beta$. Denote $\z := \tPhi_{\Delta}(\y)$.
  \begin{itemize}
    \item If $(\alpha:i, \beta:j) \not\in \Rb(\y)$, then $(\alpha:i, \beta:j) \not\in \Rb(\z)$.
    \item $z^{\beta}_j - z^{\alpha}_i \leq  y^{\beta}_j - y^{\alpha}_i - \ConstUSH\Delta$ and, if $(\alpha:i, \beta:j) \in \Rb(\y)$ and $\tinter_{\alpha:i, \beta:j}(\y) \leq \Delta$, then $0 < y^{\beta}_j - y^{\alpha}_i \leq 2 \ConstBoundS \Delta$.
  \end{itemize}
\end{rk}

We deduce from the first part of the remark that, for all $\ell \in \{1, \ldots, L\}$, $\Rb(\tPhi^{\ell-1}_{\Delta}(\x)) \subset \Rb(\x)$, which allows us to rewrite
\begin{equation*}
  \sum_{\ell=1}^L \Nb_{\Delta}(\tPhi^{\ell-1}_{\Delta}(\x)) = \sum_{(\alpha:i, \beta:j) \in \Rb(\x)} \sum_{\ell=1}^L \ind{(\alpha:i, \beta:j) \in \Rb(\tPhi^{\ell-1}_{\Delta}(\x)), \tinter_{\alpha:i, \beta:j}(\tPhi^{\ell-1}_{\Delta}(\x)) \leq \Delta}.
\end{equation*}
Now for all $(\alpha:i, \beta:j) \in \Rb(\x)$, let $\ell_{\alpha:i, \beta:j}$ be the lowest index in $\{1, \ldots, L\}$ such that $(\alpha:i, \beta:j) \in \Rb(\tPhi^{\ell-1}_{\Delta}(\x))$ and $\tinter_{\alpha:i, \beta:j}(\tPhi^{\ell-1}_{\Delta}(\x)) \leq \Delta$. If there is no such index then the contribution of $(\alpha:i, \beta:j)$ in the sum above is null. Otherwise, the second part of the remark implies that after at most 
\begin{equation*}
  m := \lceil 2 \ConstBoundS/{\ConstUSH} \rceil
\end{equation*}
iterations, the pair $(\alpha:i, \beta:j)$ can no longer belong to $\Rb(\tPhi^{\ell_{\alpha:i, \beta:j}-1+m}_{\Delta}(\x))$. We deduce that
\begin{equation}\label{nbintercouplfix}
  \sum_{\ell=1}^L \ind{(\alpha:i, \beta:j) \in \Rb(\tPhi^{\ell-1}_{\Delta}(\x)), \tinter_{\alpha:i, \beta:j}(\tPhi^{\ell-1}_{\Delta}(\x)) \leq \Delta} \leq m,
\end{equation}
and therefore conclude that
\begin{equation*}
  \sum_{\ell=1}^L \Nb_{\Delta}(\tPhi^{\ell-1}_{\Delta}(\x)) \leq  md(d-1)n^2/2.
\end{equation*}
Injecting this bound in~\eqref{eq:estim2}, we complete the proof of \eqref{errLspd} with the constant
\begin{equation}\label{eq:C}	
  C := d(d-1) \ConstStab_1\ConstLip \lceil 2 \ConstBoundS/{\ConstUSH} \rceil.
\end{equation}

Now, for $t\in[(L-1)\Delta,L\Delta]$, by reasoning like in the derivation of \eqref{eq:estimstab}, one obtains
\begin{align*}\|\Phi(\x;t) - \tPhi_{\Delta}(\x;t)\|_1\leq& \sum_{\ell=1}^{L-1} \|\Phi_{t-\ell\Delta}(\Phi_{\Delta}-\tPhi_{\Delta})\tPhi^{\ell-1}_{\Delta}(\x)\|_1\\&+\|(\Phi_{t-(L-1)\Delta}-\tPhi_{t-(L-1)\Delta})\tPhi^{L-1}_{\Delta}(\x)\|_1\\\leq &\ConstStab_1\sum_{\ell=1}^{L-1} \|(\Phi_{\Delta}-\tPhi_{\Delta})\tPhi^{\ell-1}_{\Delta}(\x)\|_1\\&+\|(\Phi_{t-(L-1)\Delta}-\tPhi_{t-(L-1)\Delta})\tPhi^{L-1}_{\Delta}(\x)\|_1.
  \end{align*}
Since by Lemma \ref{lem:onestep}, 
\begin{align*}
   \|(\Phi_{t-(L-1)\Delta}-\tPhi_{t-(L-1)\Delta})\tPhi^{L-1}_{\Delta}(\x)\|_1&\leq 2(t-(L-1)\Delta)\frac{\ConstLip}{n^2} \sum_{(\alpha:i, \beta:j) \in \Rb(\y)} \ind{\tinter_{\alpha:i, \beta:j}(\y) \leq t-(L-1)\Delta }\\
&\leq 2\Delta\frac{\ConstLip}{n^2} \sum_{(\alpha:i, \beta:j) \in \Rb(\y)} \ind{\tinter_{\alpha:i, \beta:j}(\y) \leq \Delta},
\end{align*}
the above derived upper-bound for $\|\Phi^L_{\Delta}(\x) - \tPhi^L_{\Delta}(\x)\|_1$ still holds for all $t\in[(L-1)\Delta,L\Delta]$.


\subsection{Rate of convergence of the MSPD and its TSPD approximation with step \texorpdfstring{$\Delta$}{Delta}} The purpose of this subsection is to estimate the $\Ls^1$ distance between the semigroup solution $\bu=(u^1,\ldots,u^d)$ to the diagonal hyperbolic system~\eqref{eq:syst} with initial data $\bu_0=(H*m^1,\ldots,H*m^d)$ for some $\bm=(m^1,\ldots,m^d) \in \Ps(\R)^d$, and the empirical CDF $\bu_n[\x(n)]=(u_n^1[\x(n)],\ldots,u_n^d[\x(n)])$ associated with the MSPD started at some configuration $\x(n) \in D_n^d$, over finite time horizons. 

We also compare $\bu$ with the empirical CDF $\tilde{\bu}_{n,\Delta}[\x(n)]=(\tilde{u}_{n,\Delta}^1[\x(n)],\ldots,\tilde{u}_{n,\Delta}[\x(n)])$ obtained with the iterated TSPD with step $\Delta$ starting from $\x(n)$, that is
\begin{equation*}
  \tilde{u}_{n,\Delta}^{\gamma}[\x](t,x) := \frac{1}{n} \sum_{k=1}^n \ind{\tPhi_{\Delta,k}^{\gamma}(\x;t) \leq x}.
\end{equation*}

Given $\x \in D_n^d$, we set $\mu_0[\x]=(\mu^1_0[\x],\ldots,\mu^d_0[\x]) \in \Ps(\R)^d$.

\begin{theo}\label{theo:rateMSPD}
  Under Assumptions~\eqref{ass:LC:2} and~\eqref{ass:USH}, let $n \geq 1$, $\x(n) \in D_n^d$ and $\Delta >0$. Then, for all $t \geq 0$,
  \begin{equation*}
    \begin{aligned}
      & \|\bu(t,\cdot) - \bu_{n}[\x(n)](t,\cdot)\|_{\Ls^1(\R)^d} \leq {\mathcal L}_1\left(\Ws_1^{(d)}(\bm,\mu_0[\x(n)])+  \frac{t d\ConstLip}{n}\right),\\
      & \|\bu(t,\cdot) - \tilde{\bu}_{n,\Delta}[\x(n)](t,\cdot)\|_{\Ls^1(\R)^d}\leq {\mathcal L}_1\left(\Ws_1^{(d)}(\bm,\mu_0[\x(n)])+  \frac{td \ConstLip}{n}\right)+C\Delta,
    \end{aligned}
  \end{equation*}
  where the constant $C$ is given in~\eqref{eq:C}.

  When for all $\gamma\in\{1,\ldots, d\}$, $m^\gamma([a^\gamma,b^\gamma])=1$ with $-\infty<a^\gamma<b^\gamma<+\infty$ and 
  \begin{equation}\label{eq:discr:opt:2}
    \forall \gamma:k \in \Part, \qquad x_k^\gamma(n)=(u_0^\gamma)^{-1}\left(\frac{2k-1}{2n}\right),
  \end{equation} 
  then 
  \begin{equation*}
    \begin{aligned}
      & \|\bu(t,\cdot) - \bu_{n}[\x(n)](t,\cdot)\|_{\Ls^1(\R)^d}\leq {\mathcal L}_1\frac{\sum_{\gamma=1}^d(b^\gamma-a^\gamma)+t 2d\ConstLip}{2n},\\
      & \|\bu(t,\cdot) - \tilde{\bu}_{n,\Delta}[\x(n)](t,\cdot)\|_{\Ls^1(\R)^d}\leq {\mathcal L}_1\frac{\sum_{\gamma=1}^d(b^\gamma-a^\gamma)+t 2d\ConstLip}{2n}+C\Delta.
    \end{aligned}
  \end{equation*}
\end{theo}
\begin{proof}
The estimations involving $\tilde{\bu}_{n,\Delta}$ are an immediate consequence of Proposition \ref{prop:schema} and the estimations involving $\bu_{n}[\x(n)]$. To prove those estimations, following the approach described in the introduction of the article, we first write, for all $t \geq 0$,
\begin{equation}\label{eq:decomp:syst}
  \begin{aligned}
    \|\bu(t,\cdot) - \bu_n[\x(n)](t,\cdot)\|_{\Ls^1(\R)^d} & \leq \|\bu(t,\cdot) - \bu_{\infty}[\x(n)](t,\cdot)\|_{\Ls^1(\R)^d}\\
    & \quad + \|\bu_{\infty}[\x(n)](t,\cdot) - \bu_n[\x(n)](t,\cdot)\|_{\Ls^1(\R)^d},
  \end{aligned}
\end{equation}
where $\bu_{\infty}[\x(n)]$ refers to semigroup solution to the diagonal hyperbolic system with initial condition $\bu_{\infty,0}[\x(n)] = (H*\mu^1_0[\x(n)],\ldots,H*\mu^d_0[\x(n)])$. The $\Ls^1$ stability property of Theorem~\ref{theo:cvMSPD} yields
\begin{equation*}
  \|\bu(t,\cdot) - \bu_{\infty}[\x(n)](t,\cdot)\|_{\Ls^1(\R)^d}\leq {\mathcal L}_1\|\bu_0 - \bu_{\infty,0}[\x(n)]\|_{\Ls^1(\R)^d} = {\mathcal L}_1\Ws^{(d)}_1(\bm,\mu_0[\x(n)]),
\end{equation*}
which, according to~\eqref{it:discr:3} in Lemma~\ref{lem:discr}, is smaller than $\sum_{\gamma=1}^d (b^{\gamma}-a^{\gamma})/(2n)$ if $m^{\gamma}([a^{\gamma},b^{\gamma}])=1$ for all $\gamma$, and $\x(n)$ is given by~\eqref{eq:discr:opt:2}. The following proposition allows to control the second term in the right-hand side of~\eqref{eq:decomp:syst}.
\end{proof}

\begin{prop}\label{prop:rateMSPD}
  Under Assumptions~\eqref{ass:LC:2} and~\eqref{ass:USH}, let $n \geq 1$ and $\x \in D^d_n$. For all $t \geq 0$,
  \begin{equation*}
    \|\bu_{\infty}[\x](t,\cdot) - \bu_n[\x](t,\cdot)\|_{\Ls^1(\R)^d} \leq {\mathcal L}_1\frac{t d\ConstLip}{n}.
  \end{equation*}
\end{prop}
\begin{proof}
  For all $n \geq 1$, let us extend the duplication operator introduced in the proof of Proposition~\ref{prop:rateSPD} by setting
  \begin{equation*}
    \left\{\begin{array}{ccc}
      D^d_n & \to & D^d_{2n}\\
      \x & \mapsto & \hat{\x}
    \end{array}\right.
  \end{equation*}
  where $\hat{x}^\gamma_{2k-1} = \hat{x}^\gamma_{2k} = x^\gamma_k$ for all $k \in \{1, \ldots, n\}$ and $\gamma\in\{1,\ldots,d\}$.
  Notice that 
  \begin{equation}\label{stabdup}
    \forall \x,\y\in D_n^d, \qquad \|\hat{\x}-\hat{\y}\|_1=\|\x-\y\|_1.
  \end{equation}
  
  Let $\x\in D_n^d$. Like in the proof of Proposition \ref{prop:rateSPD}, we are going to estimate $\|\bu_{2n}[\hat{\x}](t,\cdot)-\bu_{n}[\x](t,\cdot)\|_{\Ls^1(\R)^d}$. The direct analysis of the MSPD being delicate, we introduce a ficticious step $\Delta>0$ to transform this analysis into the comparison between the TSPD with $n$ particles and $2n$ duplicated particles on each time-step.
Let $t>0$ and $L=\lceil t/\Delta\rceil$. One has 
  \begin{align*}
    \Phi(\hat{\x};t)-\widehat{\Phi(\x;t)}=\sum_{\ell=1}^{L-1}\Phi_{t-\ell\Delta}(\Phi_\Delta\widehat{\Phi^{\ell-1}_\Delta(\x)}-\widehat{\Phi^\ell_\Delta(\x)})+\Phi_{t-(L-1)\Delta}\widehat{\Phi^{L-1}_\Delta(\x)}-\widehat{\Phi_{t-(L-1)\Delta}}\Phi^{L-1}_\Delta(\x).
  \end{align*}
  Hence, by the triangle inequality and the stability of the MSPD,
  \begin{equation}\label{decomper}
    \begin{aligned}
      \|\bu_{2n}[\hat{\x}](t,\cdot)-\bu_{n}[\x](t,\cdot)\|_{\Ls^1(\R)^d} & =\|\Phi(\hat{\x};t)-\widehat{\Phi(\x;t)}\|_1\\
      & \leq {\mathcal L}_1\sum_{\ell=1}^{L-1}\|\Phi_\Delta\widehat{\Phi^{\ell-1}_\Delta(\x)}-\widehat{\Phi^\ell_\Delta(\x)}\|_1\\
      & \quad +\|\Phi_{t-(L-1)\Delta}\widehat{\Phi^{L-1}_\Delta(\x)}-\widehat{\Phi_{t-(L-1)\Delta}}\Phi^{L-1}_\Delta(\x)\|_1.
    \end{aligned}
  \end{equation}
  Now, by the triangle inequality and \eqref{stabdup}
  \begin{equation*}
    \begin{aligned}
      \|\Phi_\Delta\widehat{\Phi^{\ell-1}_\Delta(\x)}-\widehat{\Phi^\ell_\Delta(\x)}\|_1 & \leq \|\Phi_\Delta\widehat{\Phi^{\ell-1}_\Delta(\x)}-\tPhi_\Delta\widehat{\Phi^{\ell-1}_\Delta(\x)}\|_1\\
      & \quad +\|\tPhi_\Delta\widehat{\Phi^{\ell-1}_\Delta(\x)}-\widehat{\tPhi_\Delta}\Phi^{\ell-1}_\Delta(\x)\|_1\\
      & \quad +\|\tPhi_\Delta\Phi^{\ell-1}_\Delta(\x)-\Phi^\ell_\Delta(\x)\|_1.
    \end{aligned}
  \end{equation*}
  The second term in the right-hand side is a comparison at time $\Delta$ between the TSPD with $n$ particles starting from $\Phi^{\ell-1}_\Delta(\x)$ and the TSPD with $2n$ particles starting from the duplicated vector $\widehat{\Phi^{\ell-1}_\Delta(\x)}$. Reasoning like in the derivation of \eqref{eq:controln2n:scal}, we bound it from above by $\Delta d\ConstLip/(2n)$. By Lemma \ref{lem:onestep}, 
  \begin{equation*}
    \begin{aligned}
      & \|\Phi_\Delta\widehat{\Phi^{\ell-1}_\Delta(\x)}-\tPhi_\Delta\widehat{\Phi^{\ell-1}_\Delta(\x)}\|_1+\|\tPhi_\Delta\Phi^{\ell-1}_\Delta(\x)-\Phi^\ell_\Delta(\x)\|_1\\
      & \qquad \leq \frac{\Delta\ConstLip}{2n^2}\left(\Nb_{\Delta}(\widehat{\Phi^{\ell-1}_\Delta(\x)})+4 \Nb_{\Delta}(\Phi^{\ell-1}_\Delta(\x))\right),
    \end{aligned}
  \end{equation*}
  where $\Nb_{\Delta}$ has been defined in~\eqref{defndel}. Plugging these estimations in \eqref{decomper} and dealing in the same way with the last term in the right-hand side of \eqref{decomper}, we deduce that
  \begin{align*}
    \|\bu_{2n}[\hat{\x}](t,\cdot)-\bu_{n}[\x](t,\cdot)\|_{\Ls^1(\R)^d}\leq \frac{{\mathcal L}_1t d\ConstLip}{2n}+\frac{\Delta{\mathcal L}_1\ConstLip}{2n^2}\sum_{\ell=1}^L\left(
  \Nb_{\Delta}(\widehat{\Phi^{\ell-1}_\Delta(\x)})+4 \Nb_{\Delta}(\Phi^{\ell-1}_\Delta(\x))\right).
  \end{align*}
  Clearly, for all $\alpha:i, \beta:j \in \Part$ with $\alpha < \beta$, 
  \begin{align*}
    \sum_{\ell=1}^L \ind{(\alpha:i, \beta:j) \in \Rb(\Phi^{\ell-1}_{\Delta}(\x)), \tinter_{\alpha:i, \beta:j}(\Phi^{\ell-1}_{\Delta}(\x)) \leq \Delta} \leq 1.
  \end{align*}
  We now use the same arguments as in the derivation of~\eqref{nbintercouplfix} to obtain that, for all $\alpha:i, \beta:j \in P_{2n}^d$ with $\alpha < \beta$,
  \begin{equation}\label{eq:nbintercouplfix:2}
    \sum_{\ell=1}^L \ind{(\alpha:i, \beta:j) \in \Rb(\widehat{\Phi^{\ell-1}_\Delta(\x)}), \tinter_{\alpha:i, \beta:j}(\widehat{\Phi^{\ell-1}_\Delta(\x)}) \leq \Delta} \leq \lceil 2 \ConstBoundS/{\ConstUSH} \rceil.
  \end{equation}
  To this aim we only have to replace Remark~\ref{rk:csqUSH} with the following: let $\y \in D_n^d$ and $\alpha:i, \beta:j \in \Part$ such that $\alpha < \beta$, and denote $\z := \widehat{\Phi_{\Delta}(\y)}$.
  \begin{itemize}
    \item If $(\alpha:i, \beta:j) \not\in \Rb(\y)$, then for all $(i',j') \in \{2i-1, 2i\} \times \{2j-1,2j\}$, $(\alpha:i', \beta:j') \not\in \Rb(\z)$.
    \item For all $(i',j') \in \{2i-1, 2i\} \times \{2j-1,2j\}$, $z^{\beta}_{j'} - z^{\alpha}_{i'} \leq  y^{\beta}_j - y^{\alpha}_i - \ConstUSH\Delta$ and, if $(\alpha:i, \beta:j) \in \Rb(\y)$ and $\tinter_{\alpha:i, \beta:j}(\y) \leq \Delta$, then $0 < y^{\beta}_j - y^{\alpha}_i \leq 2 \ConstBoundS \Delta$.
  \end{itemize}
  The remainder of the argument is the same and leads to~\eqref{eq:nbintercouplfix:2}. As a consequence, 
  \begin{equation*}
    \sum_{\ell=1}^L(\Nb_{\Delta}(\widehat{\Phi^{\ell-1}_\Delta(\x)})+4 \Nb_{\Delta}(\Phi^{\ell-1}_\Delta(\x)))\leq 2d(d-1)n^2(1+\lceil 2 \ConstBoundS/{\ConstUSH} \rceil),
  \end{equation*}
  and
  \begin{equation*}
    \|\bu_{2n}[\hat{\x}](t,\cdot)-\bu_{n}[\x](t,\cdot)\|_{\Ls^1(\R)^d}\leq \frac{{\mathcal L}_1t d\ConstLip}{2n}+\Delta d(d-1){\mathcal L}_1\ConstLip (1+\lceil 2 \ConstBoundS/{\ConstUSH} \rceil).
  \end{equation*}
  Taking the limit $\Delta\to0$, we deduce that
  \begin{equation*}
    \|\bu_{2n}[\hat{\x}](t,\cdot)-\bu_{n}[\x](t,\cdot)\|_{\Ls^1(\R)^d}\leq \frac{{\mathcal L}_1t d\ConstLip}{2n}.
  \end{equation*}
  With this estimation replacing~\eqref{eq:controln2n:scal}, we shall conclude like in the proof of Proposition~\ref{prop:rateSPD}. We therefore need to show that
  \begin{equation*}
    \lim_{M \to +\infty} \|\bu_{2^Mn}[\hat{\x}^M](t,\cdot) - \bu_{\infty}[\x](t,\cdot)\|_{\Ls^1(\R)^d} = 0.
  \end{equation*}
  To this aim, we denote $\bm_n := \mu_0[\x]$ and recall the definition~\eqref{eq:chin} of the discretisation operator to write
  \begin{equation*}
    \begin{aligned}
      \|\bu_{2^Mn}[\hat{\x}^M](t,\cdot) - \bu_{\infty}[\x](t,\cdot)\|_{\Ls^1(\R)^d} & \leq \|\bu_{2^Mn}[\hat{\x}^M](t,\cdot) - \bu_{2^Mn}[\chi_{2^Mn}\bm_n](t,\cdot)\|_{\Ls^1(\R)^d}\\
      & \quad + \|\bu_{2^Mn}[\chi_{2^Mn}\bm_n](t,\cdot) - \bu_{\infty}[\x](t,\cdot)\|_{\Ls^1(\R)^d}.
    \end{aligned}
  \end{equation*}
  Using the same arguments as in the proof of Proposition~\ref{prop:rateSPD} but with Theorem~\ref{theo:cvMSPD} replacing Theorem~\ref{theo:cvSPD}, in particular the finite speed of propagation for both the MSPD and $\bu_{\infty}[\x]$, we obtain that the second term in the right-hand side vanishes. On the other hand, the discrete stability estimate of Proposition~\ref{prop:stabMSPD} yields
  \begin{equation*}
    \begin{aligned}
      \|\bu_{2^Mn}[\hat{\x}^M](t,\cdot) - \bu_{2^Mn}[\chi_{2^Mn}\bm_n](t,\cdot)\|_{\Ls^1(\R)^d} & = \|\Phi(\hat{\x}^M;t) - \Phi(\chi_{2^Mn}\bm_n;t)\|_1\\
      & \leq \ConstStab_1 \|\hat{\x}^M - \chi_{2^Mn}\bm_n\|_1\\
      & = \ConstStab_1 \Ws^{(d)}_1(\bm_n, \mu_0[\chi_{2^Mn}\bm_n]).
    \end{aligned}
  \end{equation*}
  It is a property of the discretisation operator that, for all $\gamma \in \{1, \ldots, d\}$, $\mu^{\gamma}_0[\chi_{2^Mn}\bm_n]$ converges weakly to $m^{\gamma}_n = \mu^{\gamma}_0[\x]$ when $M \to +\infty$~\cite[Lemma~8.1.5]{jr}. As a consequence, the corresponding empirical CDF converge $\dd x$-almost everywhere. Since the support of each measure $\mu^{\gamma}_0[\chi_{2^Mn}\bm_n]$ is contained in the compact set $[x^{\gamma}_1, x^{\gamma}_n]$, Lebesgue's theorem implies the convergence in $\Ls^1(\R)$ of these CDFs, which implies
  \begin{equation*}
    \lim_{M \to +\infty} \Ws^{(d)}_1(\bm_n, \mu_0[\chi_{2^Mn}\bm_n]) = 0
  \end{equation*}
  by~\eqref{eq:W1d}, and thereby completes the proof.
\end{proof}

As an immediate corollary of Theorem~\ref{theo:rateMSPD}, we obtain a convergence result for the MSPD to the semigroup solution of the system~\eqref{eq:syst} which holds under less stringent conditions on the sequence of initial configurations.
\begin{cor}\label{cor:cvMSPD}
  Under Assumptions~\eqref{ass:LC:2} and~\eqref{ass:USH}, let $\bm \in \Ps(\R)^d$ and let $(\x(n))_{n \geq 1}$ be such that for all $n \geq 1$, $\x(n) \in D_n^d$ and
  \begin{equation}\label{eq:condcvMSPD}
    \lim_{n \to +\infty} \Ws^{(d)}_1(\mu_0[\x(n)], \bm) = 0.
  \end{equation}
  For all $t \geq 0$, the empirical CDF $\bu_n[\x(n)](t,\cdot)$ converges in $\Ls^1(\R)^d$ to the semigroup solution $\bu(t, \cdot)$ of the system~\eqref{eq:syst} with initial data $\bu_0 = (H*m^1, \ldots, H*m^d)$.
\end{cor}
Under the condition~\eqref{eq:condcvMSPD}, this corollary extends both~\cite[Theorem~2.4.5]{jr} and~\cite[Theorem~2.6.5]{jr}: in the former it allows to identify the limit, in the latter it relaxes the assumption that $\x(n) = \chi_n\bm$. We however underline that a necessary condition for~\eqref{eq:condcvMSPD} to hold is that, for all $\gamma \in \{1, \ldots, d\}$, $m^{\gamma}$ have a finite first-order moment.


\section{Numerical implementation}\label{s:num}


\subsection{Scalar conservation laws}\label{ss:num:spd} This subsection is dedicated to the numerical implementation of the SPD in order to approximate the entropy solution to the scalar conservation law~\eqref{eq:scl}. The algorithm we use is described in~\S\ref{sss:BGalgo}. Then two case studies are presented:~\S\ref{sss:Burgers} addresses the Burgers equation
\begin{equation}\label{eq:Burgers}
  \partial_t u + \partial_x\left(\frac{u^2}{2}\right) = 0,
\end{equation}
with the CDF of the two-atom measure
\begin{equation}\label{eq:twoatom}
  m = \frac{1}{2}\left(\delta_{-1} + \delta_1\right)
\end{equation}
as initial datum, while~\S\ref{sss:concave} deals with the conservation law with concave flux function
\begin{equation}\label{eq:concave}
  \partial_t u + \partial_x\left(\frac{u(1-u)}{2}\right) = 0,
\end{equation}
with the CDF of the two-sided exponential measure (or Laplace distribution)
\begin{equation}\label{eq:biexpo}
  m(\dd x) = \frac{1}{2}\exp(-|x|)\dd x
\end{equation}
as initial datum.

\subsubsection{Numerical computation of the SPD}\label{sss:BGalgo} Given $t \geq 0$, a vector of (ranked) initial positions $\rx = (x_1, \ldots, x_n) \in D_n$, and a vector $\rblambda = (\rblambda_1, \ldots, \rblambda_n) \in \R^n$ of initial velocities, we use a remark due to Brenier and Grenier~\cite[Section~4]{bregre} to devise an algorithm computing the vector $\phi[\rblambda](\rx;t)$ of positions at time $t$ in $\grandO(n)$ elementary operations, without following the detailed trajectory of each particle on the time interval $[0,t]$. 

Let us define the free transport flow $\psi[\rblambda](\rx;t) \in \R^n$ by, for all $k \in \{1, \ldots, n\}$,
\begin{equation*}
  \psi_k[\rblambda](\rx;t) := x_k + t \rblambda_k,
\end{equation*}
and introduce the functions $P_t$ and $Q_t$ on $[0,1]$ such that $P_t(0)=Q_t(0)=0$ and, for all $k \in \{1, \ldots, n\}$,
\begin{equation*}
  P_t(k/n) := \sum_{k'=1}^k \phi_{k'}[\rblambda](\rx;t), \qquad Q_t(k/n) := \sum_{k'=1}^k \psi_{k'}[\rblambda](\rx;t), 
\end{equation*}
with linear interpolation on $[(k-1)/n, k/n]$. Brenier and Grenier pointed out that $P_t$ is the convex hull of $Q_t$. Thus, our algorithm consists in computing the vector $\{Q_t(k/n), k = 1, \ldots, n\}$ first, which obviously requires $\grandO(n)$ operations, and then deducing $P_t$ from the algorithm described by the pseudo-code below, which follows Andrew's monotone chain algorithm~\cite{Andr79}, based on the Graham scan~\cite{Grah72}. \revision{Note that the parallelisation of the computation of the vector $\{Q_t(k/n), k = 1, \ldots, n\}$ is straightforward, while a parallelisable method to determine the convex hull of a sorted point set in the plane was devised in~\cite{Goo87}.}

The input is the vector {\tt Q} of size {\tt n+1}, with elements {\tt Q(k)} = $Q_t(k/n)$ indexed by $k \in \{0, \ldots, n\}$. The algorithm constructs a list {\tt L} of integers {\tt k} (ranked in decreasing order) containing the successive points at which $P_t(k/n) = Q_t(k/n)$. Once {\tt L} is computed, $P_t$ is reconstructed by linear interpolation, which finally yields $\phi[\rblambda](\rx;t)$. The length of the list {\tt L} is denoted by {\tt |L|} and its elements are indexed starting from {\tt 1}. It is assumed that its first and second elements can be accessed and removed in constant time.
\begin{verbatim}
  initialise L = [1,0]
  for k = 2 to n
    while |L| > 1 and (Q(k)-Q(L(1)))/(k-L(1)) < (Q(L(1))-Q(L(2)))/(L(1)-L(2))
      remove the first element from L
    end while
    insert k at the beginning of L
  end for  
\end{verbatim}
It is easily checked, by induction on $k \in \{1, \ldots, n\}$, that after the $(k-1)$-th iteration of the {\tt for} loop, {\tt L} contains the indices of the points at which $Q_t$ coincides with the convex hull of $\{Q_t(k'/n), k'=0, \ldots, k\}$. Indeed, the {\tt while} loop consists in removing from {\tt L} the points {\tt L(1)} at which the piecewise linear function interpolating between the values of $Q_t$ at the points {\tt L(2)}, {\tt L(1)} and {\tt k} is concave, see Figure~\ref{fig:algoL}.
\begin{figure}[ht]
  \includegraphics[width=3.5cm]{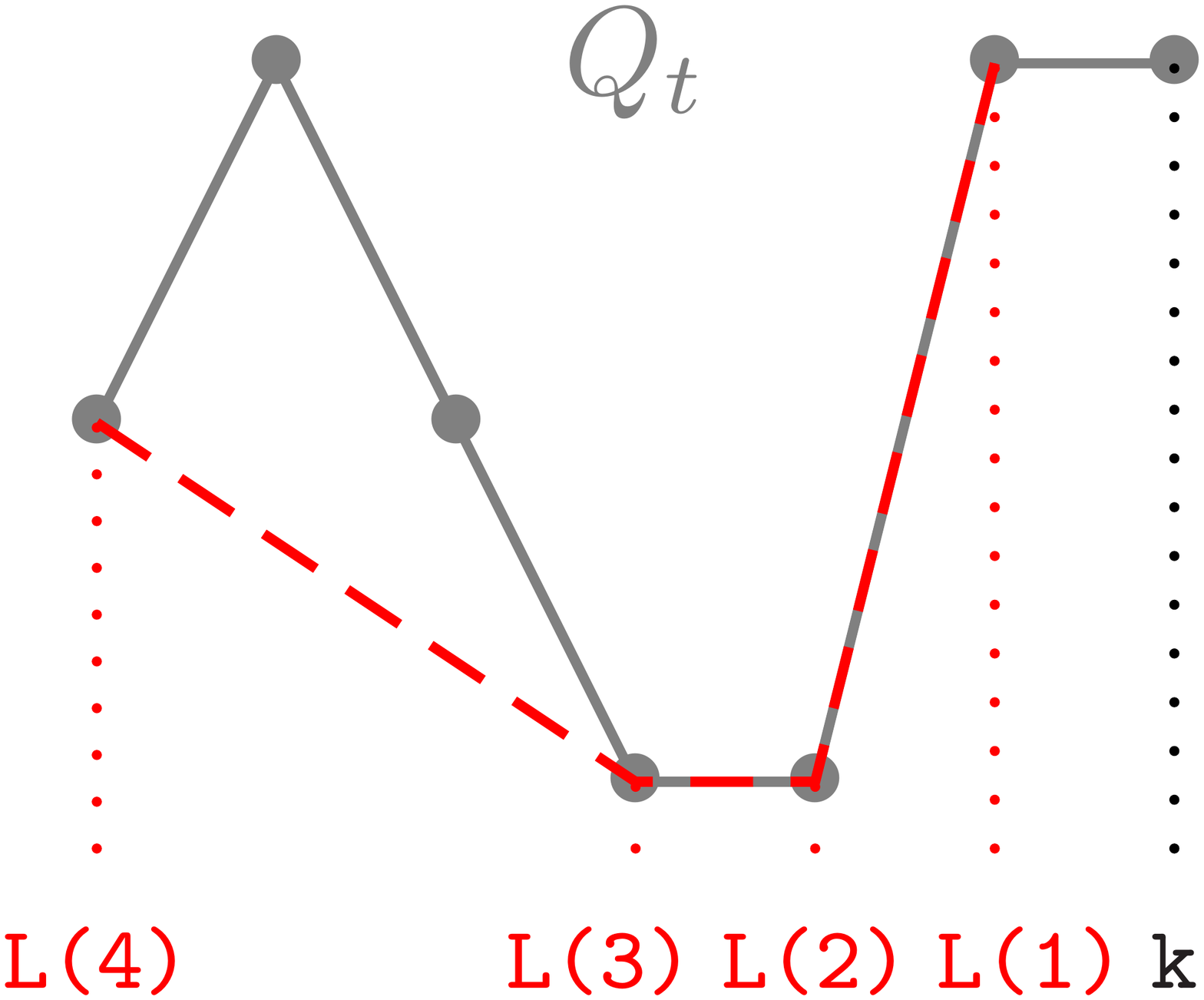}\hskip 2cm\includegraphics[width=3.5cm]{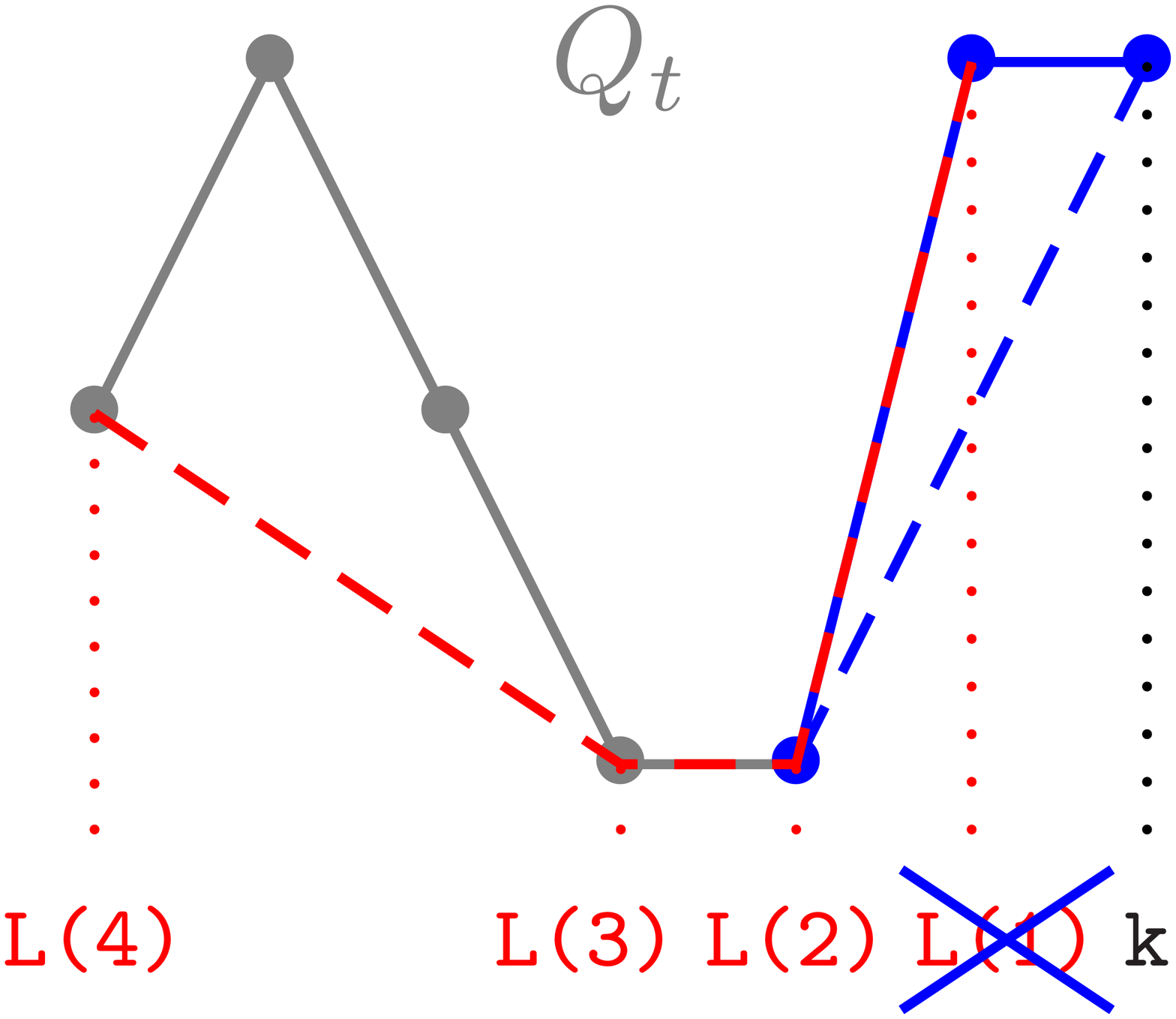}\hskip 2cm\includegraphics[width=3.5cm]{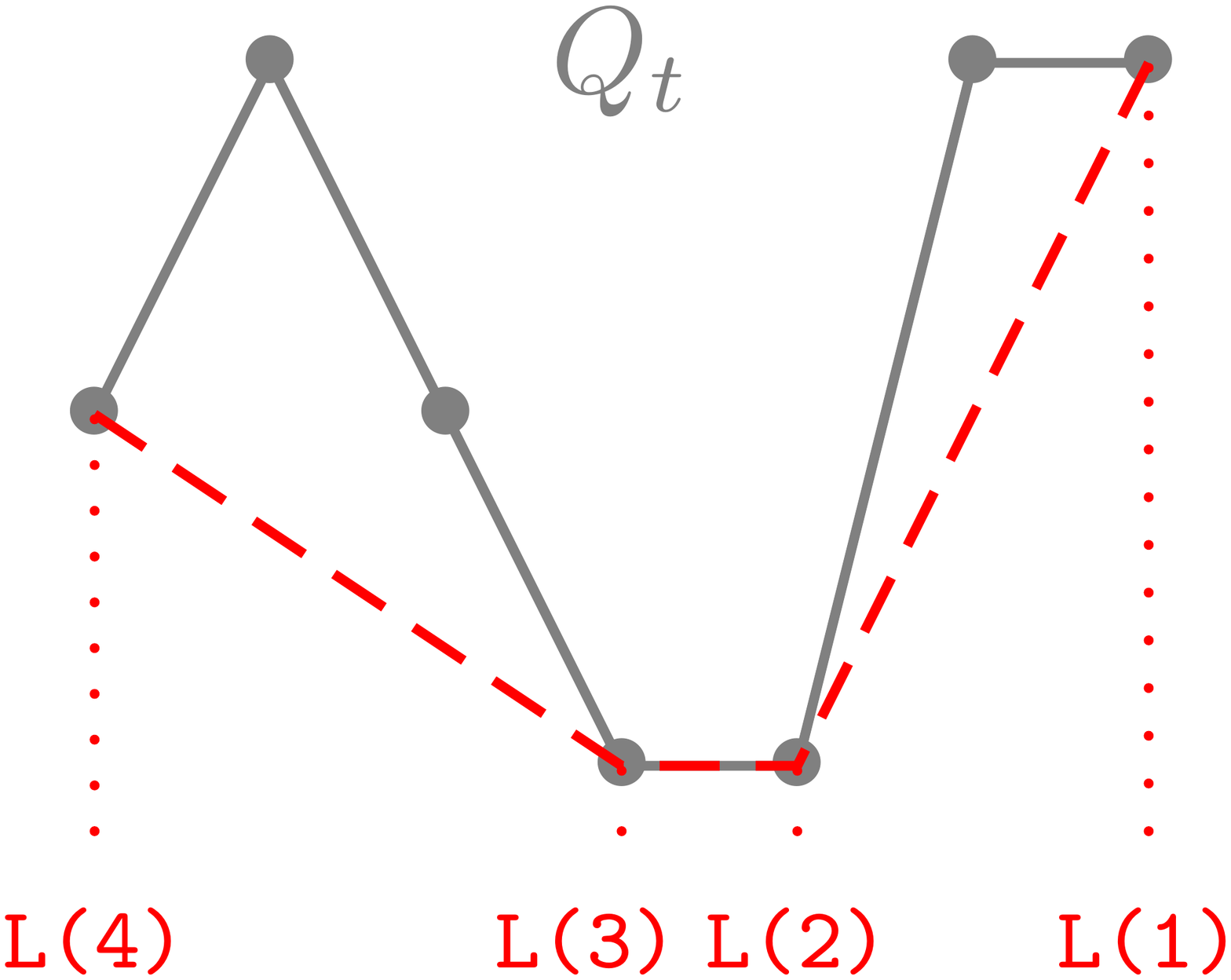}
  \caption{An iteration of the algorithm computing {\tt L}. The left-hand figure displays the composition of {\tt L} at the beginning of the {\tt (k-1)}-th iteration. On the central figure, the point {\tt k-1} is removed from {\tt L}. The right-hand figure displays the composition of {\tt L} at the end of the {\tt k}-th iteration.}
  \label{fig:algoL}
\end{figure}

From the fact that the list {\tt L} is browsed at each iteration, one could think that the algorithm uses $\grandO(n^2)$ operations. This is actually not the case, as elements of {\tt L} for which the test in the {\tt while} loop returns {\tt true} are removed from {\tt L} and will therefore not appear again in the next iterations. As a consequence, the actual complexity of this algorithm is $\grandO(n)$.

Notice that an algorithm computing the explicit space-time points of collisions in the SPD would also take $\grandO(n)$ operations, as there are at most $n-1$ such points. As a consequence, the method presented here has the same computational efficiency as the explicit simulation of the SPD. However, the Brenier-Grenier trick allows for a significative simplification of the implementation. 

\subsubsection{Burgers equation with two-atom initial measure}\label{sss:Burgers} We consider the Burgers equation~\eqref{eq:Burgers} with the CDF of the two-atom measure~\eqref{eq:twoatom} as initial datum. In the SPD, two fans of particles are created, respectively originating from the points $-1$ and $1$, see Figure~\ref{fig:Burgers} (a). These fans correspond to the fact that the entropy solution is the superposition of two rarefaction waves, respectively located at time $t$ on $[-1,-1+t/2]$ and $[1+t/2,1+t]$, see Figures~\ref{fig:Burgers} (b) and (c). 

The $\Ls^1$ error between the particle approximation and the solution is plotted as a function of $n$, for several given terminal times $t$, on Figure~\ref{fig:Burgers} (d). In accordance with Proposition~\ref{prop:rateSPD} and Remark~\ref{rk:rateSPD}, and since there is no discretisation error of the initial condition here (for even $n$), it is observed that this error is of the order of magnitude $1/n$.

\begin{figure}[ht]
  \includegraphics[clip=true,trim=.5cm 0cm 1cm 1cm,width=7.5cm]{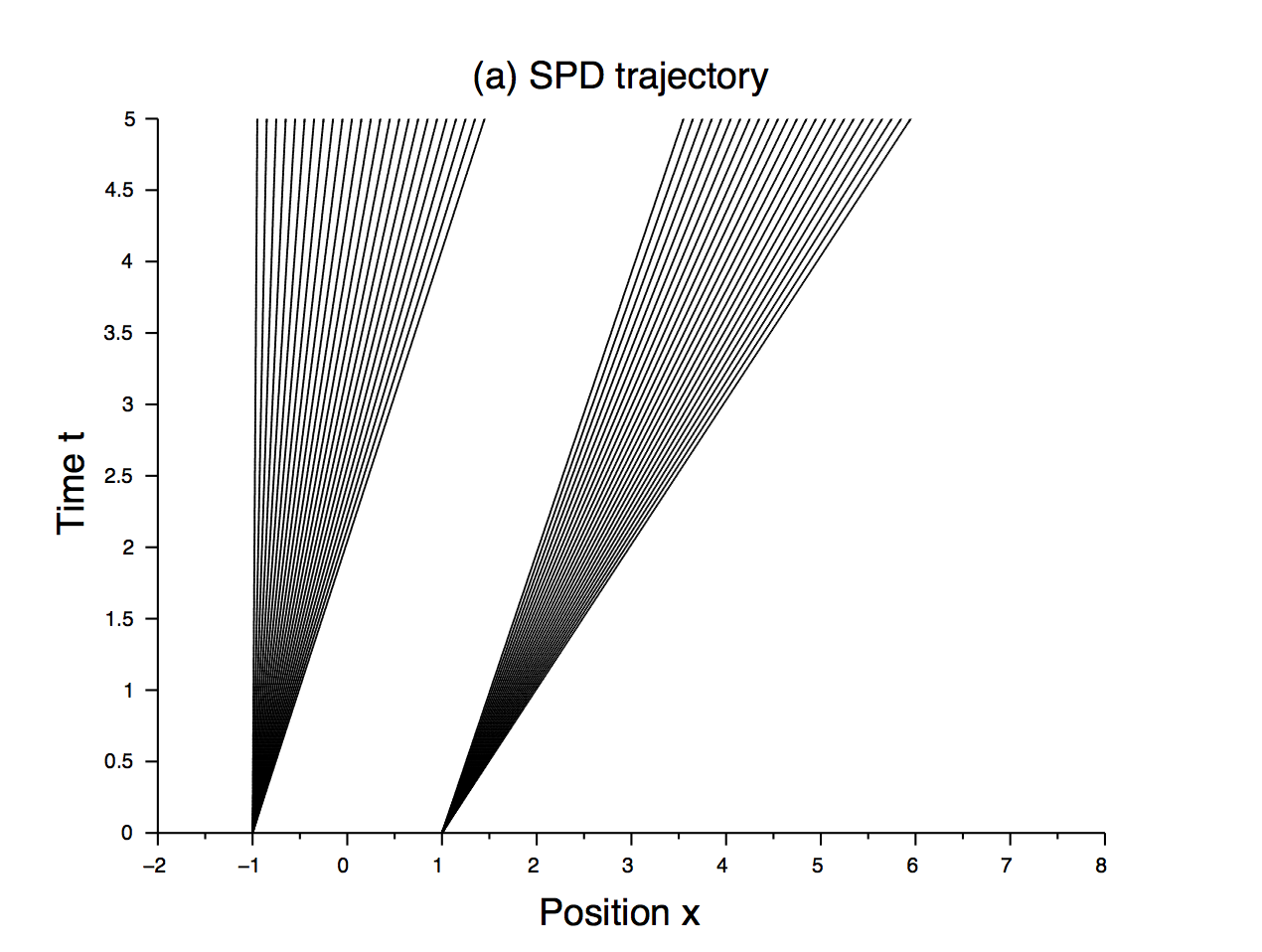}\includegraphics[clip=true,trim=.5cm 0cm 1cm 1cm,width=7.5cm]{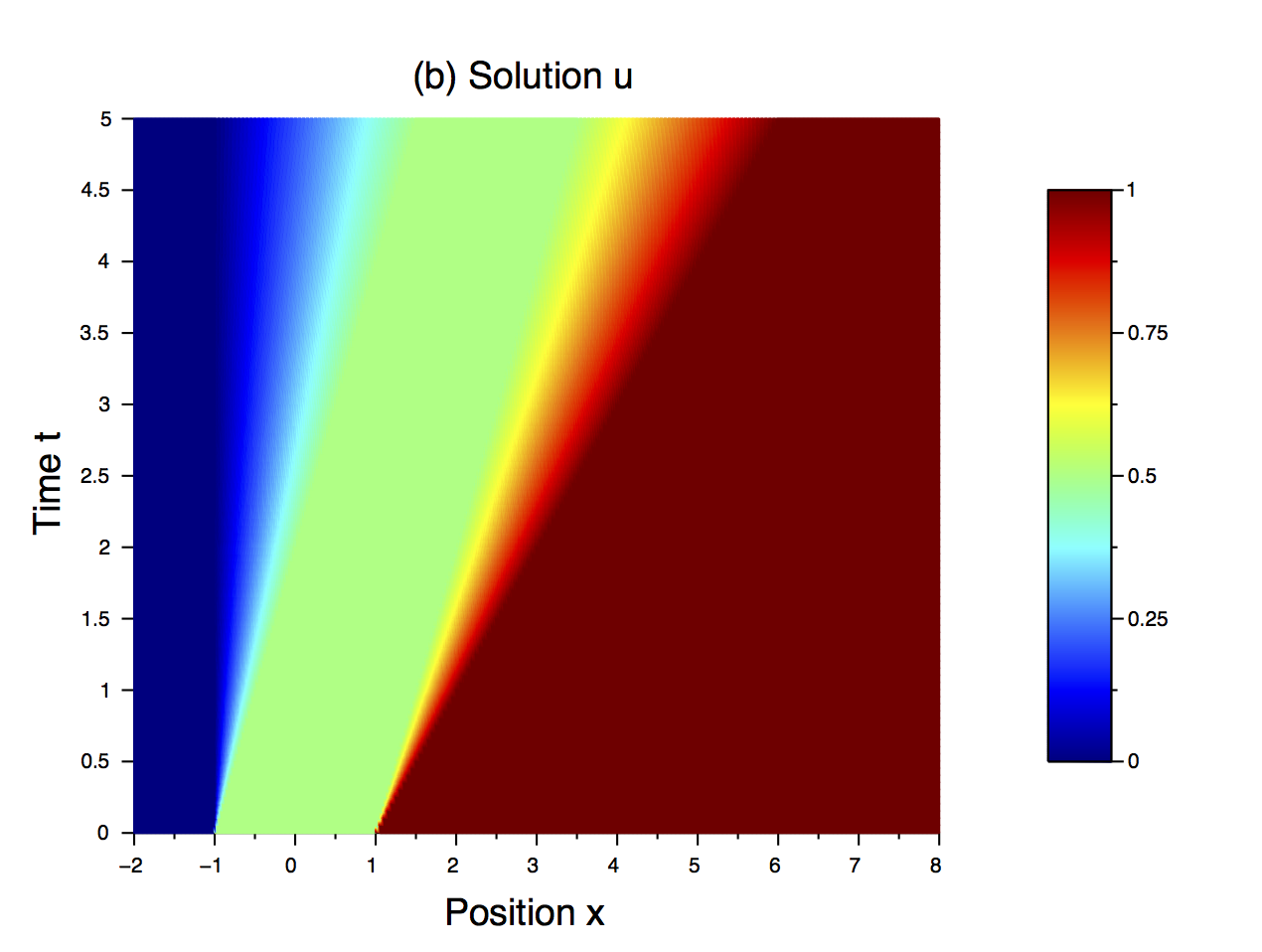}
  
  \includegraphics[clip=true,trim=.5cm 0cm 1cm 1cm,width=7.5cm]{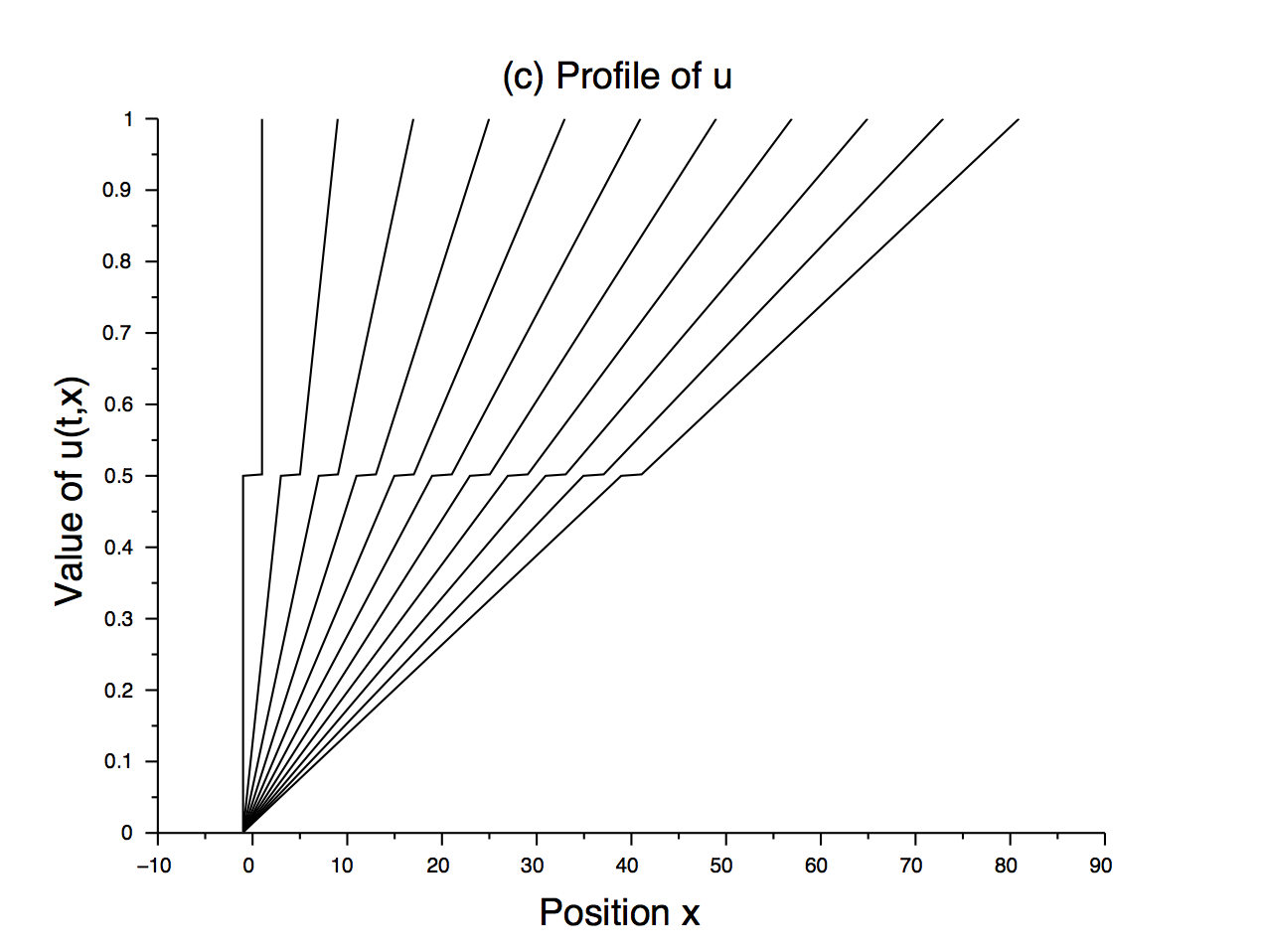}\includegraphics[clip=true,trim=.5cm 0cm 1cm 1cm,width=7.5cm]{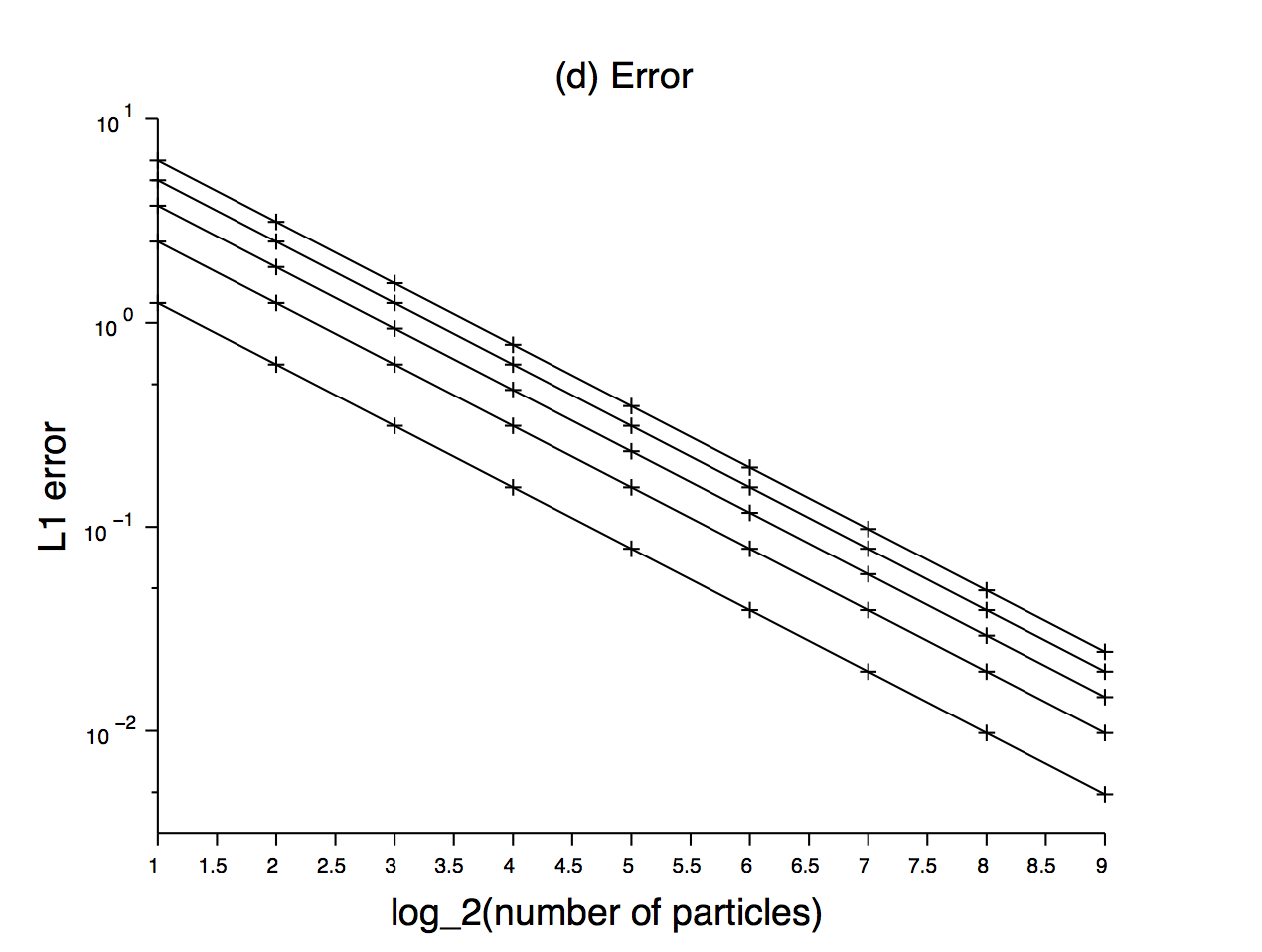}
  \caption{Numerical results for the Burgers equation with two-atom initial measure: (a) trajectories of 50 particles in the space-time plane, (b) value of the solution $u(t,x)$ in the space-time plane, (c) profile of the solution $u(t,x)$ at successive times $t=0, 8, 16, \ldots, 80$ and (d) logarithmic plot of the $\Ls^1$ error between the approximation obtained with $2^p$ particles and the solution, as a function of $p = 1, \ldots, 9$. The different lines correspond to different values of $t$, namely $t=10, 20, \ldots, 50$, with the higher curves corresponding to the larger times. The slope of each line is $- 0.693 \simeq -\log 2$, which expresses the order $1/n$ of the error.}
  \label{fig:Burgers}
\end{figure}

\subsubsection{Concave flux function}\label{sss:concave} We now consider the conservation law with concave flux function~\eqref{eq:concave} and the CDF of the two-sided exponential measure~\eqref{eq:biexpo} as initial datum. As is made clear on Figure~\ref{fig:concave} (a), the particles progressively aggregate at $0$. It results in the formation of a shock wave in the solution, see Figures~\ref{fig:concave} (b) and (c). The $\Ls^1$ error is displayed on Figure~\ref{fig:concave} (d) and exhibits the following behaviour: given $t \geq 0$, there exists a critical number of particles such that:
\begin{itemize}
  \item below this number, the error does not vary with $n$,
  \item above this number, the error decreases when $n$ increases at the same rate as for the discretisation of the initial measure.
\end{itemize}
This behaviour is explained by the fact that, for $n$ small, all the particles have arrived at $0$ at time $t$, so that the approximate solution is the Heaviside function whatever $n$. But as soon as $n$ is large enough to allow some particles to have an initial position far enough from $0$ so that they have not reached $0$ at time $t$ yet, then the contribution of these particles in the approximate solution allows the latter to fit better the part of $u$ outside of the shock wave, at the same rate as for the initial discretisation since the shape of $u$ outside of the shock wave is merely an affine transformation of the initial profile. Following the conclusions of Section~\ref{s:ci}, this discretisation error is of the order $\log(n)/n$.

Of course, the larger $t$ is, the larger the magnitude of the shock wave is, therefore the better the Heaviside function approximates $u$ and the more particles it takes to reach the critical number, which explains the ordering of the different curves on the picture.

\begin{figure}[ht]
  \includegraphics[clip=true,trim=.5cm 0cm 1cm 1cm,width=7.5cm]{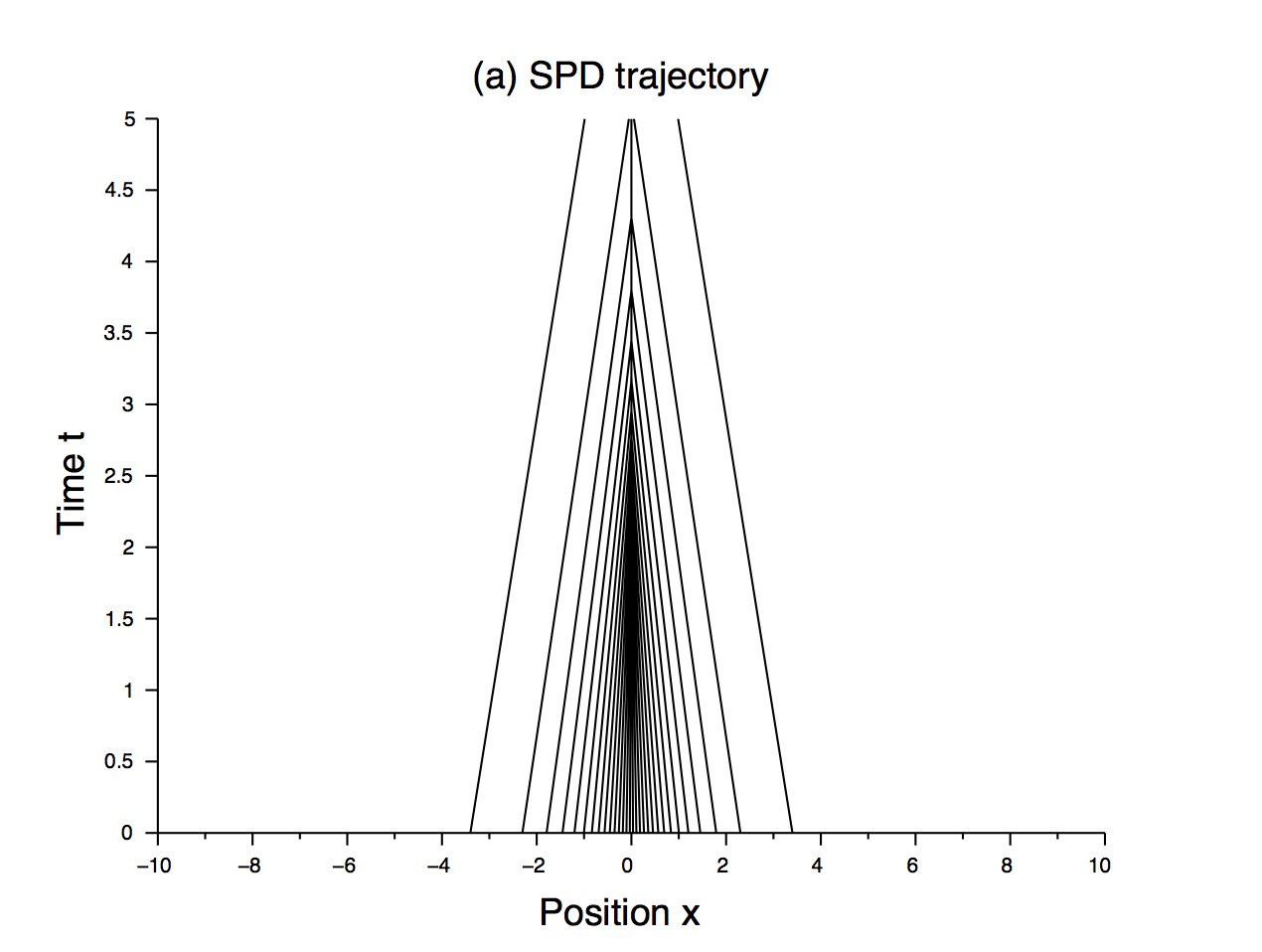}\includegraphics[clip=true,trim=.5cm 0cm 1cm 1cm,width=7.5cm]{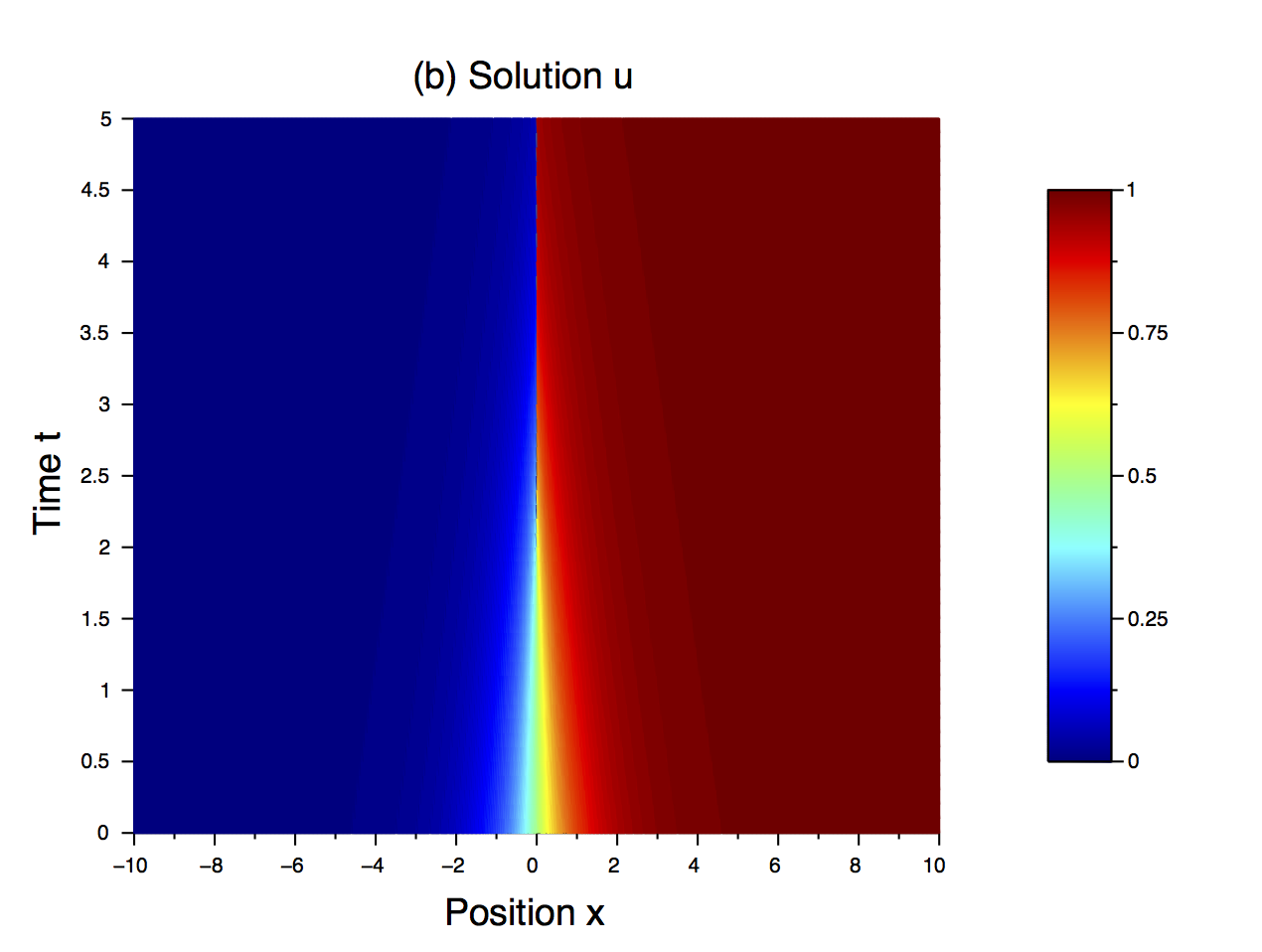}
  
  \includegraphics[clip=true,trim=.5cm 0cm 1cm 1cm,width=7.5cm]{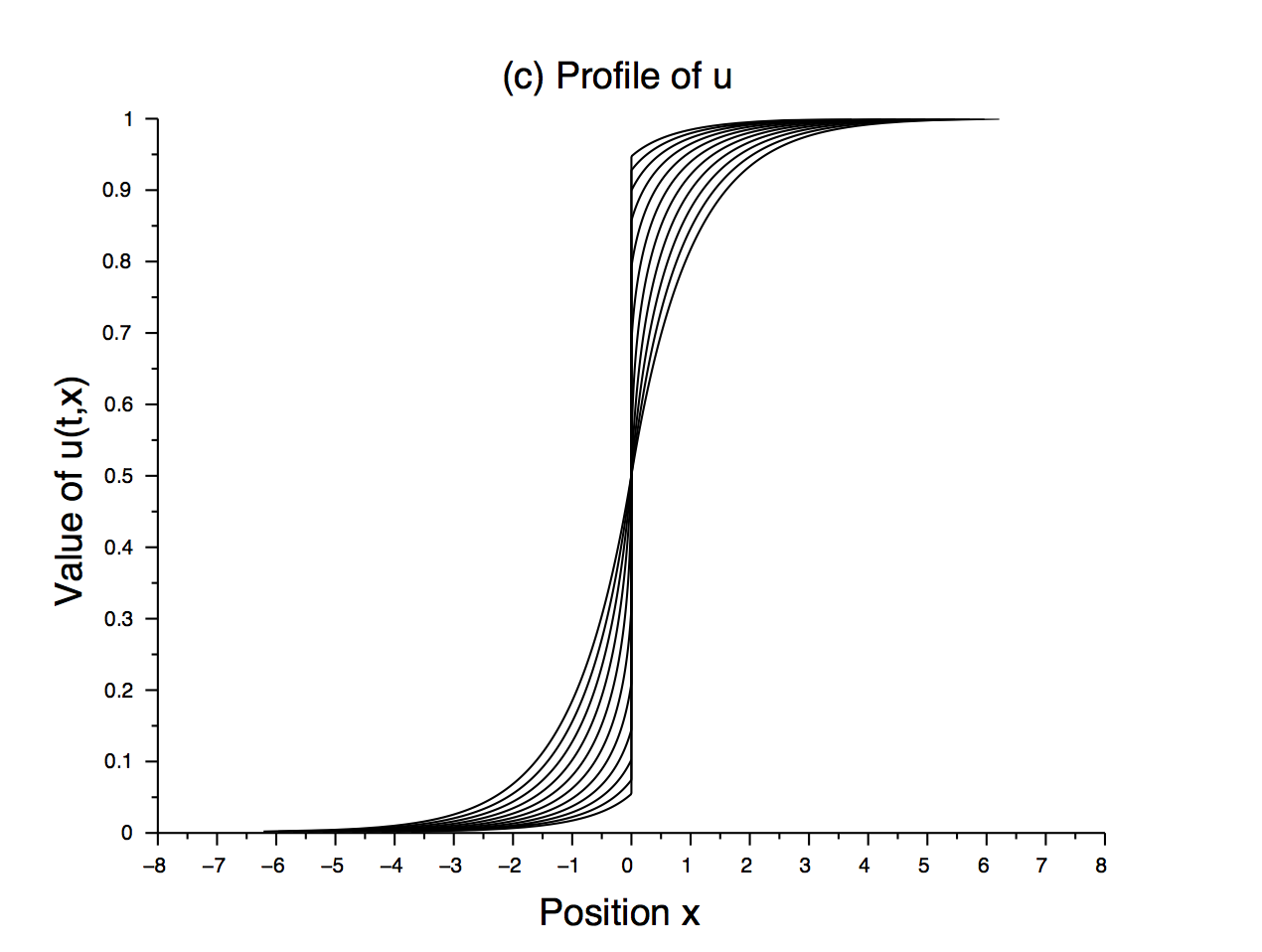}\includegraphics[clip=true,trim=.5cm 0cm 1cm 1cm,width=7.5cm]{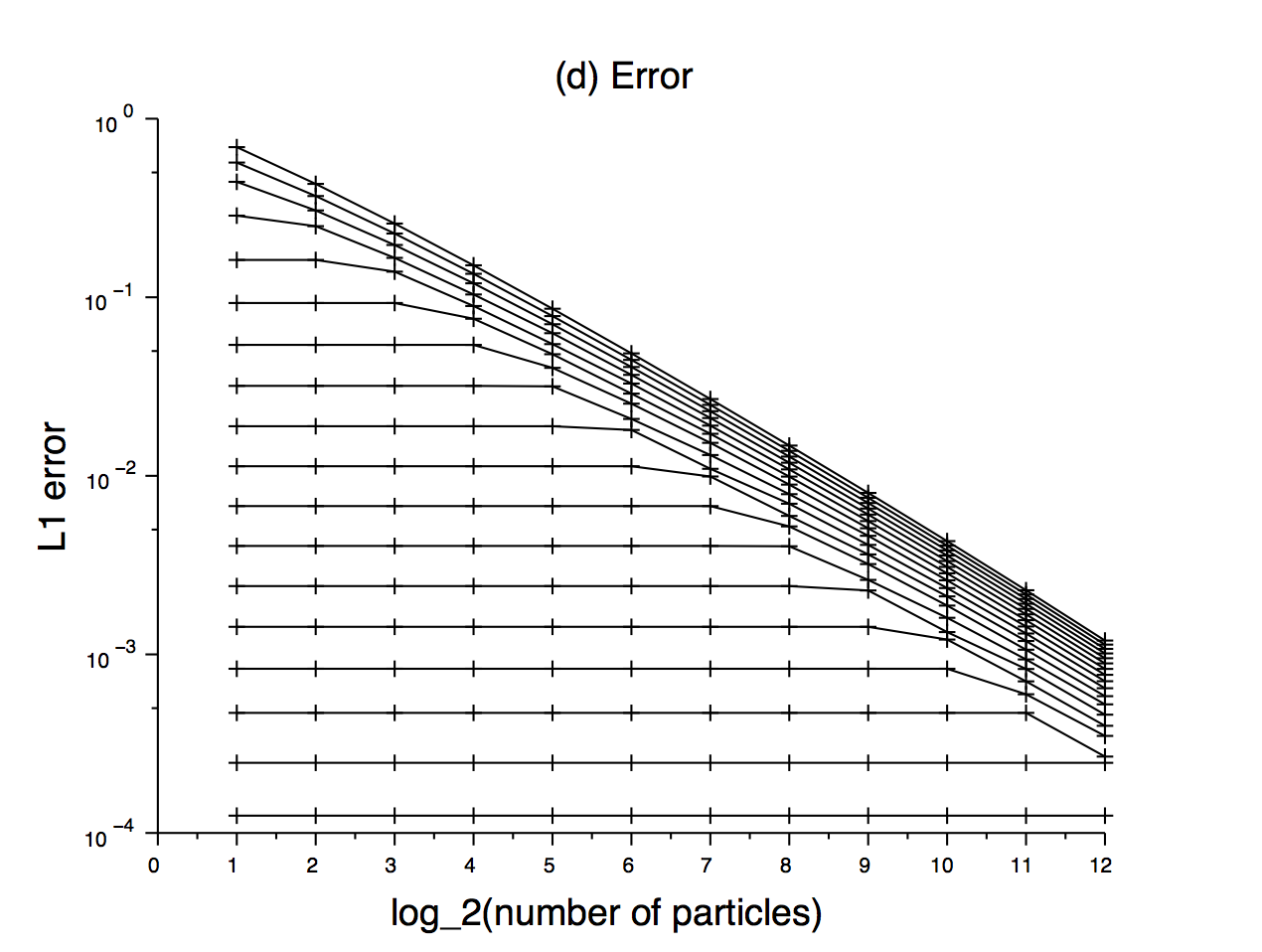}
  \caption{Numerical results for a concave flux function with two-sided exponential initial measure: (a) trajectories of 30 particles in the space-time plane, (b) value of the solution $u(t,x)$ in the space-time plane, (c) profile of the solution $u(t,x)$ at successive times $t=0, 0.5, 1, \ldots, 5$, (d) logarithmic plot of the $\Ls^1$ error between the approximation obtained with $2^p$ particles and the solution, as a function of $p = 1, \ldots, 12$. The different lines correspond to different values of $t$, namely $t=0, 1, \ldots, 17$, and the higher curves correspond to the smaller times.}
  \label{fig:concave}
\end{figure}


\subsection{Diagonal hyperbolic systems} We now turn to the numerical resolution of the diagonal hyperbolic system~\eqref{eq:syst} thanks to the MSPD. 

A first method to simulate the trajectory of the MSPD obviously consists in computing the exact space-time position of each collision (between particles of the same type, or between particles of different types). The number of such collisions is at most of order $n^2$: indeed, because of Assumption~\eqref{ass:USH}, there are at most $n^2d(d-1)/2$ collisions between particles of different types, and whenever particles of the same type collide, the space-time point of the next collision with a particle of another type is the same for all the particles in the current cluster. As a consequence, an algorithm computing the exact trajectory of each particle is expected to perform $\grandO(n^2)$ elementary operations. We however believe that such an algorithm with optimal complexity would require a rather technical implementation.

We therefore suggest to use a second method, which consists in approximating the MSPD by the iterated TSPD scheme described in Subsection~\ref{ss:approxMSPD}. Thanks to the Brenier-Grenier algorithm introduced in the scalar case, each iteration of the SPD is easily implemented and requires $\grandO(n)$ elementary operations. Then we shall show below that updating the velocities after each step also requires $\grandO(n)$ operations. As a consequence, computing the iterated TSPD on $L$ steps requires $\grandO(nL)$ elementary operations. On the other hand, the error between the solution to the system~\eqref{eq:syst} and the approximated solution provided by the iterated TSPD scheme was proved in Theorem~\ref{theo:rateMSPD} to be of order $\grandO(t/n+\Delta)$ at time $t$. For the terms $t/n$ and $\Delta$ to be of the same order, one therefore has to run the iterated TSPD scheme on $L \simeq t/\Delta \simeq n$ iterations, which leads to a total number of elementary operations in $\grandO(n^2)$. As a conclusion, this method has the same cost as the exact simulation of the MSPD, but it seems easier to implement.

\revision{It follows from this discussion that to reach an approximation error of order $\epsilon$ at time $t \geq 0$, the iterated TSPD scheme requires $\grandO(t^2/\epsilon^2)$ elementary operations. In comparison, standard upwind schemes for the hyperbolic system~\eqref{eq:syst}, with a time step $\Delta t$ and a mesh size $\Delta x$ satisfying the CFL condition $\ConstBoundS\Delta t \leq \Delta x$, are generally expected first-order accurate~\cite{Lev02}, so that the approximation error at time $t$ writes $C(t) \Delta x$, with a constant $C(t)$ that depends neither on $\Delta t$ nor on $\Delta x$, and grows at least linearly with $t$. Besides, at each iteration of such a scheme, $\grandO(1/\Delta x)$ elementary operations are necessary to compute the values of the fluxes and of the solution on the grid, so that after $L \simeq t/\Delta t$ iterations, $\grandO(t/(\Delta t \Delta x))$ elementary operations have been performed. As a consequence, the minimal number of elementary operations to reach a precision of order $\epsilon$ at time $t$ is obtained when the CFL condition is saturated, and it is in $\grandO(t C(t)^2/\epsilon^2)$, which has the same dependence on $\epsilon$ as the iterated TSPD scheme.}

\subsubsection{Description of the algorithm} In order to simulate the MSPD, we use the approximation of the latter by the TSPD on small time steps $\Delta$, as is described in Subsection~\ref{ss:approxMSPD}. Given $\x \in D_n^d$, we thus compute $\tPhi_{\Delta}^L(\x)$ instead of $\Phi(\x;L\Delta)$. To this aim, we use an elementary iterative algorithm which will therefore perform $L$ steps. At each iteration $\ell \in \{1, \ldots, L\}$, it is necessary to:
\begin{enumerate}[label=(\roman*), ref=\roman*]
  \item compute the vector of velocities for the TSPD started from $\tPhi^{\ell-1}_{\Delta}(\x)$,
  \item compute the evolution of each subsystem of particles according to the TSPD.
\end{enumerate}
Of course, the second step uses the algorithm described in Subsection~\ref{ss:num:spd} and therefore makes $\grandO(nd)$ elementary operations. The first step is realised by the following pseudo-code, the input of which is an array {\tt x} of size $d \times n$, such that {\tt x(gamma,k)} contains the initial position $x^{\gamma}_k$ of the particle $\gamma:k$. We recall that, for fixed $\gamma$, $x_k^{\gamma} \leq x_{k+1}^{\gamma}$ for all $k \in \{1, \ldots, n-1\}$.

\begin{verbatim}
  current_indices = vector [0, ..., 0] of size d 
  while min(current_indices) < n
    gamma = max( argmin( x(g,current_indices(g)) ; 
                         g such that current_indices(g)<n ) )
    k = current_indices(gamma)
    current_indices(gamma) = k+1
    velocity(gamma,k) = lambda(gamma,current_indices)
  end while
\end{verbatim}

In this pseudo-code, {\tt lambda(gamma,[k\_1, ..., k\_d])} returns the velocity 
\begin{equation*}
  n \int_{w=(k_{\gamma}-1)/n}^{k_{\gamma}/n} \lambda^{\gamma}\left(\frac{k_1}{n}, \ldots, \frac{k_{\gamma-1}}{n}, w, \frac{k_{\gamma+1}}{n}, \ldots, \frac{k_d}{n}\right)\dd w
\end{equation*}
of the particle $\gamma:k$, so that at the end of the algorithm, the vector {\tt velocity(gamma,:)} contains the initial velocities of the particles of type {\tt gamma} for the SPD.

There are $nd$ iterations of the {\tt while} loop, and to select {\tt gamma} it is necessary to scan the vector {\tt current\_indices}, which costs $d$ operations. As a consequence, the computation of the velocities requires $\grandO(nd^2)$ operations. Since $d$ is a physical parameter, we only consider the complexity with respect to the numerical parameter $n$ and therefore conclude that the computation of $\tPhi_{\Delta}^L(\x)$ is made in $\grandO(nL)$ operations.

\subsubsection{Case study: $p$-system} The $p$-system 
\begin{equation*}
  \left\{\begin{aligned}
    & \partial_t u = \partial_x v,\\
    & \partial_t v + \partial_x(p(u)) = 0,
  \end{aligned}\right.
\end{equation*}
is a simple model for isentropic gas dynamics in one space dimension, where $u$ is the specific volume of the gas and $v$ is its velocity. The function $p(u)$ determines the pressure in terms of the specific volume, and must generically satisfy $p'(u)<0$ for all $u \geq 0$. \revision{In the sequel, we shall furthermore assume that there exists $\nu>0$ such that
\begin{equation}\label{eq:hypp}
  \int_{u=0}^{\nu} \sqrt{-p'(u)}\dd u = 1,
\end{equation}
which is the appropriate condition to develop our probabilistic approach, and in which case the specific volume will take its values in $[0,\nu]$.}

\revision{Let us define the Riemann invariants $w^-$ and $w^+$ for this $2 \times 2$ system~\cite{serre} by
\begin{equation*}
  w^{\pm} = v \pm g(u),
\end{equation*}
where, for all $u \in [0,\nu]$,
\begin{equation*}
  g(u) := \int_{r=0}^u c(r)\dd r - \frac{1}{2}, \qquad c(u) := \sqrt{-p'(u)}.
\end{equation*}
The assumption~\eqref{eq:hypp} ensures that, if $w^-, w^+ \in [0,1]$, then $(w^+-w^-)/2$ belongs to the image of $g$, and it is immediately checked that $u$ and $v$ are recovered from the formulas
\begin{equation}\label{eq:uv}
  u = g^{-1}\left(\frac{w^+-w^-}{2}\right), \qquad v = \frac{w^++w^-}{2}.
\end{equation}
Furthermore, the Riemann invariants satisfy $\partial_t w^{\pm} = \pm c(u) \partial_x w^{\pm}$, which rewrites under the form of the $2 \times 2$ diagonal system
\begin{equation}\label{eq:systw}
  \left\{\begin{aligned}
    & \partial_t w^- + \lambda^-(w^-,w^+) \partial_x w^- = 0,\\
    & \partial_t w^+ + \lambda^+(w^-,w^+) \partial_x w^+ = 0,
  \end{aligned}\right.  
\end{equation}
with 
\begin{equation*}
  \lambda^{\pm}(w^-,w^+) := \mp c\left(g^{-1}\left(\frac{w^+-w^-}{2}\right)\right) = \mp \frac{1}{(g^{-1})'\left((w^+-w^-)/2\right)}.
\end{equation*}
Under the assumption that $p'$ be continuous and negative on $[0,\nu]$, one can define $\ell \in (0,+\infty)$ by
\begin{equation*}
  \ell := \inf_{0 \leq u \leq \nu} c(u),
\end{equation*}
and get 
\begin{equation}\label{eq:ellUSH}
  \forall w^-, w^+ \in [0,1], \qquad \lambda^+(w^-,w^+) \leq -\ell < \ell \leq \lambda^-(w^-,w^+),
\end{equation}
so that, for initial conditions $w^-_0$ and $w^+_0$ given by the CDFs of probability measures, the system satisfies Assumption~\eqref{ass:USH} with constant $\ConstUSH = 2\ell$.}

\revision{We now present numerical approximations of $u$ and $v$ for the choice of pressure function
\begin{equation*}
  p(u) = \frac{\kappa}{4\nu\argsinh^2(\kappa/2)}\left[\arctan\left(\frac{\kappa}{2}\right)-\arctan\left(\kappa\left(\frac{u}{\nu}-\frac{1}{2}\right)\right)\right], \qquad u \in [0,\nu],
\end{equation*}
where $\nu > 0$ is a given reference specific volume and $\kappa > 0$ is a dimensionless shape parameter. This choice implies
\begin{equation*}
  g(u) = \frac{\argsinh\left(\kappa\left(\frac{u}{\nu}-\frac{1}{2}\right)\right)}{2\argsinh\left(\frac{\kappa}{2}\right)},
\end{equation*}
so that
\begin{equation*}
  \lambda^{\pm}(w^-,w^+) = \mp\frac{\kappa}{2\nu\argsinh(\kappa/2)}\cdot\frac{1}{\cosh\left[(w^+-w^-)\argsinh\left(\frac{\kappa}{2}\right)\right]}.
\end{equation*}
The relation~\eqref{eq:uv} yields
\begin{equation*}
  u = \nu\left(\frac{1}{2} + \frac{1}{\kappa} \sinh\left[(w^+-w^-)\argsinh\left(\frac{\kappa}{2}\right)\right]\right), \qquad v = \frac{w^++w^-}{2}.
\end{equation*}}

We first address the case where the initial conditions $w_0^-$ and $w_0^+$ are the respective CDFs of the shifted two-sided exponential distributions $m^-$ and $m^+$ defined by
\begin{equation*}
  m^{\pm}(\dd x) := \frac{1}{2}\exp(-|x \pm x_0|),
\end{equation*}
for some $x_0 \geq 0$. In this case, $w^+_0(x) \geq w^-_0(x)$ for all $x \in \R$, which implies that $w^+(t,\cdot)-w^-(t,\cdot)$ remains nonnegative at all times $t \geq 0$. This is indeed easily checked at the level of the MSPD, a trajectory of which is plotted on Figure~\ref{fig:mspd:1}: if $w^-_0$ and $w^+_0$ are the empirical CDFs of two vectors $\rx^{\pm} = (x^{\pm}_1, \ldots, x^{\pm}_n) \in D_n$, then $w^+_0 \geq w^-_0$ if and only if, for all $k \in \{1, \ldots, n\}$, $x^+_k \leq x^-_k$. By~\eqref{eq:ellUSH}, for all $t \geq 0$ we have
\begin{equation*}
  \Phi^+_k(\x;t) \leq x^+_k - \ell t, \qquad \Phi^-_k(\x;t) \geq x^-_k + \ell t,
\end{equation*}
with the obvious notation $\x=(\rx^-,\rx^+) \in D_n^2$, so that the corresponding empirical CDF satisfies $w^+(t,\cdot) \geq w^-(t,\cdot)$. That this inequality still holds true in the limit $n \to +\infty$ can be checked using the notion of trajectories introduced in~\cite[Section~5]{jr}, see in particular~\cite[Corollary~5.1.2]{jr}.

One can observe on Figure~\ref{fig:mspd:1} that particles of the same type never collide with each other. This is due to the fact that, for fixed $w^- \in [0,1]$, the mapping $w^+ \mapsto \lambda^+(w^-,w^+)$ is increasing on $[w^-,1]$. As a consequence, two consecutive particles of type $+$ with no particle of type $-$ between them can only have velocities taking them away from each other. The same phenomenon occurs for particles of type $-$. However, collisions between particles of different types modify the velocities of these particles. Thus, the Riemann invariants $w^-$ and $w^+$, respectively plotted on Figures~\ref{fig:w:1} (a) and (b), undergo two interacting rarefaction waves, drifting away from each other on account of Assumption~\eqref{ass:USH}, without forming any shock. The specific volume $u$ and the velocity $v$ are finally plotted on Figures~\ref{fig:w:1} (c) and (d). 

\begin{figure}[ht]
  \includegraphics[clip=true,trim=.5cm 0cm 1cm 1cm,width=12cm]{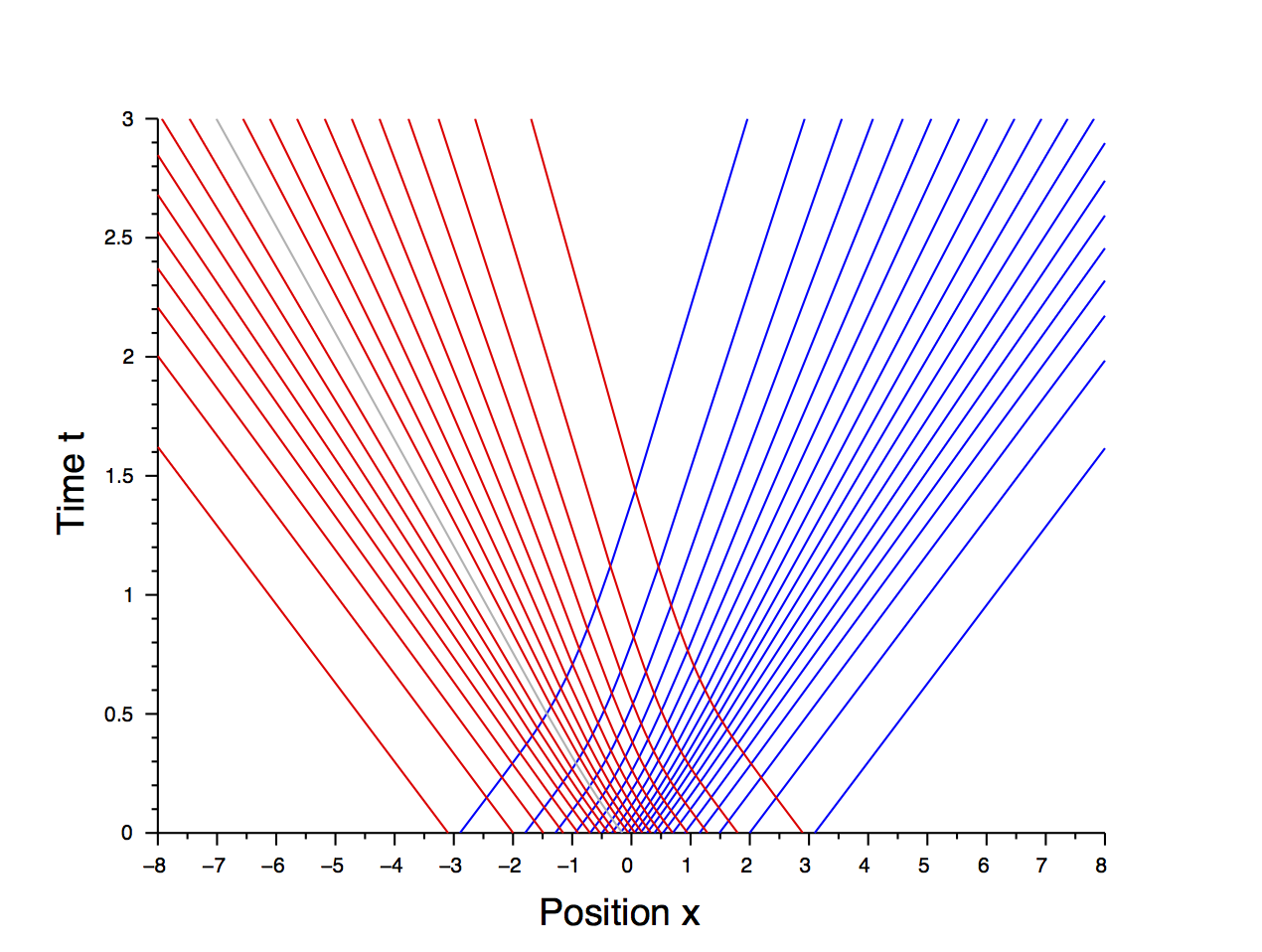}
  \caption{Trajectory of the MSPD (obtained with the iterated TSPD scheme with $\Delta=0.03$) with $20$ particles per type associated with the $p$-system for $w^+_0 \geq w^-_0$. Blue rays correspond to particles of type $-$, red rays correspond to particles of type $+$. Here $x_0=0.1$ and $\nu=0.5$, $\kappa=5$.}
  \label{fig:mspd:1}
\end{figure}

\begin{figure}[ht]
  \includegraphics[clip=true,trim=.5cm 0cm 1cm 1cm,width=7.5cm]{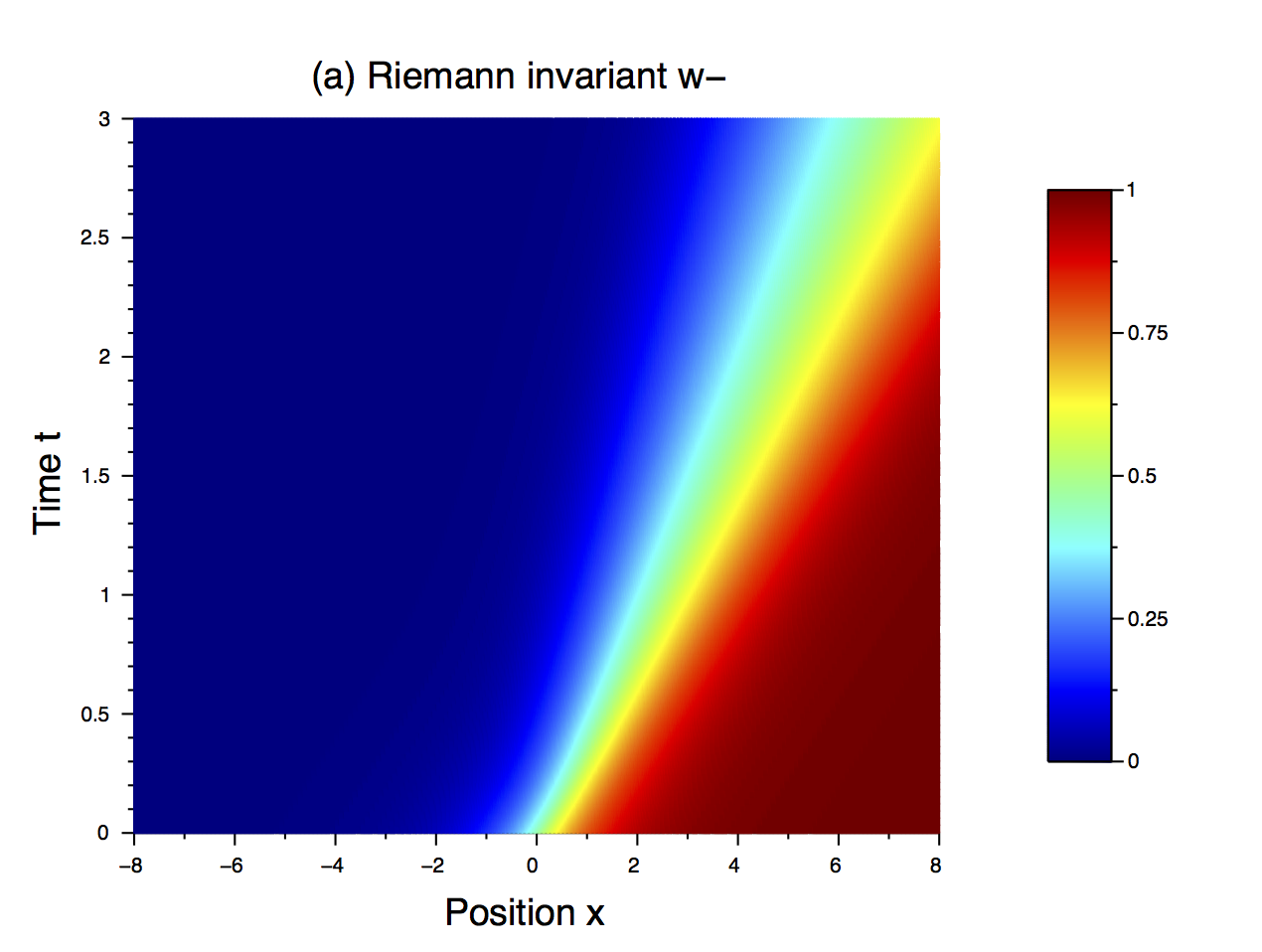}\includegraphics[clip=true,trim=.5cm 0cm 1cm 1cm,width=7.5cm]{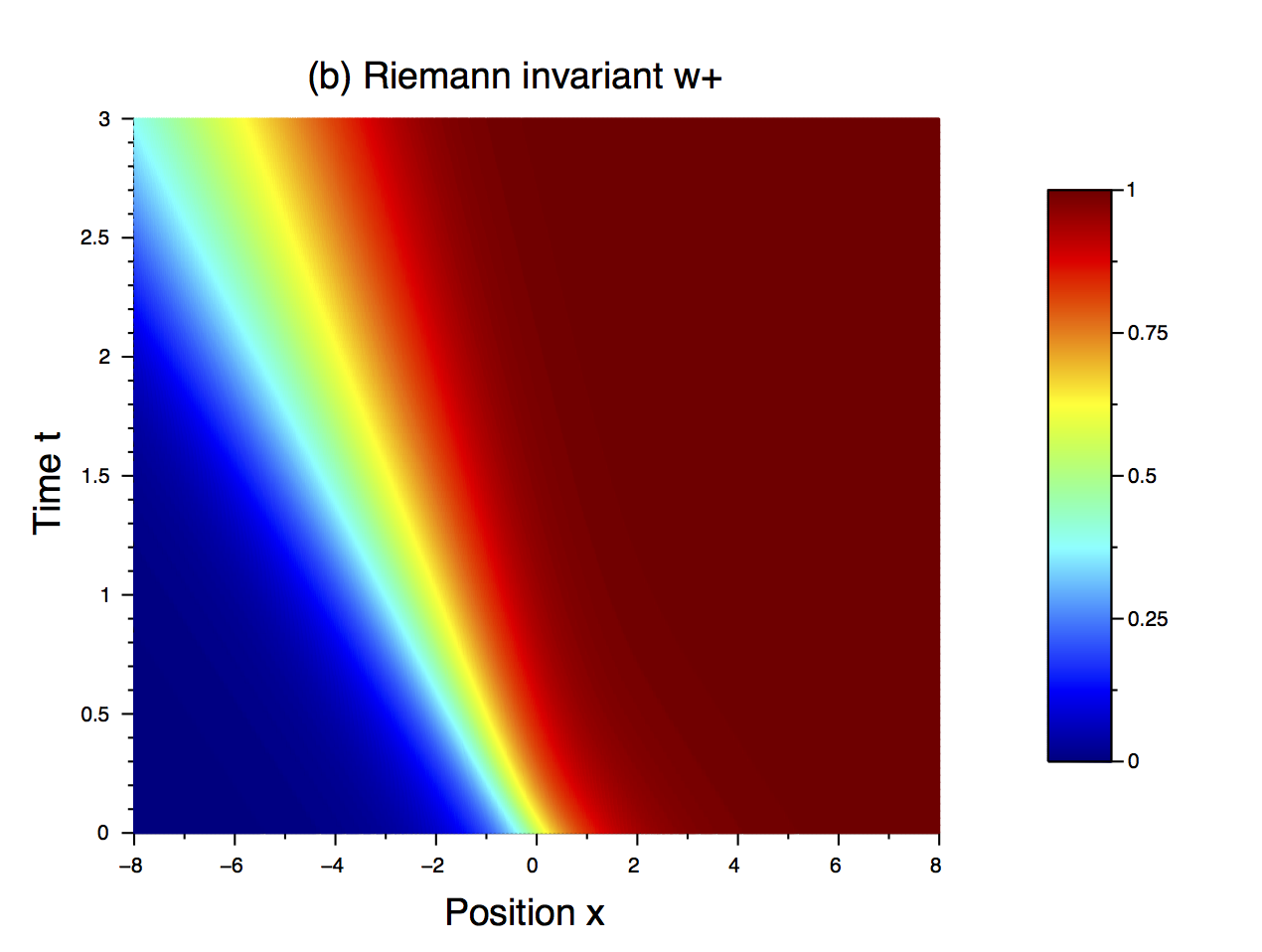}
  
  \includegraphics[clip=true,trim=.5cm 0cm 1cm 1cm,width=7.5cm]{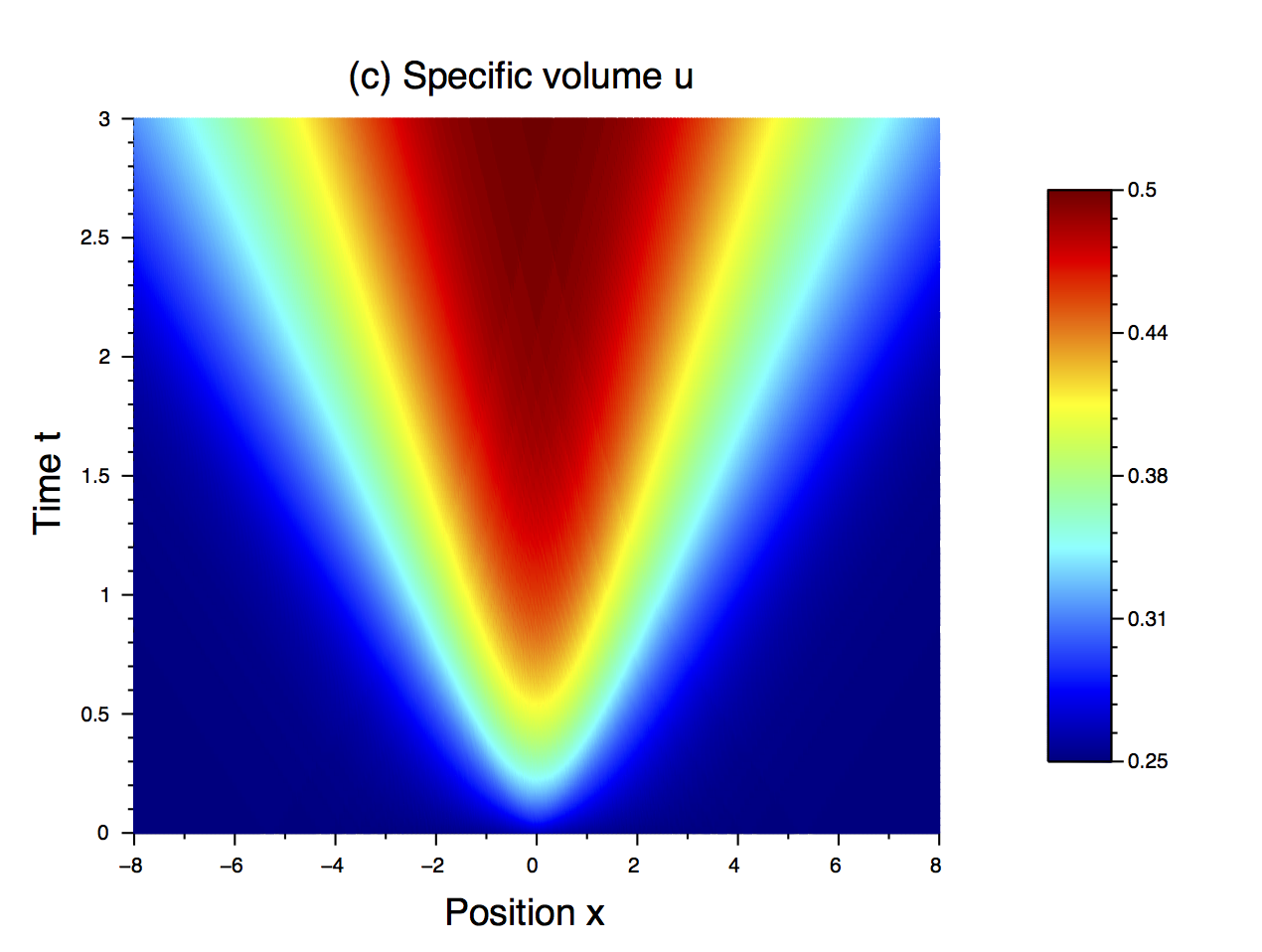}\includegraphics[clip=true,trim=.5cm 0cm 1cm 1cm,width=7.5cm]{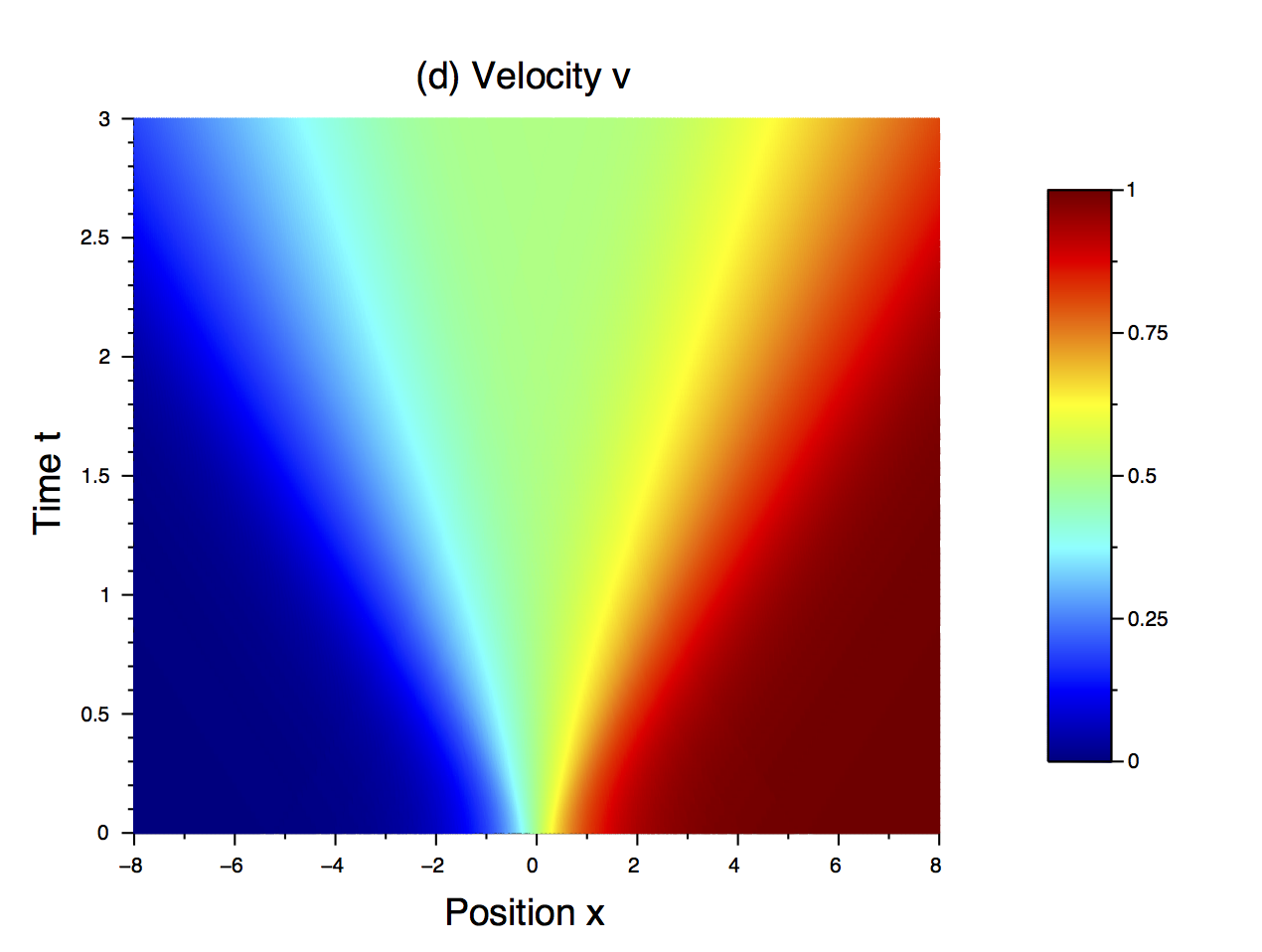}
  \caption{Representation in the space-time plane of the Riemann invariants $w^-$ (a) and $w^+$ (b), of the specific volume $u$ (c) and of the velocity $v$ (d) for the $p$-system with $w^+_0 \geq w^-_0$. The simulation uses the iterated TSPD algorithm with parameters $n=200$ and $\Delta=0.03$, and $\nu=0.5$, $\kappa=5$.}
  \label{fig:w:1}
\end{figure}

As a sticky particle dynamics where particles never stick together may seem a little disappointing, we now choose initial conditions that do not satisfy the condition that $w^+_0(x) \geq w^-_0(x)$ for all $x \in \R$. To this aim, we still assume $w^-_0$ to be given by the CDF of $m^-(\dd x) = \exp(-|x-x_0|)/2$ with $x_0 \leq 0$, and take $w^+_0(x) = \ind{x \geq 0}$. Particles of both type can now aggregate into clusters, as is depicted on Figure~\ref{fig:mspd:2}. But on account of Assumption~\eqref{ass:USH}, after a finite time (that generally depends on the number of particles), the property that $w^+ \geq w^-$ is recovered and the particles start drifting away from each other again. This is also observed on Figure~\ref{fig:mspd:2}, where two clusters blow up under the effect of a collision. As a result, the functions $w^-$, $w^+$, $u$ and $v$ exhibit shocks on short times, and are essentially given by interacting rarefaction waves on longer times, see Figure~\ref{fig:w:2}.

\begin{figure}[ht]
  \includegraphics[clip=true,trim=.5cm 0cm 1cm 1cm,width=12cm]{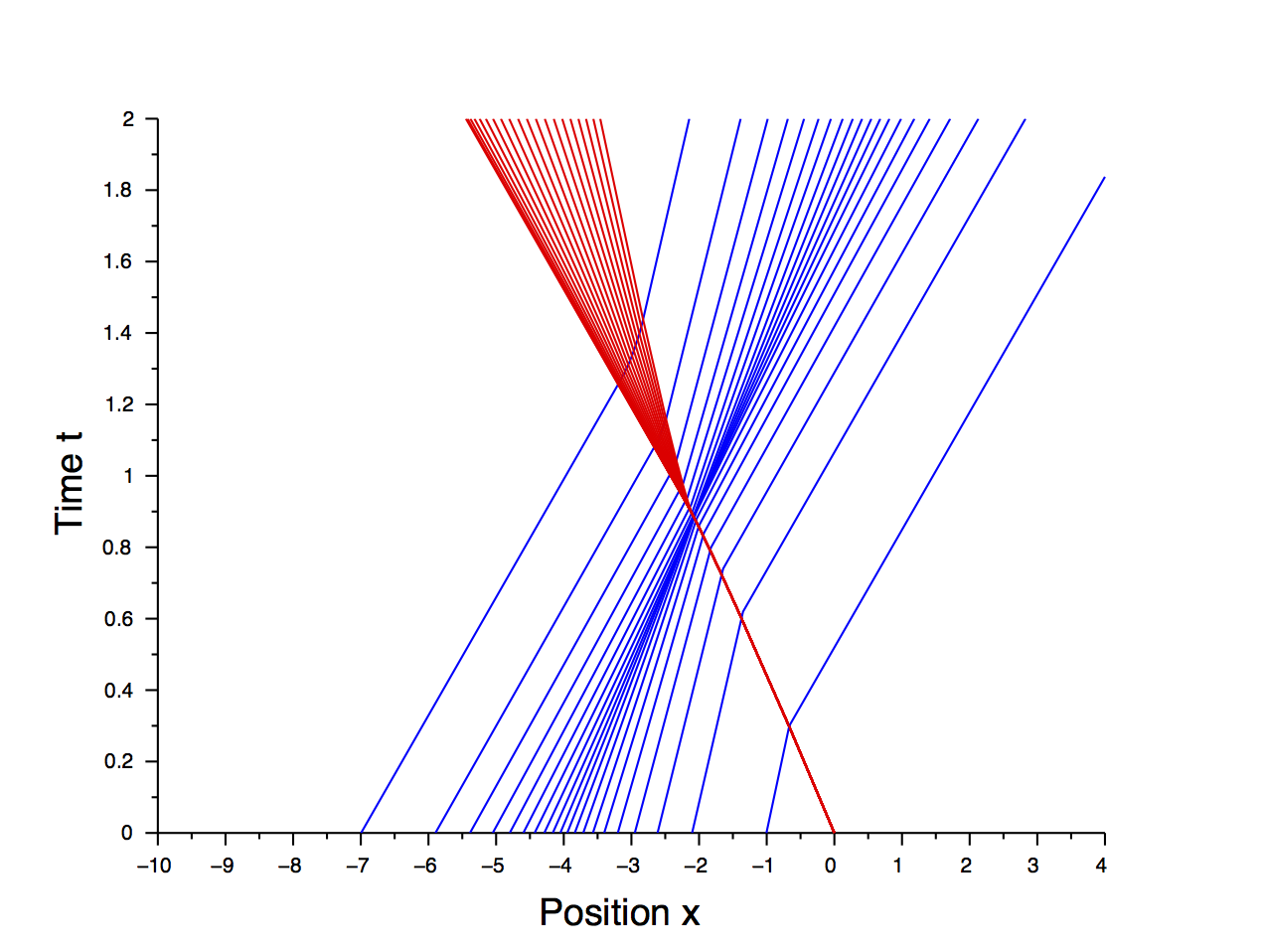}
  \caption{Trajectory of the MSPD (obtained with the iterated TSPD scheme with $\Delta=0.02$) with $20$ particles per type associated with the $p$-system for initial conditions with shocks. Blue rays correspond to particles of type $-$, red rays correspond to particles of type $+$. Red particles first remain aggregated into a single cluster up to the collision with the median blue particle, which makes it blow up then. Here $x_0=-4$ and $\nu=0.5$, $\kappa=5$.}
  \label{fig:mspd:2}
\end{figure}

\begin{figure}[ht]
  \includegraphics[clip=true,trim=.5cm 0cm 1cm 1cm,width=7.5cm]{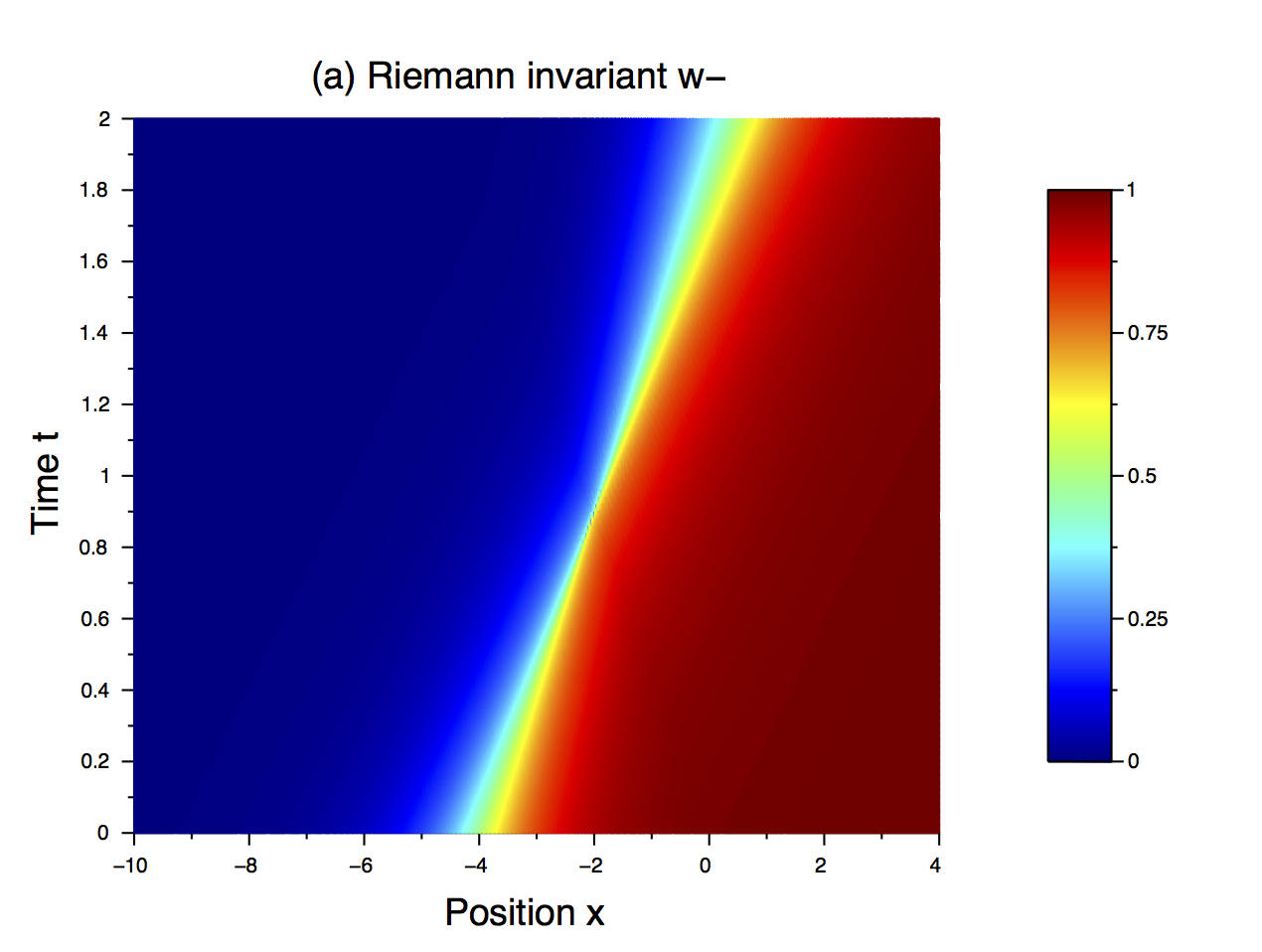}\includegraphics[clip=true,trim=.5cm 0cm 1cm 1cm,width=7.5cm]{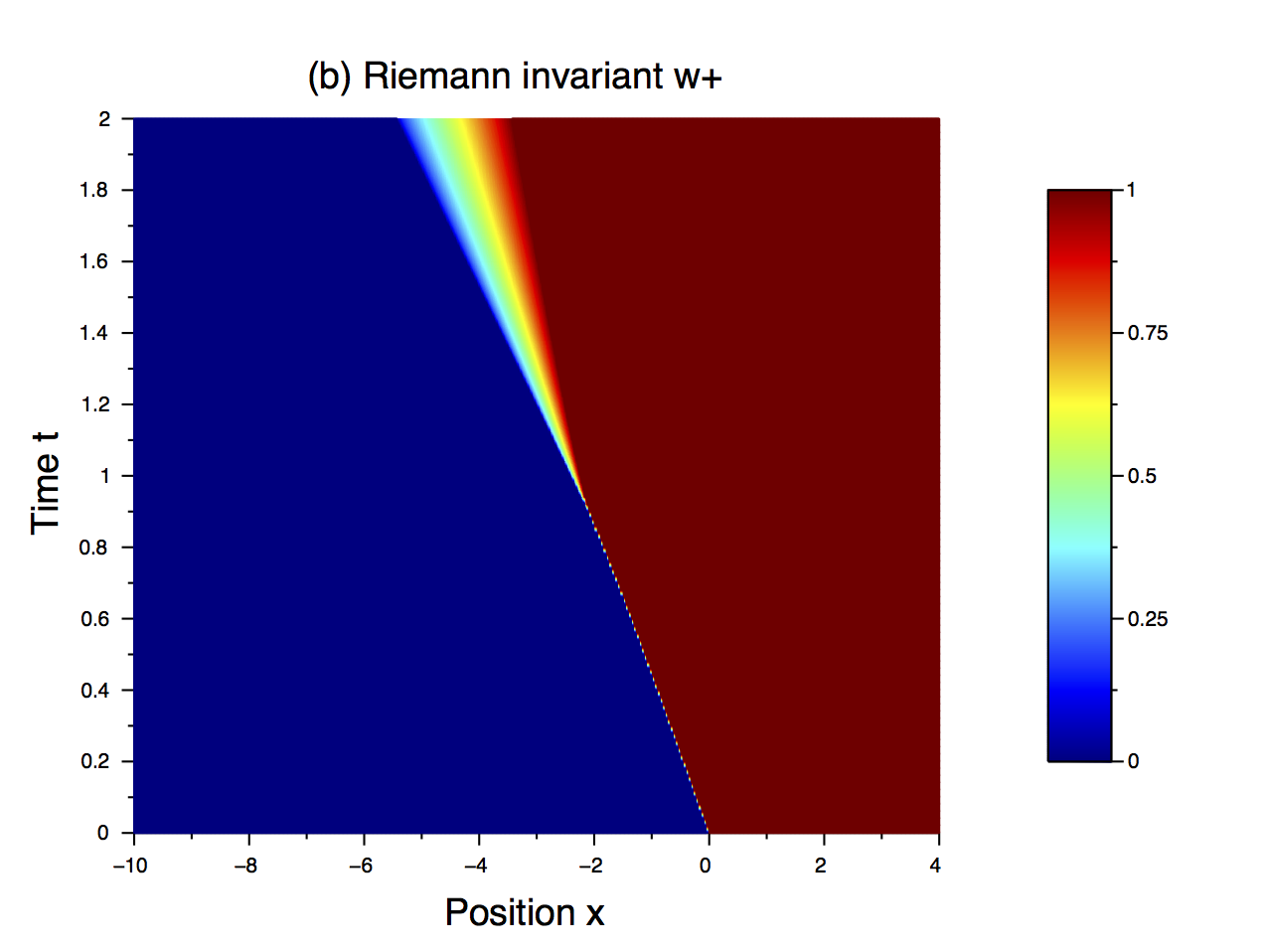}
  
  \includegraphics[clip=true,trim=.5cm 0cm 1cm 1cm,width=7.5cm]{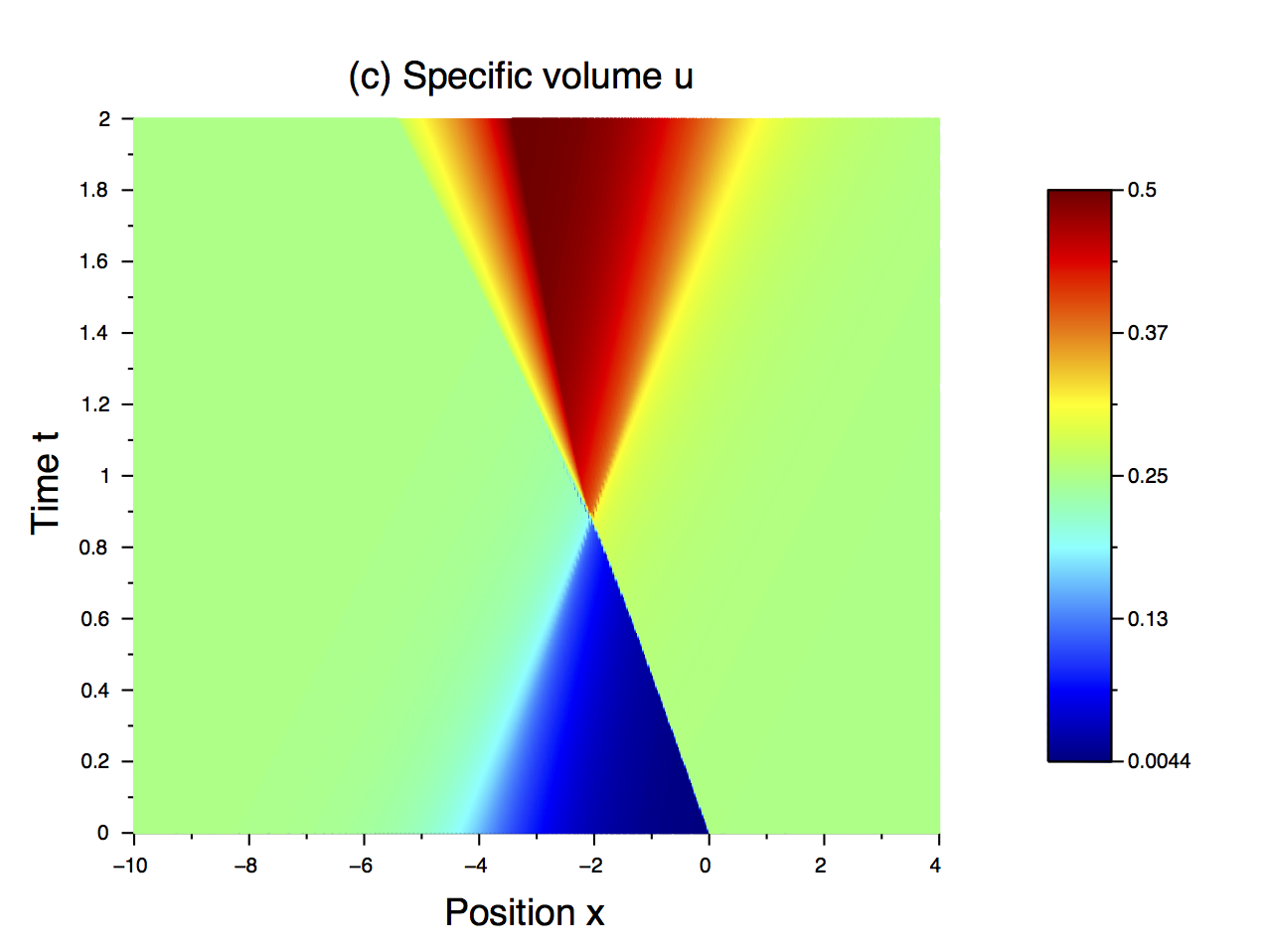}\includegraphics[clip=true,trim=.5cm 0cm 1cm 1cm,width=7.5cm]{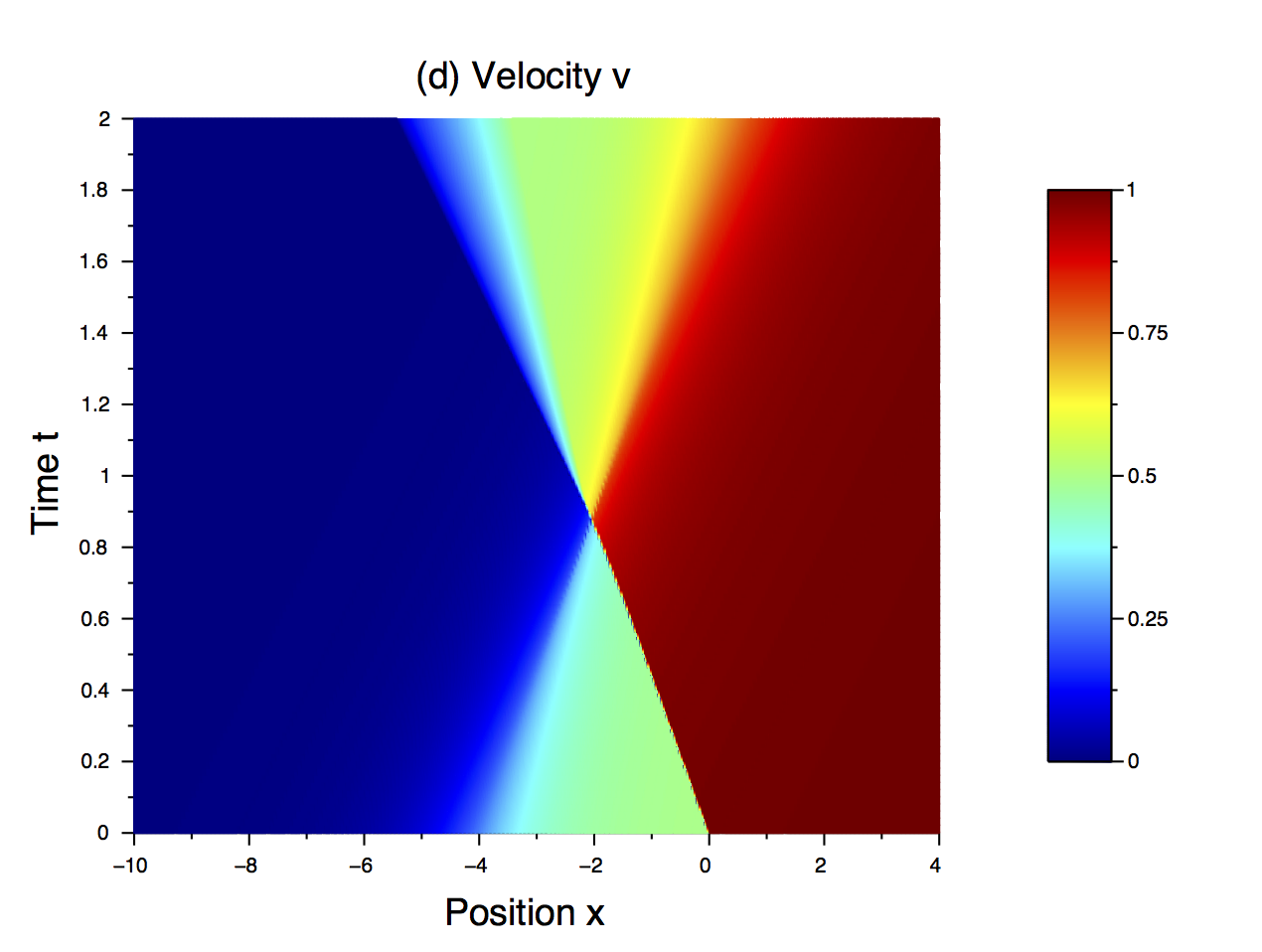}
  \caption{Representation in the space-time plane of the Riemann invariants $w^-$ (a) and $w^+$ (b), of the specific volume $u$ (c) and of the velocity $v$ (d) for initial conditions with shocks. The simulation uses the iterated TSPD algorithm with parameters $n=200$ and $\Delta=0.02$, and $\nu=0.5$, $\kappa=5$.}
  \label{fig:w:2}
\end{figure}


\appendix

\section{Proof of Proposition~\ref{prop:genstaconscal}}\label{app:pf}

Since $(t,x)\mapsto u(t+s,x)$ (resp. $(t,x)\mapsto v(t+s,x)$) is the entropy solution to \eqref{eq:scl} with initial condition $x\mapsto u(s,x)$ (resp. $x\mapsto v(s,x)$), it is enough to deal with the case $s=0$. 

We define $m$ and $m'$ in $\Ps(\R)$ by $u_0 = H*m$, $v_0 = H*m'$, and use the discretisation of $m$ and $m'$ corresponding to~\eqref{eq:chin}, namely
\begin{equation*}
  x_i(n)=(n+1)\int_{w=(2i-1)/(2(n+1))}^{(2i+1)/(2(n+1))} u_0^{-1}(w)\dd w, \qquad y_i(n)=(n+1)\int_{w=(2i-1)/(2(n+1))}^{(2i+1)/(2(n+1))} v_0^{-1}(w)\dd w,
\end{equation*}
for all $i \in \{1, \ldots, n\}$. Let $\rx(n)=(x_1(n),\ldots,x_n(n))$ and $\ry(n)=(y_1(n),\ldots,y_n(n))$. According to~\cite[Lemma~8.1.5]{jr}, $\mu_0[\rx(n)]$ (resp. $\mu_0[\ry(n)]$) converges weakly to $m$ (resp. $m'$) as $n\to+\infty$. Moreover, by~\cite[Lemma~8.1.6]{jr}, 
\begin{equation*}
  \lim_{n\to+\infty}\Ws_1(\mu_0[\rx(n)],\mu_0[\ry(n)])=\Ws_1(m,m').
\end{equation*}
By~\eqref{eq:stabSPD}, for 
\begin{equation*}
  \rblambda = n\left(\int_{w=0}^{1/n}\lambda(w)\dd w, \ldots, \int_{w=1-1/n}^{1}\lambda(w)\dd w\right)
\end{equation*}
and
\begin{equation*}
  \rbmu = n\left(\int_{w=0}^{1/n}\mu(w)\dd w, \ldots, \int_{w=1-1/n}^{1}\mu(w)\dd w\right),
\end{equation*}
one has
\begin{equation*}
  \begin{aligned}
    \|\phi[\rblambda](\rx(n);t) - \phi[\rbmu](\ry(n);t)\|_1 & \leq \Ws_1(\mu_0[\rx(n)],\mu_0[\ry(n)]) + t \sum_{i=1}^n \left|\int_{w=(i-1)/n}^{i/n}(\lambda(w)-\mu(w))\dd w\right|\\
    & \leq \Ws_1(\mu_0[\rx(n)],\mu_0[\ry(n)]) + t \int_{w=0}^1|\lambda(w)-\mu(w)|\dd w.
  \end{aligned}
\end{equation*}
One concludes by taking the limit $n\to+\infty$ in this inequality since Theorem~\ref{theo:cvSPD} and the lower semi-continuity of $\Ws_1$ with respect to the weak convergence topology~\cite[Remark~6.12]{villani} ensure that 
\begin{equation*}
  \|u(t,\cdot) - v(t,\cdot)\|_{\Ls^1(\R)} \leq \liminf_{n\to+\infty} \|\phi[\rblambda](\rx(n);t) - \phi[\rbmu](\ry(n);t)\|_1.
\end{equation*}


\subsection*{Acknowledgements} We thank our colleague R\'egis Monneau (CERMICS) for numerous fruitful discussions which motivated this work.


\end{document}